\documentclass[11pt]{scrartcl}
\usepackage[letterpaper,top=2cm,bottom=2cm,left=3cm,right=3cm,marginparwidth=1.75cm]{geometry}
\usepackage{csquotes}

\usepackage{amsmath, amsthm}
\usepackage{amssymb}
\usepackage{graphicx}
\usepackage{hyperref}
\usepackage[linesnumbered, ruled]{algorithm2e}

\usepackage{tikz}
\usepackage{mathrsfs}
\usepackage{standalone}
\usepackage{authblk}

\graphicspath{{figures/}{../figures/}}

\newtheorem{theorem}{Theorem}
\newtheorem{lemma}[theorem]{Lemma}
\newtheorem{remark}{Remark}

\newcommand{\diss}{u}

\newcommand{\bc}{g}

\date{\today}
\title{Learning Mesh Motion Techniques with Application to Fluid-Structure Interaction\footnote{Ottar Hellan and Miroslav Kuchta acknowledge support from the Research Council of Norway, grant 303362. Johannes Haubner and Marius Zeinhofer acknowledge support from the Research Council of Norway, grant 300305.}}

\author[1]{Johannes Haubner\thanks{johannes.haubner@uni-graz.at}}
\author[2]{Ottar Hellan}
\author[2]{Marius Zeinhofer}
\author[2]{Miroslav Kuchta}
\affil[1]{University of Graz, Universitätsplatz 3, 8010 Graz, Austria}
\affil[2]{Simula Research Laboratory, Kristian Augusts gate 23, 0164 Oslo, Norway}

\begin{document}

\maketitle

\begin{abstract}
     Mesh degeneration is a bottleneck for fluid-structure interaction (FSI) simulations and for shape optimization via the method of mappings. In both cases, an appropriate mesh motion technique is required. The choice is typically based on heuristics, e.g., the solution operators of partial differential equations (PDE), such as the Laplace or biharmonic equation. Especially the latter, which shows good numerical performance for large displacements, is expensive. Moreover, from a continuous perspective, choosing the mesh motion technique is to a certain extent arbitrary and has no influence on the physically relevant quantities. Therefore, we consider approaches inspired by machine learning. We present a hybrid PDE-NN approach, where the neural network (NN) serves as parameterization of a coefficient in a second order nonlinear PDE. We ensure existence of solutions for the nonlinear PDE by the choice of the neural network architecture. Moreover, we present an approach where a neural network corrects the harmonic extension such that the boundary displacement is not changed. In order to avoid technical difficulties in coupling finite element and machine learning software, we work with a splitting of the monolithic FSI system into three smaller subsystems. This allows to solve the mesh motion equation in a separate step. We assess the quality of the learned mesh motion technique by applying it to a FSI benchmark problem. In addition, we discuss generalizability and computational cost of the learned mesh motion operators.
\end{abstract}
 
\section{Introduction}

Following the past few decades' success of deep learning in fields such as image processing and speech recognition \cite{lecun_deep_learning}, researchers have turned their use to solving differential equations of scientific computing. 
In \cite{lagaris1998artificial, raissi2019physics}, it is proposed to solve forward and inverse PDEs by parameterizing the solution as a neural network and solving the non-convex optimization problem of minimizing the squared residual of the PDE's strong form, an approach called Physics-Informed Neural networks that has many variants \cite{Cuomo2022_sciml_pinns}. Moreover models known as neural operators have been developed, that are trained to solve an entire class of PDEs at the same time, for instance the DeepONet \cite{Lu2021deeponet} and the Fourier Neural Operator \cite{li2020fourier_neural_operator}, in a paradigm known as operator learning.
Another line of work is to use neural networks to accelerate parts of numerical algorithms that are driven by heuristics or parameter tuning. If a parameter in an algorithm affects its performance and must be tuned by hand, it can instead be chosen during runtime by a neural network with access to some performance indicator. For instance, in \cite{Antonietti2023_amg-ann, tassi2021machine} the authors use neural networks to respectively pick a stabilization parameter of a finite element scheme and to automate the selection of a parameter affecting how coarse problems are constructed in algebraic multigrid methods. In \cite{BECK2020shock_detection}, the authors use a convolutional neural network to accurately detect and localize shocks, allowing more efficient shock capturing and avoiding parameter tuning. Automating the selection of parameters in algorithms can improve both the efficiency of the numerical algorithm itself and the scientific workflow, by allowing the user to spend time on other tasks. If a step in an algorithm is driven by heuristics, there is no obvious way of making the optimal choice and learning how to do this from data might be feasible.     
We take inspiration from operator learning and attempt to solve with machine learning techniques the problem of boundary deformation extension, a computational subproblem that is typically defined by the solution of heuristically chosen PDEs. We do this using two distinct approaches. First, we present an approach where a neural network parameterizes a PDE, the solution of which defines an extension operator. Secondly, we present an approach where a neural network defines a correction to a standard extension operator.

Extending boundary deformation onto the whole domain is a task that can be a crucial and limiting factor in applications, especially in settings, where one wants to avoid remeshing techniques to reduce computational effort, see e.g. \cite{stein2004automatic, johnson1994mesh}. It appears in monolithic arbitrary Lagrangian-Eulerian (ALE) formulations of fluid-structure interaction (FSI) problems, e.g., \cite{Wick, Shamanskiy2020}, and the method of mappings for partial differential equation (PDE) constrained shape optimization problems, e.g., \cite{MuellerKuehnetal, onyshkevych2021mesh, pinzon2023fluid}.
The task can be stated as follows. Let $\Omega \subset \mathbb R^d$, $d \in \lbrace 2, 3 \rbrace$, be a domain and $\partial \Omega$ be its boundary. Given a boundary displacement $\bc: \partial \Omega \to \mathbb R^d$, one needs to apply an extension operator to find a deformation field $\diss: \Omega \to \mathbb R^d$ such that $\diss \vert_{\partial \Omega} = \bc$ and $(\mathrm{id} + \diss)(\Omega)$ is a well-defined (Lipschitz) domain. 
This imposes requirements on $\bc$ and the extension operator. 
Among other things, $\mathrm{id} + \diss: \Omega \to (\mathrm{id} + \diss) (\Omega)$ needs to be bijective, which also implies that $\bc$ has to be chosen appropriately (e.g. not leading to a self-intersecting boundary of $\Omega$).  For computational purposes, it is mandatory that for any fixed mesh that discretizes $\Omega$ the transformed mesh is non-degenerate, see Figure \ref{Figure_mesh_degen}. 
In ALE formulations of fluid-structure interaction,  the task of finding a suitable extension operator is typically based on heuristic choices, and has no (from a continuous perspective) or little (from a numerical perspective) influence on the physical solution as long as it is ensured that it fulfills the above described requirements (and regularity requirements), see e.g. \cite{donea2004arbitrary}. At the same time, it is the limiting factor for numerical simulations as soon as these requirements are violated. 
In this work, we apply ideas from machine learning to find a suitable extension operator for fluid-structure interaction problems formulated in the ALE framework.

\begin{figure}
    \includegraphics[width=0.3\textwidth]{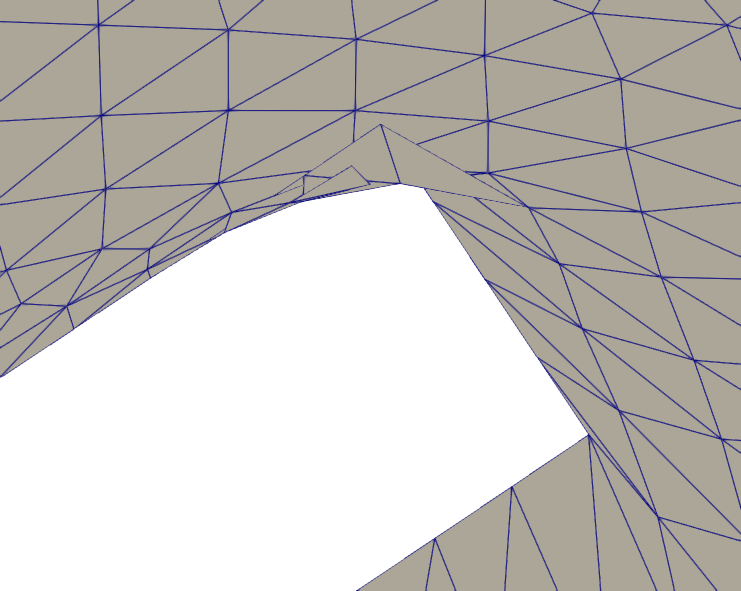}
    \includegraphics[width=0.3\textwidth]{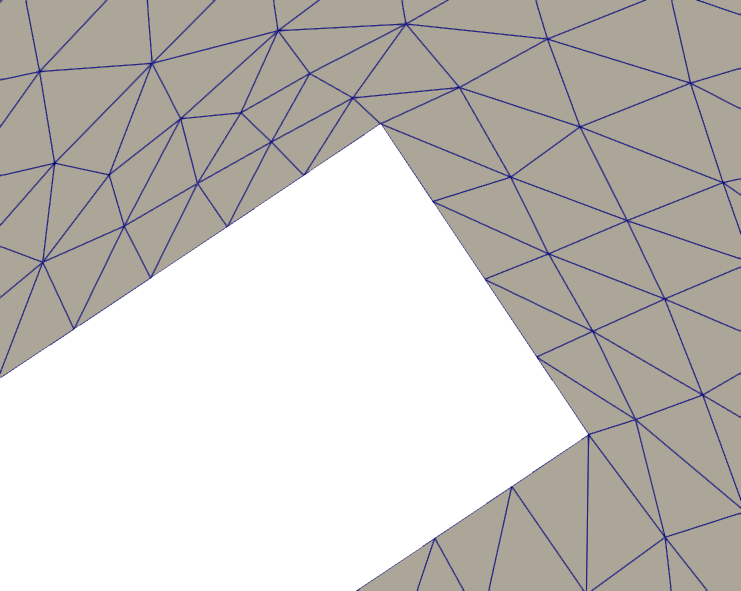}
    \centering
    \includegraphics[width=0.3\textwidth]{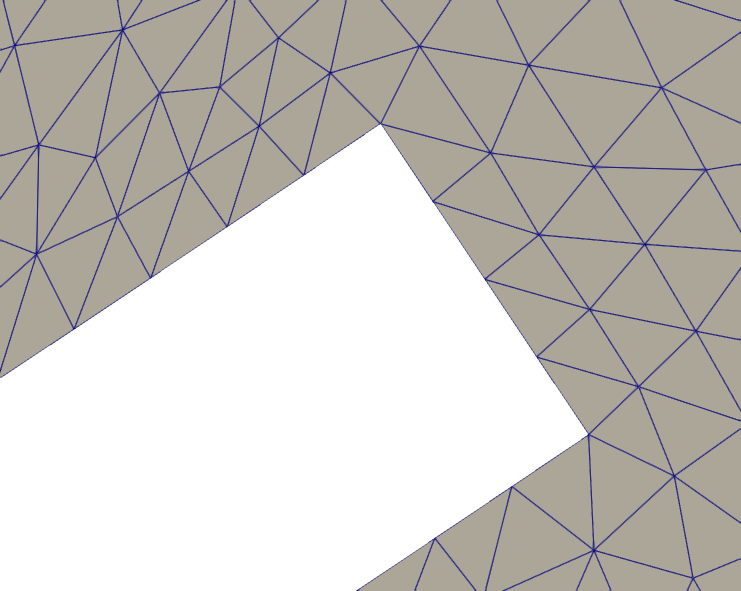}
    \vspace{-10pt}
    \caption{Mesh degeneration at the top corner of the elastic structure for harmonic extension (left), but not for the learned extension operator (middle; here shown for hybrid PDE-NN approach with artificial training data) and biharmonic extension (right).}
    \label{Figure_mesh_degen}
\end{figure}

Several approaches to define the extension operator have been introduced in the literature. In \cite{Wick}, harmonic extension operators, a linear elastic and a biharmonic model have been compared in the context of monolithic ALE formulations of FSI problems. The harmonic extension operator is obtained by solving the PDE
\begin{align}\label{eq:harmonic_extension}
    - \Delta \diss = 0 ~~ \text{in } \Omega, \quad \diss = \bc ~~\text{on } \partial \Omega.
\end{align}
The linear elastic model, e.g. also used in \cite{stein2004automatic}, is given by 
\begin{align*}
    - \mathrm{div}(\eta_\lambda \mathrm{tr}(\epsilon) I + 2 \eta_\mu \epsilon) = 0 ~~ \text{in } \Omega, \quad \diss = \bc ~~\text{on } \partial \Omega, 
\end{align*}
with $\epsilon = \frac12 (\nabla \diss + \nabla \diss^\top)$ and scalar valued functions $\eta_\lambda, \eta_\mu$ that are chosen, e.g., mesh dependent. 
In the context of shape optimization, \cite{SchulzSiebenborn} chooses $\eta_\lambda = 0$ and $\eta_\mu = \mu$, where $\mu$ solves the PDE
\begin{align*} 
    - \Delta \mu = 0 ~~\text{in } \Omega, \quad \mu = \mu_{\max} ~~ \text{on } \Gamma, \quad \mu = \mu_{\min} ~~\text{on } \partial \Omega \setminus \Gamma, 
\end{align*}
for constants $\mu_{\max,}~ \mu_{\min} > 0$. The manually chosen boundary value allows to increase rigidity in areas where large deformations occur, e.g., around $\partial \Omega$.  Finally, for large displacements the biharmonic model given by
\begin{align}\label{eq:biharmonic_equation}
    \Delta^2 \diss = 0 ~~\text{in } \Omega,  \quad \diss = \bc, ~~ \nabla \diss \cdot n = 0 ~~\text{on } \partial \Omega,
\end{align}
shows better numerical behavior compared to the harmonic or linear elastic extension, see e.g. \cite{Wick}. However, it is also the most expensive of these options since it requires either $H^2$-conforming finite elements, or other techniques such as a mixed approach, cf. e.g. \cite[Sec. 4.4.5]{Wick}, or weakly imposed continuity of the normal derivatives \cite{georgoulis2009discontinuous}. 
Recently, in the context of shape optimization, also different types of extension equations were used. In contrast to fluid-structure interaction simulations, the choice of nonlinear extension operators is justified in many shape optimization applications since solving governing time-dependent PDEs for the state equation is the computational bottleneck of the shape optimization simulation. Differently from the fluid-structure interaction application, the choice of the extension operator can be part of the modeling and, e.g., used to define a shape metric based on the corresponding Steklov-Poincar\'e operator \cite{Schulz2016}. In \cite{MuellerKuehnetal}, the $p$-Laplacian for $p \geq 2$ was introduced as an extension equation. For $1 < p < \infty$, the problem is given by 
\begin{align*}
    - \mathrm{div} ( \| \nabla \diss \|^{p-2} \nabla \diss ) = 0 ~~ \text{in } \Omega, \quad \diss = \bc ~~\text{on } \partial \Omega, 
\end{align*} 
where $\| \cdot \|$ denotes the Frobenius norm.
The nonlinearity increases the rigidity for increasing deformation gradient. 

In the context of mesh extension operators, machine learning algorithms 
have been applied in several works. For example, 
in \cite{stadler2011-mm_ann}, it is proposed to use neural networks as an ansatz class for mesh motion, by training a network of only coordinate inputs to match the deformation of the boundary vertices. Working in the PINNs methodology, \cite{AYGUN2023106660} learns a deformation map by linear elastic mesh motion. 
In \cite{NEURIPS2022_mesh_movement}, neural networks are trained to produce mesh motion for $r$-adaptivity, taking inputs describing the PDE instance.

In this work we develop two approaches to learn an extension operator. First, we consider a hybrid PDE-NN model \cite{HybridNNPDE}.
We observe that the harmonic extension operator and the $p$-Laplacian are special cases of the more general formulation 
\begin{align}
    - \mathrm{div}(\bar \alpha(\theta, \xi, \diss, \nabla \diss) \nabla \diss) = 0 ~~ \text{in } \Omega, \quad \diss = \bc ~~\text{on }  \partial \Omega, 
    \label{PDEalpha}
\end{align}
where $\bar \alpha$ is a scalar valued function, $\xi$ denotes the coordinates and $\theta$ parameters (e.g., $\theta = p$, $\theta = \mu$). Here we shall parameterize $\bar \alpha$ by using a neural network and in turn $\theta$ represents the network's 
weights and biases. The choice of $\bar{\alpha}$ is made based on theoretical considerations such that existence of solutions of \eqref{PDEalpha} is ensured. This requires a sophisticated choice of the architecture and the weights. In the scope of the work, for the sake of simplicity, we restrict ourselves to $\bar \alpha(\theta, \xi, \diss, \nabla \diss) = \alpha (\theta, \| \nabla \diss \|^2)$. We make a theoretically funded choice for the neural net, which serves as parameterization for $\bar \alpha$, such that existence of solutions of \eqref{PDEalpha} is ensured. The choice of $\bar \alpha$ then allows us to prove well-posedness of the extension operator by applying ideas from the proof that the $p$-Laplace equations have unique solutions. 
Moreover, we follow a supervised learning approach that aims to find 
a neural network such that the solution of \eqref{PDEalpha} is close to the solution of the biharmonic extension. The resulting extension equation requires a nonlinear PDE solve. In order to make it suitable in the setting of the unsteady FSI equations, we follow a ``lagging nonlinearity'' approach, i.e. we consider the linear extension equation $- \mathrm{div}(\bar \alpha(\theta, \xi, \bar \diss, \nabla \bar \diss) \nabla \diss) = 0$ with $\bar \diss$ being the deformation variable of the previous time-step in order to compute the new deformation field.

Our second approach is to learn an extension operator where a neural network corrects the harmonic extension operator. The choice of network architecture acts locally and ensures that the resulting extension operator respects the boundary condition constraints, see \eqref{eq:corrected_harm_extension}. Hence, it differs from other approaches where such constraints are penalized in the loss function.
This is done by including in the network architecture multiplication of the output by a function vanishing on the boundary, which is found by solving a Poisson problem on the computational domain $\Omega$. Also here we follow a supervised learning approach, training the corrected harmonic extension operator to match the biharmonic extension operator.

In the hybrid PDE-NN approach we generalize the harmonic extension to a class of extensions and then select one of them by training a neural network. In the corrected harmonic extension approach, we instead use the harmonic extension as a base that is improved by a trained correction.
Both approaches incorporate prior knowledge in the design of the neural networks, e.g. via respecting theoretical requirements or sophisticated choice of features.
This is a common strategy. Features like periodicity, can be guaranteed in the network architectures by using a feature transformation, see e.g. \cite{yazdani2020systems}.
Theoretical considerations and classical discretization schemes can be used as prior knowledge and 
be respected in the NN architecture. The choice in \cite{ruthotto2020deep} 
is, e.g., based on theory of unsteady partial differential equations. On unstructured meshes, the connection between message-passing neural networks and classical discretization schemes is demonstrated in \cite{lienen2022learning}. Physics-informed neural networks \cite{raissi2019physics} incorporate knowledge or modeling assumptions on the physical model in the objective function. The hybrid PDE-NN approach proposed in \cite{innes2019differentiable, Xu, HybridNNPDE} can itself be viewed as a neural network approach that learns a solution operator $g \mapsto u$ of \eqref{PDEalpha} from data and involves a solve of a partial differential equation in the output layer, which makes evaluations expensive compared to standard neural network architectures. Learning partial differential equations from data is also addressed e.g. in \cite{HollerM, HintermuellerM, Voelkneretal}.
    
To assess the quality of the learned extension operator, we
apply it to the FSI benchmark II \cite{FSIbenchmark}. More precisely, we do a comparison similar to 
\cite[Fig. 6]{Wick}. In order to be able to include different extension operators and assess their quality without having to modify the FSI system, we introduce a novel splitting scheme for the monolithic system. 
We discuss the hybrid PDE-NN approach in section \ref{sec::hybridNNPDE}. In section \ref{sec:nn-corrected} we introduce the approach of correcting standard extension operators with a neural network. Section \ref{sec::FSI} presents the novel FSI splitting scheme and numerical results for the FSI benchmark II. Moreover, generalizability and computational costs of the approaches are discussed in sections \ref{sec:generalizability_comp_cost} and \ref{sec:runtime_comparison}.

\section{Approach 1: Hybrid NN-PDE Approach}
\label{sec::hybridNNPDE}

In order to find the extension operator via \eqref{PDEalpha}, we formulate an optimization problem to find $\bar \alpha$:

Given boundary displacements $\bc^i$, $i \in \lbrace 1, \ldots, N_d \rbrace$, $N_d \in \mathbb N$, we search for weights $\theta$ such that
\begin{align}
    \begin{split}
        \min_{\theta \in \Theta, \diss \in W} \frac1{N_d} \sum_{i=1}^{N_d} |\!|\!| \diss^i - \diss_{\mathrm{biharm}}^i|\!|\!|^2 &+ \lambda \mathcal R(\theta) \\
        \text{s.t.} \quad - \mathrm{div}(\bar \alpha(\theta, \xi, \diss^i, \nabla \diss^i) (\nabla \diss^i)) = 0 \quad & \text{in } \Omega , \\
        \diss^i = \bc^i \quad & \text{on } \partial \Omega,
    \end{split}
    \label{optimizationprob}
\end{align}
where $\lambda$ denotes a regularization parameter,  $\mathcal R(\theta)$ a regularization term, $\Theta$ and $W$ closed subsets of suitable Banach spaces, and $|\!|\!| \cdot |\!|\!|$ is defined by 	\begin{align}
    |\!|\!| \cdot |\!|\!|^2 = \| \cdot \|_{L^2(\Omega)}^2 + \| \nabla \cdot \|_{L^2(\Omega)}^2.
    \label{normdef}
\end{align} 
Moreover,  $\diss_{\mathrm{biharm}}^i = \diss_{\mathrm{biharm}}^i(g^i)$ solves the biharmonic equation for boundary conditions $\bc = \bc^i$. 
In the numerical realizations in this work, we neglect the regularization term and set $\lambda = 0$. For future work, the term $\mathcal{R}$ can be used to add mesh quality measures to the objective function.

In order for the optimization problem to be well-defined, we have to choose $\bar \alpha$ in a way such that the PDE constraint in \eqref{optimizationprob} is solvable for all admissible $\theta \in \Theta$. This requires a suitable choice of $\Theta$ and $\bar \alpha$, which we address in section \ref{SubsectionNN}. We choose $\bar \alpha$ such that it is defined via the derivative of a convex function. Conformal approximations of convex functions via finite elements or polynomials is challenging \cite{chone2001non, wachsmuth2017conforming}. It requires, e.g., the choice of appropriate basis functions and non-negative coefficients, or a set of constraints. Therefore, we model $\bar \alpha$ via input convex neural nets, see \cite{AmosXuKolter, IOCNN} and section \ref{SubsectionNN}.

\subsection{Choice of \texorpdfstring{$\bar \alpha$}{bar alpha}}

The partial differential equation \eqref{PDEalpha} is not necessarily solvable if we choose
$\bar \alpha(\theta, \xi, \diss, \nabla \diss)$ arbitrarily.
In order to find conditions for $\bar \alpha$ under which \eqref{PDEalpha} can be ensured to be uniquely solvable, we take inspiration from the $p$-Laplace equations.
The weak formulation of the $p$-Laplace equations, $p>1$, can be viewed as the optimality conditions of the optimization problem 
\begin{align}
    \min_{\diss \in W} \int_\Omega \| \nabla \diss \|^p d \xi.
    \label{pLopt}
\end{align}
Since \eqref{pLopt} is a convex optimization problem with strictly convex objective, \eqref{pLopt} has a unique solution. More precisely, we have the following result.

\begin{lemma}[see {\cite[Thm. 2.16]{lindqvist2017}}]
    Let $\Omega$ be a bounded Lipschitz domain, $p \geq 2$,  $g \in W^{1,p}(\Omega)$.
    Moreover, let \begin{align}W = \lbrace \diss \in W^{1,p}(\Omega)~:~ \diss \vert_{\partial \Omega} = g\vert_{\partial \Omega} \rbrace. \label{Wspace} \end{align} Then, there exists a unique minimizer of the optimization problem 
    \eqref{pLopt} and the solution of the optimization problem is characterized by 
    \begin{align*}
        (\| \nabla \diss \|^{p-2} \nabla \diss, \nabla \eta)_{L^2(\Omega)} = 0 \quad \forall \eta \in W_0^{1,p}(\Omega).
    \end{align*}
    \label{Lemma24}
\end{lemma}

We prove a similar result for a more general class of PDEs. Hence, we consider a mapping $\Lambda: \mathbb R \to \mathbb R$ and the optimization problem
\begin{align}
    \min_{\diss \in W} \int_\Omega \Lambda(\| \nabla \diss \|^2) d \xi
    \label{opt_prob_lambda}
\end{align}
and work with
\begin{align}
    \alpha (\|\nabla \diss\|^2 ) = 2 \Lambda^\prime(\| \nabla \diss \|^2).
    \label{definition_alpha}
\end{align}

\begin{lemma}
    Let $\Omega$ be a bounded Lipschitz domain and $g$, $W$ be defined as in lemma \ref{Lemma24}.
    Let $\Lambda: \mathbb R \to \mathbb R$ be such that $\Lambda$ is convex,
     strictly increasing,
     continuously differentiable, and assume that
    there exist $a,b,d > 0$, $c \in \mathbb R$ and $p > 2$ (or $p=2$ and $\Lambda$ being affine) such that \begin{align}
            &|\Lambda(t)| \leq a + b t^\frac{p}2,
            \label{upperbound} \\
            &|\Lambda^\prime (t) | \leq c + d t^\frac{p}{p-2}.
            \label{diffAlpha}
        \end{align}
    Moreover, let $\alpha$ be defined by \eqref{definition_alpha}. Then the mapping \begin{align}
        F: W \to \mathbb R, ~\diss \mapsto \int_\Omega \Lambda(\| \nabla \diss \|^2) d \xi,
    \end{align}
    is convex, continuous and Fr\'echet differentiable with derivative
    \begin{align}
        F'(\diss): W^{1,p}_0(\Omega) \to \mathbb R, ~ h_\diss \mapsto \int_\Omega \alpha(\|\nabla \diss\|^2) (\nabla \diss : \nabla h_\diss) d \xi.
    \end{align}
    \label{LemmaF}
\end{lemma}

\begin{proof}
    In order to show continuity and differentiability of $F$, we rewrite
    $F (\diss)= f_3 \circ f_2 \circ f_1 (\diss),$
    where 
    \begin{align}
        \begin{split}
            &f_1: \quad W^{1,p}(\Omega)^d \to L^p(\Omega)^{d \times d}, \quad v \mapsto \nabla v, \\
            &f_2: \quad L^p (\Omega)^{d \times d } \to L^{\frac{p}2} (\Omega), \quad B \mapsto \| B \|^2, \\
            &f_3: \quad L^{\frac{p}{2}}(\Omega) \to \mathbb R, \quad \beta \mapsto \int_\Omega \Lambda(\beta) d \xi.
        \end{split}
    \end{align}
    Since $f_1$ is linear, $f_2$ is strictly convex and $\Lambda$ is strictly increasing and convex, $F$ is convex.
    Due to \cite[Sec. 4.3.3]{Troeltzsch} we know that the superposition operator
    $f_2$ is continuous and continuously differentiable with derivative $f_2^\prime(B): L^p(\Omega)^{d \times d} \to L^{\frac{p}2}(\Omega), ~ h_B \mapsto 2 (B: h_B)$, where $B:h_B := \sum_{i,j} B_{ij} (h_B)_{ij}$. Due to \eqref{upperbound} and \eqref{diffAlpha}, we further obtain continuity and differentiability of $f_3$ with derivative $f_3^\prime(\beta): L^{\frac{p}2}(\Omega) \to \mathbb R, ~ h_\beta \mapsto \int_\Omega \Lambda^\prime(\beta) h_\beta d \xi$. Due to linearity and boundedness, $f_1$ is continuous and differentiable with $f_1^\prime(v): W^{1,p}(\Omega)^d \to L^p(\Omega)^{d \times d}, ~ h_v \mapsto \nabla h_v$. Hence, by applying the chainrule we obtain continuity and differentiability of $F: W^{1,p}(\Omega)^d \to \mathbb R$ with derivative $F^\prime(v)(h_v) = \int_\Omega 2 \Lambda^\prime(\| \nabla v \| ^2) (\nabla v : \nabla h_v) d \xi$. Since $W$ is a closed, affine linear subspace of $W^{1,p}(\Omega)^d$, this concludes the proof. 
\end{proof}

\begin{lemma}
    Let the prerequisites of lemma \ref{LemmaF} be fulfilled. Assume further that there exist $e \geq 0$, $f > 0$, such that
    \begin{align}
        |\Lambda(t)| \geq e + f t^\frac{p}2.
        \label{lowerbound}
    \end{align}
    Moreover, let $g \in W^{1,p}(\Omega)$ and $W$ be defined by \eqref{Wspace}.
    Then, there exists a unique minimizer of the optimization problem 
    \eqref{pLopt} and the solution of the optimization problem is characterized by 
    \begin{align}
        (\alpha( \|\nabla \diss \|^2) \nabla \diss, \nabla \eta)_{L^2(\Omega)} = 0 \quad \forall \eta \in W_0^{1,p}(\Omega).
        \label{firstordercond}
    \end{align}
    \label{Lemma23}
\end{lemma}

\begin{proof}
    The proof follows the line of argumentation of \cite[Thm. 2.16]{lindqvist2017}.
    Assume $\diss_1, \diss_2 \in W$ are solutions of \eqref{opt_prob_lambda}.
    Since $\Lambda$ is convex and strictly increasing and $\nabla \diss \mapsto \| \nabla \diss \|^2$ is strictly convex, we know that $\nabla \diss \mapsto  \Lambda(\| \nabla \diss \|^2)$ is strictly convex. Hence, we know that $\nabla \diss_1 = \nabla \diss_2$. Therefore, $\diss_1 = \diss_2 +C$ for a constant $C> 0$. Due to $(\diss_1 - \diss_2)\vert_{\partial \Omega}  = 0$, $C = 0$. This shows uniqueness of minimizers for \eqref{opt_prob_lambda}. 
    
    In the following, we show existence of solutions. By the continuity of $F$ shown in lemma \ref{LemmaF} and \eqref{lowerbound}, we have
    \begin{align}
        0 \leq I_0 := \inf_{v \in W} \int_\Omega \Lambda(\| \nabla v \|^2) d \xi 
        \leq \int_\Omega \Lambda (\| \nabla g \|^2) d \xi < \infty.
    \end{align}
    Therefore, we can choose a sequence $(v_j)_{j \in \mathbb N} \subset W$ such that 
    $\int_\Omega \Lambda(\| \nabla v_j \|^2) d \xi < I_0 + \frac1j$ for all $j \in \mathbb N$. Due to \eqref{lowerbound}, we know that $\| \nabla v_j \|_{L^p(\Omega)}$ is bounded. Since, by the Poincar\'e inequality, there exists a constant $C>0$ such that $\| w \|_{L^p(\Omega)} \leq C \| \nabla w \|_{L^p(\Omega)}$ for all $w \in W_0^{1,p}(\Omega)$, we know that
    $\| v_j - g \|_{L^p(\Omega)} \leq C \| \nabla (v_j - g) \|_{L^p(\Omega)}$ for all $j \in \mathbb N$. This implies boundedness of the sequence $( \| v_j \|_{W^{1,p}(\Omega)} )_{j \in \mathbb N}$. Hence, there exists a weakly convergent subsequence $(v_j )_{j \in J \subset \mathbb N}$ and $v \in W$ such that $v_j \rightharpoonup v$ weakly in $W^{1,p}(\Omega)$ for $J \ni j \to \infty$. 
    Since, by lemma \ref{LemmaF}, $F: W \to \mathbb R$ is convex and continuous, it is weakly lower semicontinuous. Therefore,
    $F(v) \leq \liminf_{J \ni j \to \infty} F(v_j) = I_0$ and $v$ is the unique minimizer of \eqref{opt_prob_lambda}. The first order necessary optimality conditions for the unconstrained optimization problem then yield \eqref{firstordercond}.
\end{proof}

\subsection{Choice of the neural network}
\label{SubsectionNN}

We do a slight abuse of notation and write $\alpha(\theta, s)$ instead of $\alpha(s)$, as well as $\Lambda(\theta, s)$ instead of $\Lambda(s)$, in order to stress that our choices of $\alpha$ and $\Lambda$ depend on the weights and biases $\theta$ of a neural net.
By the considerations of the previous section, we choose 
$$\bar \alpha(\theta, \xi, u, \nabla u) = \alpha(\theta, \| \nabla u \|^2) = 2 \frac{d}{ds} \Lambda(\theta, \| \nabla u \|^2),$$ where $\Lambda: \Theta \times \mathbb R \to \mathbb R$ fulfills the requirements of \eqref{firstordercond}. More precisely, we work with 
\begin{align}
    \alpha(\theta, s) := 1 + (s - \eta_1)_{+, \epsilon} \frac{d}{d\,s} \widetilde \Lambda (\theta, s) + (s - \eta_2)_{+, \epsilon},
    \label{choiceNN}
\end{align}
where $\eta_1 > 0$, $\eta_2 \gg \eta_1$, $\epsilon > 0$, $(\cdot)_{+, \epsilon}$ denotes a monotonically increasing, smooth approximation of the $\max(\cdot, 0)$ function and $\widetilde{\Lambda}(\theta, s)$ denotes a continuously differentiable, monotonically increasing, non-negative Input Convex Neural Network \cite{AmosXuKolter, IOCNN}, i.e. $s \mapsto \tilde \Lambda (\theta, s)$ is convex, that fulfills $\tilde \Lambda (\theta, s) = \mathcal O(s)$ for $s \to \infty$. The last summand is not realized in the numerics and only needed if the second summand is zero.

The reason why we choose this specific form of $\alpha$ is due to the fact that for small displacements the harmonic extension operator is a computationally cheap and appropriate choice and we want to keep the linearity of the extension operator close to the identity, which is also beneficial for Strategy 3 in section \ref{ssHybrid}.

The $\Lambda$ associated to \eqref{choiceNN} fulfills the properties of lemma \ref{Lemma23}:
Since $s \mapsto \tilde \Lambda(\theta, s)$ is convex and monotonically increasing and the mappings $s \mapsto (s - \eta_1)_{+, \epsilon}$ and $s \mapsto (s- \eta_2)_{+, \epsilon}$ are non-negative and  monotonically increasing, $\frac{d}{ds} \alpha(\theta, s)$ is non-negative and hence the corresponding $\Lambda$ is convex.
Moreover, $\alpha(\theta, s) \geq 1$ and $\Lambda$ is strictly increasing. 
Continuous differentiability of $\Lambda$ follows from $\alpha$ being a composition of continuously differentiable functions. The lower and upper bound estimates \eqref{upperbound}, \eqref{diffAlpha} and \eqref{lowerbound} hold for $p = 4$.

In order to fulfill the requirements we ensure that the weights of $\widetilde{\Lambda}$ are non-negative (we realize this by using squared weights). 
Let $\sigma$ denote an activation function. 
The neural network $\widetilde{\Lambda}(\theta, s) = x_{k+1}$ is given by the recursion
\begin{align}
\begin{split}
    &x_{k+1} = W_k x_k + b_k, \\
    &x_{\ell+1} = \sigma( W_{\ell} x_\ell + b_\ell), \quad \text{for } \ell \in \lbrace 0, \ldots, k-1 \rbrace\\
    &x_0 = s, 
\end{split}
\label{MLP}
\end{align}
with weights $((W_0, b_0), (W_1, b_1), \ldots, (W_k, b_k))$. This type of neural networks is known as a multilayer perceptron (MLP). 

For the hybrid PDE-NN approach, in order to ensure the properties of Lemma \ref{Lemma23} we choose
$$((W_0, b_0), (W_1, b_1), \ldots, (W_k, b_k)) = \psi(\theta),$$ with $\theta = ((W_0^\theta, b_0^\theta), (W_1^\theta, b_1^\theta), \ldots, (W_k^\theta, b_k^\theta))$ and
\begin{align*}
    \psi: ((W_0^\theta, b_0^\theta), (W_1^\theta, b_1^\theta), \ldots, (W_k^\theta, b_k^\theta)) \mapsto ((s[W_0^\theta], b_0^\theta), (s[W_1^\theta], b_1^\theta), \ldots, (s[W_k^\theta], b_k^\theta)),
\end{align*}
where $s: t \mapsto t^2$, and we use the notation in \cite{hiai2009monotonicity} to denote by $s[\cdot]$ the entrywise application of $s$ to a matrix $\cdot$. We work with the differentiable activation function
$\sigma(x) = \mathrm{ln}(1+e^x)$, 
with its derivative, the sigmoid function,
$\sigma^\prime(x) = \frac{e^x}{1 + e^x} = \frac{1}{1 + e^{-x}}.$
Moreover, since it is not present in the derivative of the network, we choose the bias in the output layer as $b_k^\theta = 0$, where $k$ is the depth of the neural network. Its derivative is therefore given by
$\frac{d}{d\,s} \widetilde{\Lambda}(\theta, s) = y_{k+1} $, where
\begin{align*}
    & y_{k+1} = W_k y_k, \\
    &(y_{\ell+1})_i = \sum_j \sigma^\prime( W_{\ell} x_\ell + b_\ell)_i (W_{\ell})_{i,j} (y_\ell)_j, \quad \text{for } \ell \in \lbrace 0, \ldots, k-1 \rbrace, \\
    &x_{\ell} = \sigma( W_{\ell-1} x_{\ell-1} + b_{\ell-1}),\quad \text{for } \ell \in \lbrace 0, \ldots, k-1 \rbrace, \\
    &x_0 = s, ~ y_0 = 1.
\end{align*}
Hence, we work with a neural network architecture that is motivated by theoretical considerations. 

\subsection{Approximation Properties of Input Convex Neural Network Architectures}
\label{sec::approxCINNs}

We discuss approximation capabilities of a class of shallow, input-convex neural networks. We show that for one-dimensional input the ansatz class guarantees both convexity and universal approximation which renders it a promising candidate for later numerical studies. The convexity is crucial to ensure the solvability of the PDE at every iteration, whereas the universality provides the guarantee that any convex function can be approximated.
\begin{lemma}\label{LemmaRepresentation_pcf}
    Let $f\in C^0(\mathbb R)$ be convex, coercive, piecewise affine linear with $n\in\mathbb N$ linear regions. Then $f$ can be written in the form 
    \begin{equation}\label{eq:one_d_icnn}
        f(x) = c + \sum_{i=1}^n a_i \operatorname{ReLU}(w_ix + b_i), 
    \end{equation}
    for $a_i\geq0$ and $w_i,b_i,c\in\mathbb R$.
\end{lemma}
\begin{proof}
    Follows via induction over the number of linear regions. 
\end{proof}

\begin{lemma}
    Let $f \in C^0(\mathbb R)$ be convex. Then $f$ can be approximated by ReLU networks as specified in Lemma \ref{LemmaRepresentation_pcf} uniformly on compacta. This means, for any compact subset $K\subset \mathbb R$ and $\varepsilon > 0$ there is a network $n = n(f, K, \varepsilon):\mathbb R \to \mathbb R$ of the form \eqref{eq:one_d_icnn} such that 
    \begin{equation*}
        \sup_{x\in K} |f(x) - n(x)| < \varepsilon/2.
    \end{equation*}
\end{lemma}
\begin{proof}
    Follows from Lemma \ref{LemmaRepresentation_pcf} and uniform continuity of $f$ on any compact set $K \subset \mathbb R$.
\end{proof}

 \begin{remark}
    For $d \geq 2$, shallow input convex neural networks of the form 
    \begin{align}
        f_\theta(x) = \sum_{i=1}^N c_i \rho(\tilde w_i \cdot x + \tilde b_i),
        \label{eq::architectureshallow}
    \end{align}
    where $c_i\geq 0$ and $\rho = ReLU$, yield no universal approximation to convex functions, see section \ref{appendix}. Similar results for classical discretizations can e.g. be found in \cite{chone2001non, wachsmuth2017conforming}. There are other ansatz classes for input convex neural networks that have better approximation properties, see e.g. \cite{balazs2015near}.
 \end{remark}
    
\section{Approach 2: NN-corrected Harmonic Extension Approach}\label{sec:nn-corrected}

In contrast to the approach of learning $\bar{\alpha}$ in a non-linear PDE \eqref{PDEalpha} defining an extension operator, we explore a more direct approach by learning a neural network that realizes a pointwise, additive correction to the harmonic extension operator. This neural network takes pointwise input features from the harmonic extension and produces a correction to deal with boundary deformations where the harmonic extension degenerates. Typical boundary extension operators are based on PDEs, so we choose to use the component values and first order derivatives as input features, along with spatial coordinate $\xi$. We define the neural network-corrected harmonic extension as 
\begin{equation}\label{eq:corrected_harm_extension}
    u(\xi) = u_\mathrm{harm}(\xi) + l(\xi) \cdot \mathcal{N}_\theta \left(\xi, u_\mathrm{harm}(\xi), \nabla u_\mathrm{harm}(\xi) \right),
\end{equation}
where $u_\mathrm{harm}$ solves \eqref{eq:harmonic_extension} with boundary condition $u = g$ and $\mathcal{N}_\theta$ is a learned function with parameters $\theta$. The function $l$ is zero on the boundary $\partial\Omega$ and positive in the interior $\Omega$, such that \eqref{eq:corrected_harm_extension} satisfies the boundary condition $u\vert_{\partial\Omega} = g$ exactly and the neural network affects the extension in every point in the interior\cite{McFall2009bc_ann}. We define the learned extension in this way because even for small deviations from the boundary conditions, the discretized geometry will be incorrect for the problem and introduce numerical artifacts for FSI simulations.

We construct our neural network as a standard multi-layer perceptron, defined by the recursion
\eqref{MLP} 
and  $\mathcal{N}_\theta(x) = x_{k+1}$ 
for a network with $k$ hidden layers. We write $\theta$ for the collection of weights $W_\ell$ and biases $b_\ell$, the parameters of the network to be trained, and refer to the dimensionality of $x_\ell$ as the width of the $\ell$-th layer. We choose the Rectified Linear Unit, $\mathrm{ReLU}(x) = \max(0, x)$, as activation function.

We choose to base our extension on the harmonic extension, since it is a linear equation with fast, order optimal solvers available and is suitable for modest boundary deformations. It is therefore a good starting point for the correction, without being prohibitively expensive to compute. 
Moreover, for a fixed architecture, evaluating the learned neural network correction $\mathcal{N}_\theta$ scales linearly with the number of mesh points in terms of runtime and thus 
the NN-corrected harmonic extension approach should scale well for larger meshes.

We determine the function $l$ in \eqref{eq:corrected_harm_extension} as the solution of the Poisson problem
\begin{equation}\label{eq:mask_poisson}
    -\Delta l = f \text{ in } \Omega, \quad
    l = 0 \text{ on } \partial\Omega,
\end{equation}
which for a wide class of $f$ give solutions $l$ that are strictly positive in the interior $\Omega$ and smooth, see lemma \ref{result:mask_regularity}. The option of $f \equiv 1$ is one choice of $f$ such that this lemma holds. 

\begin{lemma}\label{result:mask_regularity}
    If $f \in C^\infty(\Omega)$, is bounded, non-negative and not identically zero, $\Omega$ is bounded, and $\partial\Omega$ is \emph{regular} at every point, then there exists a unique classical solution $l$ of \eqref{eq:mask_poisson}, satisfying $l \in C(\bar{\Omega}) \cap C^\infty(\Omega)$ and $l(\xi) > 0$ for all $\xi \in \Omega$. 
\end{lemma}
\begin{proof}

    By \cite[Thm. 4.3]{Gilbarg-Trudinger}, the Dirichlet problem \eqref{eq:mask_poisson} has a unique classical solution if every boundary point of $\Omega$ is regular and $f$ is bounded and locally Hölder continuous over $\Omega$. Thus, there exists a classical solution $l$ of \eqref{eq:mask_poisson}, which is by definition $C^2(\Omega) \cap C(\bar{\Omega})$. 
    By the strong maximum principle, \cite[Thm. 3.5]{Gilbarg-Trudinger}, since $\Delta l \leq 0$ in $\Omega$ and $l = 0$ on $\partial \Omega$, $l$ is positive in the interior. Since $l$ is a classical solution of \eqref{eq:mask_poisson}, it is also a weak solution in $H^1(\Omega)$ and is $C^\infty(\Omega)$ when $f$ is $C^\infty(\Omega)$, \cite[Sec. 6.3, Thm. 3]{evans-pde}. 
\end{proof}

\begin{remark}{
    The domain in the FSI benchmark example is two-dimensional and bounded by a finite number of simple closed curves, hence all boundary points are regular \cite[p.26]{Gilbarg-Trudinger}.}
\end{remark}

The fact that $l$ is strictly positive in the interior ensures that it is zero nowhere, so the neural network $\mathcal{N}_\theta$ can correct the extension everywhere. 
The values of $l$ can also be normalized such that $0 < l(\xi) < 1$ everywhere in $\Omega$, since by scaling $f$ the solution of the linear PDE \eqref{eq:mask_poisson} will scale accordingly. 

The function $f$ was chosen such that the resulting $l$ solving \eqref{eq:mask_poisson} weights the areas where mesh degeneration is expected to happen, and is proportional to 
\begin{equation}\label{eq:mask_rhs_hand_tuned}
    \tilde{f}(x, y) = 2  (x+1)(1-x) \exp( -3.5 x^7 ) + 0.1,
\end{equation}
which we then normalized. Figure \ref{Figure_mask_functions} shows the difference in $l$ resulting from choosing $f$ by \eqref{eq:mask_rhs_hand_tuned} compared to using the simple $f = 1$. The expression for $f$ by \eqref{eq:mask_rhs_hand_tuned} was chosen by first assuming $f$ only depends on $x$, due to the near symmetry of the domain. If \eqref{eq:mask_poisson} is actually one-dimensional, the equation reduces to choosing the curvature of $l$, by $-l'' = f$ in $\Omega$. Then, $f$ close to zero in the right half of the domain makes $l$ close to zero as well, due to the boundary condition $l=0$ on the right boundary and positivity due to the maximum principle. With this motivation, we experimented with terms that kept $f$ positive over the whole domain and small in the right half, until we found an $l$ we were satisfied with. The significance of the choice of $f$ is analyzed in section \ref{sec:parameter_study}.

\begin{figure}
    \centering
    \includegraphics[width=0.49\textwidth, trim={0 1cm 0 0},clip]{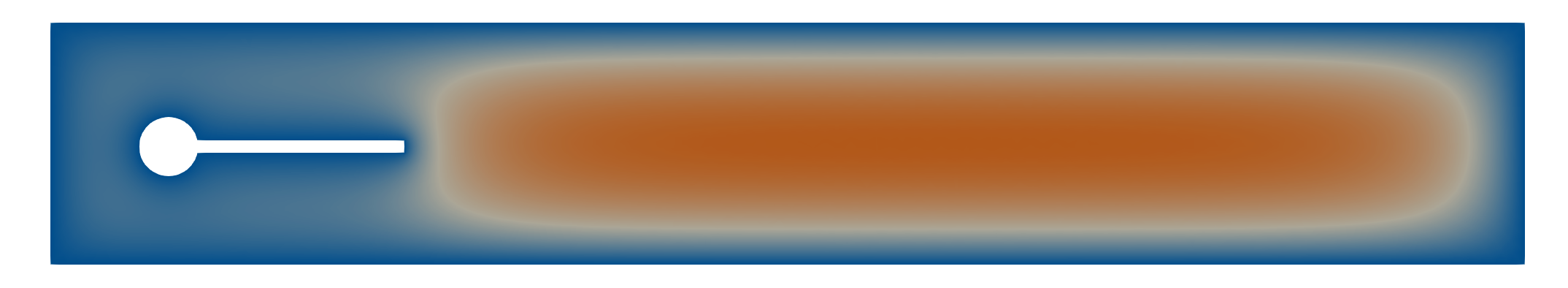}
    \includegraphics[width=0.49\textwidth, trim={0 1cm 0 0},clip]
    {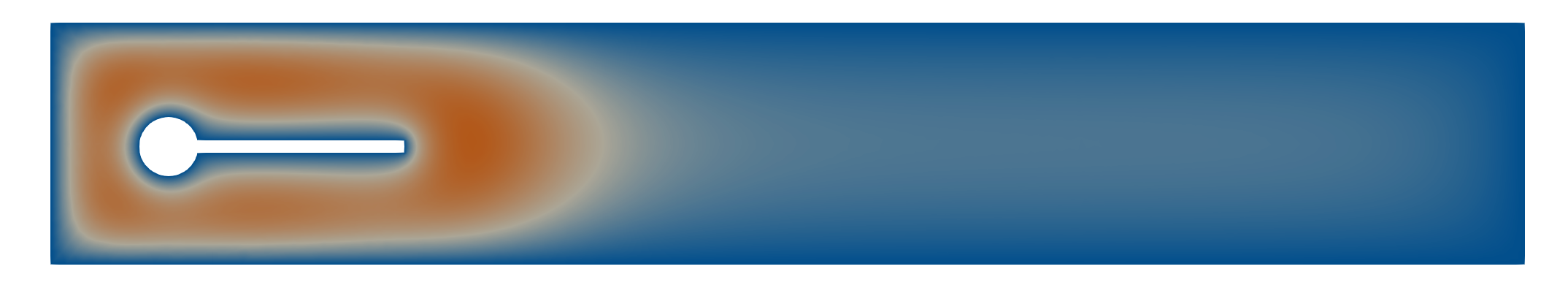}
    \caption{Comparison of solutions to \eqref{eq:mask_poisson} with $f=1$ (left) and the hand-tuned $f$ \eqref{eq:mask_rhs_hand_tuned} (right).
    }
    \label{Figure_mask_functions}
\end{figure}

Since closed form solutions of \eqref{eq:mask_poisson} for arbitrary domains $\Omega$ and source terms $f$ are in general not available, the solution is approximated using an appropriate discretization. 
Maximum principles do not necessarily translate into discrete ones, so since positivity of $l$ in the interior is necessary, we selected continuous linear Lagrange ($\mathbb{P}_1$) finite elements with a Delaunay triangulated mesh \cite[Thm. 3.1]{Huang2011discrete_max}, which together satisfy a discrete maximum principle.
We note that we only need to solve for $l$ once for any given mesh and reuse the solution afterwards.

The neural network correction $\mathcal{N}_\theta$ defines a $d$-dimensional vector field over the domain that we can evaluate where needed to solve the total computational problem. However, for the use in the FSI test problem it is convenient to embed the learned 
mesh displacement into a finite element function space. To this end, similar to the hybrid PDE-NN approach, we use the space of continuous piecewise quadratic functions ($\mathbb{P}_2$). For simplicity, the 
$\mathbb{P}_2$ embedding of $\mathcal{N}_\theta$ is obtained by linear interpolation (into midpoints) of its $\mathbb{P}_1$ representation. In particular, this construction only requires evaluation of $\mathcal{N}_\theta$ at mesh vertices.

In this work we restrict ourselves to neural networks operating on pointwise values, allowing the same network to be used for any triangulation of the domain, just by applying the network to each mesh point. We then need to encode sufficient information into the network to allow for correction of the mesh extension of varying boundary deformations. For example, predicting the extension correction using only the spatial coordinates $\xi$ is not practicable, since the network would not be able to distinguish the different boundary conditions $u\vert_{\partial\Omega} = g$. In addition to $\xi$ we also include information from the already computed harmonic extension $u_\mathrm{harm}$. To complement this (vertex-)local information, the final network input is an approximate gradient of $u_\mathrm{harm}$ at the vertex. We remark that since $u_\mathrm{harm}$ is represented by a $C^0$-conforming finite element function, its gradient in a vertex cannot be simply obtained by standard nodal interpolation. Instead, we apply the Cl{\'e}ment interpolation \cite{clement1975approximation}, which computes the approximation of $\nabla u_\mathrm{harm}$ as the $L^2$ projection over a patch composed of the finite element cells connected to a given vertex. We note that the $L^2$ approximation error of the recovered gradient (as $\mathbb{P}_1$ function) decreases linearly with the mesh size \cite{clement1975approximation}. Let us finally note that the local averaging gathers information from the neighboring vertices and is in this sense similar to aggregation in graph-neural networks \cite{ZHOU202057graph_neural_networks_review}. In summary, 
the inputs to our neural network $\mathcal{N}_\theta$ at $\xi$ are $(\xi, u_\mathrm{harm}(\xi), D_c u_\mathrm{harm}(\xi))$, where $D_c$ denotes approximating the Jacobian by Clément interpolation.

\section{FSI example}
\label{sec::FSI}

In order to test the learned extension operators and compare them with respect to extension quality, we apply them to the FSI benchmark II. 
{\footnote[1]{The code is available at \url{https://github.com/JohannesHaubner/LearnExt} and \url{https://github.com/ottarph/learnext-correction}.}} 
We consider the coupling of the Navier-Stokes equations with
St. Venant-Kirchhoff (STVK) or Incompressible Mooney-Rivlin (IMR) type material and choose a monolithic ALE setting formulated on a 
fixed reference domain. 

For the sake of convenience and brevity, we present the PDE system with a harmonic extension equation for the fluid displacement - other extension equations are more appropriate for large displacements and can be straightforwardly included into the system of equations \cite{Wick, HaubnerDiss} as long as they can be represented as PDEs. Let $\rho_f$, $\rho_s$ denote the fluid density and structure density, and $\nu_f$ denote the fluid viscosity. Moreover let $\mu_s$, $\lambda_s$ and $\mu_2$ denote Lam\'e parameters. The FSI model reads as follows.
\begin{equation}\label{eq:monolithic_fsi}
\left.\begin{aligned}
    J \rho_f \partial_t v_f + J \rho_f((F^{-1}(v_f - \partial_t w_f))\cdot \nabla) v_f - \mathrm{div} (J \sigma_f F^{-\top}) = J \rho_f f &\quad \text{on } \Omega_f \times (0,T), \\
    \mathrm{div}(JF^{-1} v_f) = 0& \quad \text{on } \Omega_f \times (0,T), \\
    v_f = v_{fD}& \quad \text{on }\Gamma_{f, D} \times (0,T), \\
    v_f(0) = 0 & \quad \text{on } \Omega_f, \\
    \rho_s \partial_t v_s - \mathrm{div}(J \sigma_s F^{-\top}) = \rho_s f_s & \quad \text{on } \Omega_s \times (0,T), \\
    \rho_s \partial_t w_s - \rho_s v_s = 0& \quad \text{on } \Omega_s \times (0,T), \\
    w_s(0) = 0& \quad \text{on } \Omega_s, \\
    v_s (0) = 0& \quad \text{on } \Omega_s, \\
    - \Delta w_f = 0& \quad \text{on } \Omega_f \times(0,T), \\
    w_f = 0 & \quad \text{on } \Gamma_f \times (0,T), \\
    \partial_t w_s = v_s = v_f& \quad \text{on }\Gamma_i \times (0,T), \\
    - J \sigma_f F^{-\top} n_f = J \sigma_s F^{-\top} n_s& \quad \text{on } \Gamma_i \times (0,T), \\
    w_f = w_s& \quad \text{on }\Gamma_i \times (0,T),
\end{aligned}\right\}
\end{equation}
where 
$T> 0$, $\Omega_f$ denotes the fluid domain with boundary $\Gamma_f \cup \Gamma_i$. $\Gamma_{f,D} \subset \Gamma_f$ denotes a subset of the boundary on which Dirichlet boundary conditions for the fluid flow are prescribed. $\Omega_s$ denotes the structure domain with boundary $\Gamma_f \cup \Gamma_i$, where $\Gamma_i$ denotes the interface between the fluid and structure domain. $v_f$ and $v_s$ denote the fluid \text{and} structure velocity, respectively, $p_f$ the fluid pressure, $w_s$ the structure displacement, and $w_f$ the extension of the structure displacement to the fluid domain. In addition, 
\begin{align*}
    &\sigma_f(F) = \rho_f \nu_f (D v_f F^{-1} + F^{-\top} Dv_f^\top) - p_f \mathrm{I}, \\
    & F = \mathrm{I} + D w, \quad J(F) = \mathrm{det} (F), \quad E(F) = \frac12 (F^\top F - \mathrm I)
\end{align*}
and 
\begin{align}
    \sigma_s(F) = J(F)^{-1} F (\mu_s (F^\top F - \mathrm I) +\lambda_s \mathrm{tr}(E(F)) \mathrm I) F^\top &\quad \text{(STVK)}, \label{eq:stvk_material}\\ 
    \sigma_s(F) = \mu_s F F^\top - \mu_2 F^{-\top} F^\top - p_s \mathrm{I} & \quad \text{(IMR)} \label{eq:imr_material}
\end{align}
for STVK or IMR type material. 
In case of IMR type material we add the equation 
\begin{align*}
    J(F) - 1 = 0 \quad \text{in } \Omega_s \times (0,T).
\end{align*}
In order to be able to work with functions defined on the whole domain, in the case of STVK type material, we add the equation
\begin{align*}
- \alpha_p \Delta p_s = 0 \quad \text{in } \Omega_s \times (0,T)
\end{align*}
for a constant $\alpha_p = 10^{-9}.$
Whenever it is clear from the context we write $J$, $E$, $\sigma_f$, $\sigma_s$ instead of $J(F)$, $E(F)$, $\sigma_f(F)$, $\sigma_s(F)$. In order to simplify notation, we also do not explicitly write the subscripts for the states $v$, $p$ and $w$.

$D \cdot$ denotes the Jacobian of function $\cdot$. We have compatibility conditions: $v_{fD} = \partial_t w_{fD}$, $v_{fD}(0) = 0$. 
Let 
\begin{align*}
    &V \subset \lbrace v \in H^1(\Omega)^d ~:~ v \vert_{\Gamma_{f}} = v_{fD} \rbrace, \\
    & V_0 \subset \lbrace v \in H^1(\Omega)^d ~:~v \vert_{\Gamma_f} = 0 \rbrace, \\
    & W \subset \lbrace v \in H^1(\Omega)^d ~:~ w\vert_{\Gamma_{f}} = w_{fD} \rbrace, \\
    & W_0 \subset \lbrace v \in L^2(\Omega)^d~:~ v\vert_{\Omega_f} \in H_0^1(\Omega_f)^d, ~ v \vert_{\Omega_s} \in  H^1 (\Omega_s)^d \rbrace, \\
    & P \subset \lbrace p \in L^2(\Omega_f)~:~ \int_{\Omega_f} p d\xi = 0 \rbrace.
\end{align*}
For the sake of clarity, we restrict the presentation to the STVK type material, mention explicitly when IMR type material is considered, and assume that $\Gamma_{f,D} = \Gamma_f$ for the presentation of the weak formulation and the splitting scheme (see section \ref{Section2}).
The weak formulation (for STVK type material) is given by: Find $(v,p, w) \in V \times P \times W$ such that
\begin{align}
\begin{split}
    &\mathcal A(v, p, w)(\psi^v, \psi^p, \psi^w) \\
    &= (J \rho_f \partial_t v, \psi^v)_{\Omega_f} + (J \rho_f ((F^{-1}(v - \partial_t w))\cdot \nabla)v, \psi^v)_{\Omega_f} + (J \sigma_f F^{-\top}, D \psi^v)_{\Omega_f}  \\
    &- (J \rho_f f, \psi^v)_{\Omega_f} + (\rho_s \partial_t v, \psi^v)_{\Omega_s} + (J \sigma_s F^{-\top}, D \psi^v)_{\Omega_s} 
    + (\rho_s (\partial_t w - v), \psi^w)_{\Omega_s}  \\ 
    & + \alpha_w (D w, D \psi^w)_{\Omega_f} + (\mathrm{div}(J F^{-1} v), \psi^p)_{\Omega_f} + \alpha_p ( \nabla p, \nabla \psi^p)_{\Omega_s} = 0
\end{split}
\label{weakform}
\end{align}
for all $(\psi^v, \psi^p, \psi^w) \in V_0 \times P \times W_0$. 

\subsection{Splitting the FSI problem into smaller subproblems}
\label{Section2}

\begin{algorithm}[t]
  \SetKwInOut{Input}{Input}
  \SetKwInOut{Output}{Output}
  \SetKwProg{try}{try}{:}{}
  \SetKwProg{catch}{catch}{:}{end}
  \SetKwRepeat{Do}{do}{while}
  \Input{$(v(t), p(t), w(t))$ for a time-point $t\geq0$ (discretized as described in section 4.2); initial time-step size $\Delta t_{init}$; boundary conditions $v_{fD}(t + \delta t)$ for all $\delta t \in (0, \Delta t_{init}]$; $\theta \in [0,1]$ (for time-stepping)}
  $\Delta t \gets \Delta t_{init}$; $\mathrm{success} \gets \texttt{False}$ \\
  \While{$\mathrm{success} = \texttt{\upshape False}$ and $\Delta t \geq \Delta t_{min}$}{
  \try{}{
    compute $(v(t+\Delta t)$, $p(t + \Delta t)$ by solving (29) (discretized as described in section 4.2) \\
    compute $w(t + \Delta t)$, by $w(t) + \Delta t( (1 - \theta) v(t) + \theta v(t + \Delta t)$ on the solid domain and with the mesh motion operator on the fluid domain \\
    update $(v(t+\Delta t)$, $p(t + \Delta t)$ by solving (31) (discretized as described in section 4.2) \\
    $\mathrm{success} \gets \texttt{True}$
  }
  \catch{Exception}{
    reduce $\Delta t$ (see section 4.2, e.g. $\Delta t \gets \frac{\Delta t}2$)
  }
  }
  \uIf{$\mathrm{success} = \texttt{\upshape True}$}{
  \Return $(v(t+\Delta t), p(t+ \Delta t), w(t+ \Delta t))$; $\Delta t$
  }\Else{
  \textbf{raise} \textit{Exception}
     }
  \caption{Time-step to solve FSI system}
  \label{alg:step}
\end{algorithm}

\begin{algorithm}[t]
  \SetKwInOut{Input}{Input}
  \SetKwInOut{Output}{Output}
  \SetKwProg{try}{try}{:}{}
  \SetKwProg{catch}{catch}{:}{end}
  \SetKwRepeat{Do}{do}{while}
  \Input{$(v(0), p(0), w(0))$ (discretized as described in section 4.2); maximal time-step size $\Delta t_{max}$; minimal time-step size $\Delta t_{min}$; simulation time $T > 0$; boundary conditions $v_{fD}(t)$ for all $t \in [0, T]$; $\theta \in [0,1]$ (for time-stepping)}
  $t \gets 0$; $\Delta t \gets \Delta t_{max}$; $\mathrm{exception\_raised} \gets \texttt{False}$ \\
  \While{$t < T$ and $\mathrm{exception\_raised} = \texttt{\upshape False}$ }{
    set $\Delta t_{init}$ (see section 4.2, e.g. $\Delta t_{init} \gets \min(2 \Delta t, \Delta t_{max})$ \\
    \try{}{
    compute $\Delta t$ and $(v(t + \Delta t), p(t+ \Delta t), w(t + \Delta t))$ via Algorithm \ref{alg:step} \\
    $t \gets t + \Delta t$
  }
  \catch{Exception}{
    $\mathrm{exception\_raised} \gets \texttt{True}$
  }  
  }
  \caption{Solve FSI equations}
  \label{alg:FSI}
\end{algorithm}

In order to simplify the use of arbitrary extension operators, we propose a splitting scheme that handles the extension in a separate step. To do so, we build on ideas and techniques that are also present in, e.g., \cite{Failer, XuYang}. 
The method proposed in \cite{Failer} substitutes the condition $\partial_t w_s = v_s$, i.e. $w_s (t+ \Delta t) = w_s (t) + \int_t^{t+ \Delta t} v_s(\xi) d\xi$, into $\sigma_s$. 

We motivate a similar procedure, use the reasoning of \cite{Failer} to reduce the system but combine it with another idea. After transformation to the physical domain, the choice of $w_f$ in the interior of $\Omega_f$ does not affect the result of the fluid solution (as long as it is ensured that the corresponding transformation is bi-Lipschitz). Therefore, in the first step, we use $w (t +  \Delta t) = \tilde w(t) + \int_0^{\Delta t} v(t + s) ds$, where $\tilde w(t)$ denotes a deformation field at timepoint $t$. After having computed $v(t+\Delta t)$ and $p(t+\Delta t)$, we compute $w\vert_{\Omega_s}(t+\Delta t)$ and $\tilde w (t+\Delta t)$ by applying an extension operator of the solid displacement onto the fluid domain. In this way, we solve for $(v,p)$ and $\tilde w$ separately without doing an approximation. Then, we need to recompute $(v,p)$ for the modified deformation field $\tilde w$ (which replaces $w$ for the time-stepping). Thus, the procedure requires to solve the fluid system twice, see also Algorithm \ref{alg:step}.

More precisely, assume we know the states $\tilde w$, $v$, $p$ at the time step $t$. In order to get the states at time step $t + \Delta t$ we have to solve three systems of equations. For the case of using an harmonic extension operator the procedure is given as follows.

\textbf{First system. }
Let $F(t + \delta t ) = I + D( \tilde w + \int_{t}^{t+ \delta t} v(\xi) d \xi)$, $J = \mathrm{det}(F)$ and get $v(t+\delta t)$, $p(t + \delta t)$ for any $\delta t > 0$ as the solution of 
\begin{align}
    \begin{split}
        &(J \rho_f \partial_t v, \psi^v)_{\Omega_f} + (J \sigma_f F^{-\top}, D \psi^v)_{\Omega_f}  \\
        &- (J \rho_f f, \psi^v)_{\Omega_f} + (\rho_s \partial_t v, \psi^v)_{\Omega_s} + (J \sigma_s F^{-\top}, D \psi^v)_{\Omega_s} 
        \\ 
        &+ (\mathrm{div}(J F^{-1} v), \psi^p)_{\Omega_f} + \alpha_p ( \nabla p, \nabla \psi^p)_{\Omega_s} = 0,
    \end{split}
\end{align}
with $\sigma_f = \sigma_f(F)$ and $\sigma_s = \sigma_s(F)$.

\textbf{Second system. }
Given the solution of the first system we consider the solution of
\begin{align}
    (\rho_s (\partial_t w - v), \psi_s^w)_{\Omega_s}  
    + \alpha_w (D w, D \psi_f^w)_{\Omega_f}  = 0,
\end{align}
where we impose the Dirichlet boundary conditions $w \vert_{\partial \Omega_f} = (\tilde w + \int_t^{t +\delta t} v(\xi) d \xi) \vert_{\partial \Omega_f}$. The solution of this system is used to update $\tilde w$.
(Here, we use the subscripts $\cdot_f$ and $\cdot_s$ in order to clarify that these two equations are solved independently.) It is straightforward to replace the harmonic extension with an arbitrary extension operator that extends the boundary displacement to the interior of the domain, e.g., also those that are not defined via the solution of a PDE.

\textbf{Third system. }
Given $\tilde w$ we obtain the updated $v, p$ as the solution of 
\begin{align}
    \begin{split}
        &(J \rho_f \partial_t v, \psi^v)_{\Omega_f} + (J \rho_f ((F^{-1}(v - \partial_t w))\cdot \nabla)v, \psi^v)_{\Omega_f} + (J \sigma_f F^{-\top}, D \psi^v)_{\Omega_f}  \\
        &- (J \rho_f f, \psi^v)_{\Omega_f} + (\rho_s \partial_t v, \psi^v)_{\Omega_s} + (J \sigma_s F^{-\top}, D \psi^v)_{\Omega_s}
        \\ 
        &+ (\mathrm{div}(J F^{-1} v), \psi^p)_{\Omega_f} + \alpha_p ( \nabla p, \nabla \psi^p)_{\Omega_s} = 0.
    \end{split}
\end{align}
where $F = I + D( \tilde w)$, $J = \mathrm{det}(F)$, $\sigma_f = \sigma_f(F)$ and $\sigma_s = \sigma_s( F_s)$ with $F_s(t + \delta t) =  \tilde w + \int_{t}^{t+ \delta t} v(\xi) d \xi$. 

\subsection{Discretization} \label{sec:fsi_discretization}
The discretization is done similar to \cite{Wick, Failer, HaubnerDiss}. We use a triangulation of the domains and choose the lowest-order Taylor-Hood finite elements for the velocity and pressure. The deformation is discretized using piecewise quadratic continuous finite elements. As in \cite{Wick}, we further choose a shifted Crank-Nicolson scheme ($\theta$-scheme with $\theta = \frac12 + \Delta t_{max}$ for sufficiently small $\Delta t_{max} > 0$) for performing the time-stepping in case we work with STVK type material and implicit Euler as time-stepping scheme when we work with IMR type material. The pressure term and the incompressibility condition are handled implicitly. Analogous to \cite{Wi13_fsi_with_deal}, in the first summand of the weak form $(J \rho_f \partial_t v, \psi^v)_{\Omega_f}$, $J$ is replaced by a $J_\theta := \theta J(t_n) + (1 - \theta) J(t_{n-1})$, which is a convex combination of the determinant of the deformation gradient of the previous time $t_{n-1}$ and the current time $t_n$.

Since the splitting schemes requires a smaller time-step size than classical monolithic approaches, we work with a time-step size that is iteratively adapted. We choose a time-step-size range $[\Delta t_{min}, \Delta t_{max}]$ and start with time-steps of size $\Delta t = \Delta t_{max}$. If the systems with the maximal time-step get unsolvable, we choose $\Delta t = \max(\Delta t_{min}, \frac12 \Delta t)$. This is repeated until we obtain solvability or reach $\Delta t_{min}$. If the system is not solvable for $\Delta t = \Delta t_{min}$, we stop. After each successful step we adapt the time-step by setting $\Delta t = \min(2\Delta t, \Delta t_{max})$. The algorithm is sketched in Algorithms \ref{alg:step}--\ref{alg:FSI}.

Simulations are performed on a desktop computer with an AMD Ryzen Threadripper 3970X 32-Core CPU and 128 GiB memory. Neural networks for the NN-corrected harmonic extension approach are trained on an NVIDIA GeForce GTX 1660 SUPER GPU.
Nonlinear systems are solved using a Newton based nonlinear solver with linesearch in PETSc \cite{petsc2, petsc3, petsc4}. Linear systems are solved using the direct solver MUMPS \cite{mumps1, mumps2}.

\subsection{Numerical Results for FSI benchmark problem II}

In order to validate the splitting of the FSI system in section \ref{Section2}, we apply it to the FSI benchmark II \cite{FSIbenchmark} and plot the displacement of the tip of the flap in Figure \ref{Figure_classic}, compare e.g. \cite[p. 256]{FSIbenchmark} or \cite[Fig. 5.6]{HaubnerDiss}. 
We test the harmonic extension, the biharmonic extension and an incremental version of the harmonic extension where the extension is performed on the deformed domain of the previous timestep and only the change of the displacement is extended harmonically, see \cite{Shamanskiy2020}. Moreover, Figure \ref{Figure_classic} shows the minimal determinant value of the deformation gradient. The simulations of the standard and incremental (see \cite{Shamanskiy2020}) harmonic extension break if the determinant of the deformation gradient becomes negative. Therefore, with these approaches, it is not possible to simulate the whole time interval.
\begin{figure}
    \includegraphics[width=0.48\textwidth]{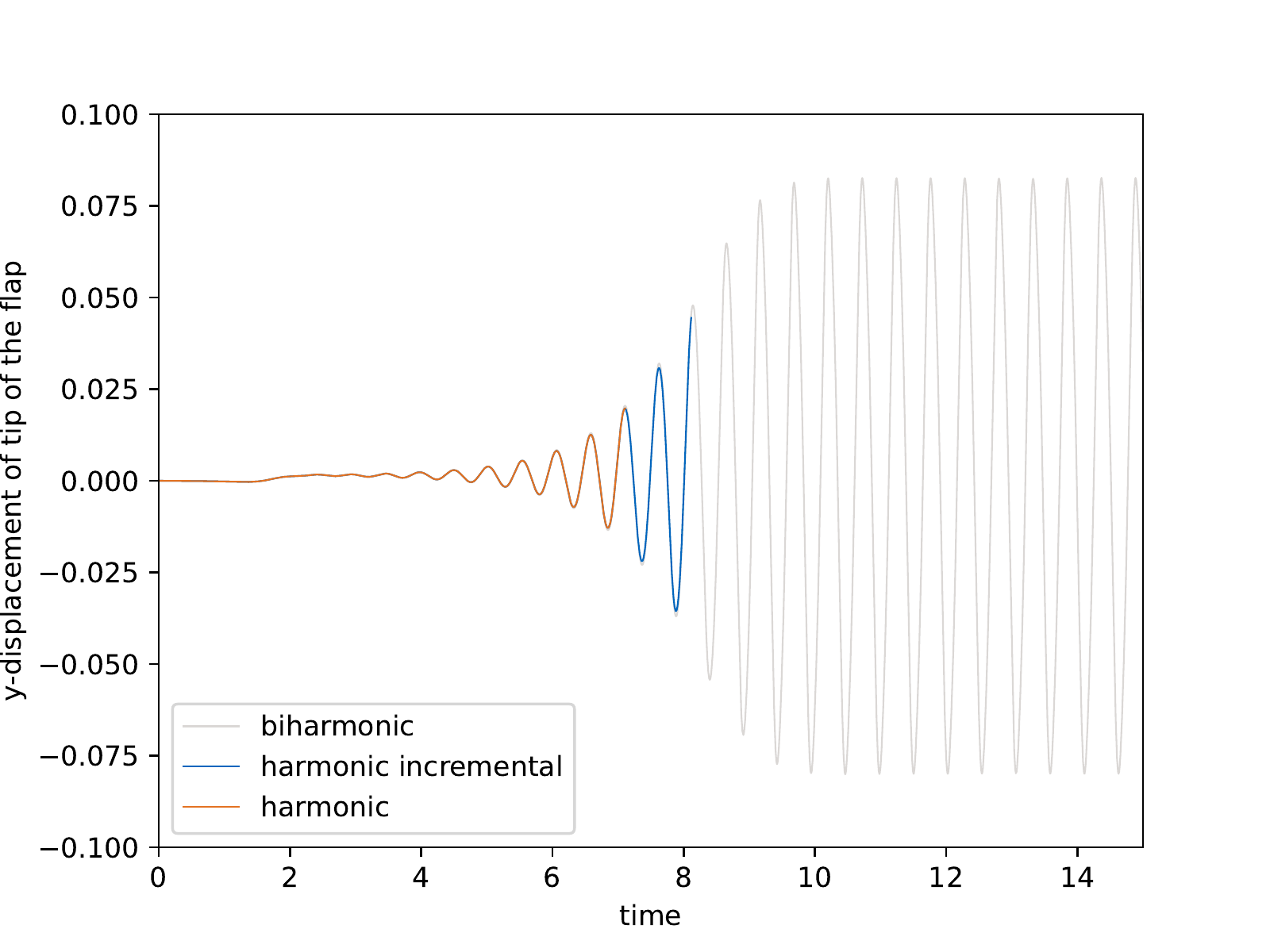}	\includegraphics[width=0.48\textwidth]{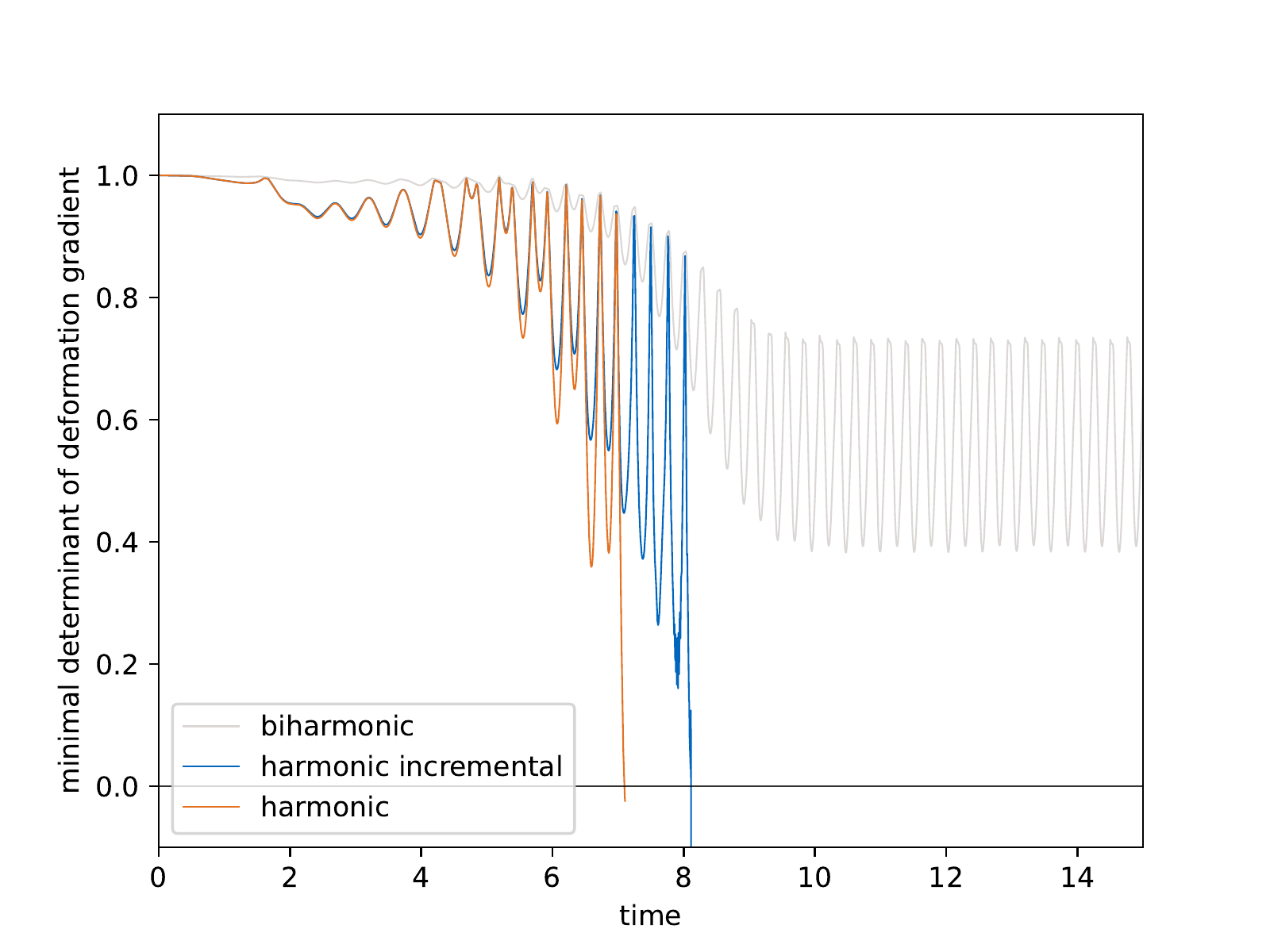}
        \vspace{-15pt}
    \caption{Numerical results for harmonic and biharmonic extensions. }
    \label{Figure_classic}
\end{figure}

\subsubsection{Generation of Training Data}
\label{ssData}

\textbf{\\Training data via FSI benchmark II.}

For a proof of concept study, we first generate training data by running the FSI benchmark II with the biharmonic extension and a time-step size of $0.0025$, and save the deformations for every time step in $(15, 17]$ for the hybrid PDE-NN (and $(11, 17]$ for the NN-corrected harmonic extension approach), i.e., we generate $N_d = 800$ (and $N_d=2400$, respectively) data points.  

\begin{remark}
    This procedure only makes sense if the FSI problem is solved several times as it is the case, e.g., if the FSI equations are the governing equations of a PDE constrained optimization problem. Moreover, since it is not necessary that the test data is physically realistic, the test set can also be artificially generated. Instead of solving the biharmonic equations and using this as a reference value, also unsupervised learning approaches 
    \begin{align*}
        \min_{\theta \in \Theta, \diss \in W} \frac{1}{N_d} &\sum_{i=1}^{N_d}\mathcal J(\diss^i) + \lambda \mathcal R(\theta) \\ 
        & \text{s.t. } \diss^i = \mathrm{Ext}_\theta(\bc^i)\quad \text{for } i \in \lbrace 1, \ldots, N_d \rbrace
    \end{align*}
    can be considered, where $\mathrm{Ext}_\theta$ denotes an extension operator that depends on the parameters $\theta$ and $\mathcal J$ denotes a quality measure for the deformed mesh. For example, one can choose $\mathcal J$ such that the determinant of the transformation gradient is bounded away from $0$, see \cite{HSU20}, or another mesh quality measure is optimized. This, however, is left for future research.
\end{remark}

\textbf{Artificial training data.}

We also created an artificial dataset to train our networks on. This artificial dataset was created by solving for the stationary deformation of a neo-Hookean body representing the solid domain, with the same Lamé parameters as the solid in the FSI simulations. Specifically, we prescribe homogeneous Dirichlet boundary conditions for the displacement on the cylinder boundary and let the 
body deform under boundary loads selected by hand to create deformations of the solid domain resembling by eye the deformations encountered in the FSI benchmark problem II. Six different load configurations were selected and for each of these configurations 101 different deformations were computed, by varying the amplitude of the loads by cosine factors for 101 equally spaced $\theta$ from $0$ to $2\pi$. 
These solid deformations were extended to the fluid domain by the harmonic and biharmonic extension, to serve as input data and target data respectively. 

In each base load configuration, there is one boundary load
\begin{equation*}
    (0, F_\mathrm{tip}) \cdot \mathrm{cos} (\theta)
\end{equation*}
acting vertically on the right side of the deformable body and also a boundary load acting vertically and equally on the bottom and top side of a limited part of the deformable body, given by
\begin{equation*}
    \begin{cases}
        (0, F_\mathrm{side}) \cdot \mathrm{cos} (\theta - \varphi ) \quad & \text{if } \lvert x - c \rvert < d, \\
        (0, 0) & \text{else. }
    \end{cases} 
\end{equation*}
The parameters used to create the artificial dataset are given in table \ref{tab:artificial_dataset_params}.

\begin{table}[]
    \centering
    \begin{tabular}{c|c c c c c c}
        $F_\mathrm{tip}$ & 1.925e3 & 0.6e3 & 1.4e3 & 1.76e3 & 0.53e3 & 1.99e3 \\
        $F_\mathrm{side}$ & -1.7e3 & -0.6e3 & -0.2e3 & -0.66e3 & 0 & -1.94e3 \\
        $\varphi$ & 0 & $-\pi/4$ & 0 & 0 & 0 & 0 \\
        $c$ & 0.4 & 0.4 & 0.5 & 0.45 & 0.45 & 0.4 \\
        $d$ & 0.02 & 0.02 & 0.04 & 0.04 & 0.04 & 0.02
    \end{tabular}
    \caption{Parameters for the six base load configurations used to create the artificial dataset.}
    \label{tab:artificial_dataset_params}
\end{table}

    \begin{figure}
    \centering
    \includegraphics[width = 0.15\textwidth]{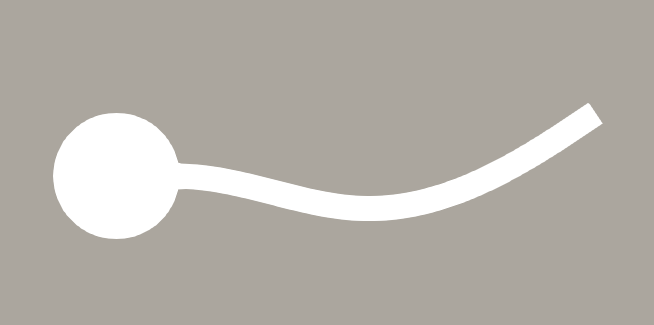}\hfill
    \includegraphics[width = 0.15\textwidth]{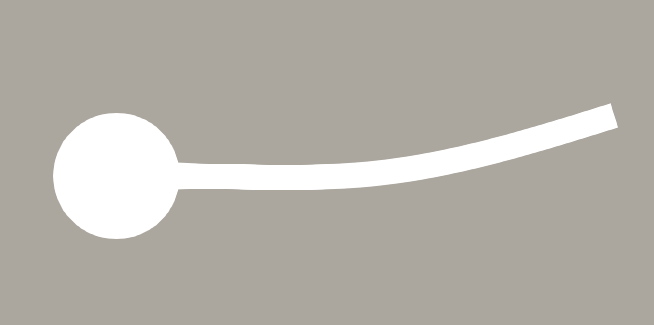}\hfill
    \includegraphics[width = 0.15\textwidth]{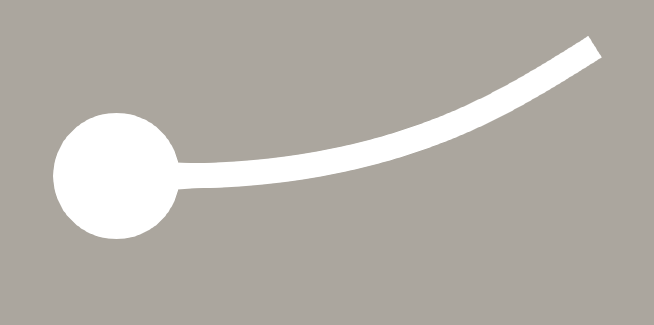}\hfill
    \includegraphics[width = 0.15\textwidth]{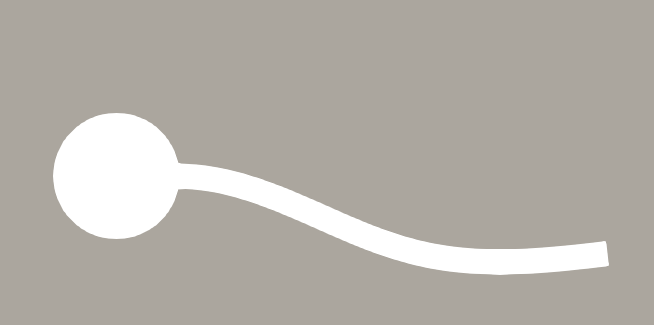}\hfill
    \includegraphics[width = 0.15\textwidth]{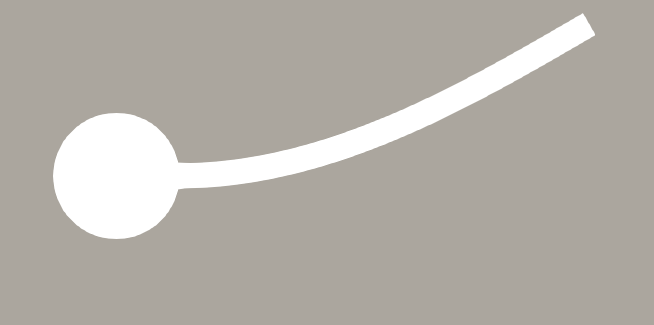}\hfill
    \includegraphics[width = 0.15\textwidth]{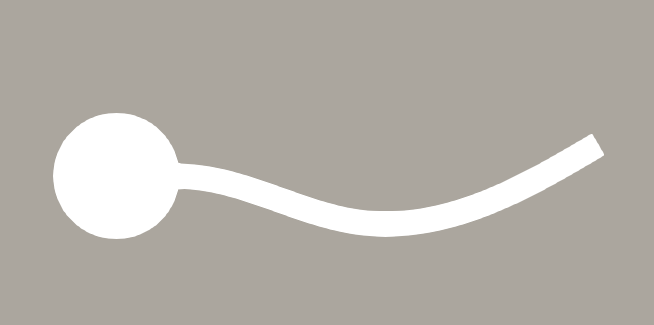}
        \vspace{-10pt}
\caption{The body deformations caused by the six base load configurations the artificial dataset is based on.}
    \label{fig:artificial_dataset_base_loads}
\end{figure}

\subsubsection{Hybrid PDE-NN Approach}
\label{ssHybrid}

In order to reduce the computational burden, we only take every $\lfloor \frac{N_d}N \rfloor$th training data point into account and work with $N = 20$ for the FSI training data set and $N = 60$ for the artificial data set. Every objective function evaluation requires the solution of $N$ nonlinear PDEs (which could potentially be done fully in parallel). In our numerical realization we do not use regularization, i.e., we work with $\lambda = 0$.

We choose $\alpha$ to be a neural network with two hidden layers of width $5$ with bias on the two hidden layers. 
Let $\theta_{opt}$ denote the approximate solution of the optimization problem \eqref{optimizationprob} that we obtain by applying an L-BFGS method with at maximum $100$ iterations. 
Figure \ref{LearnedNet} visualizes a neural network for $\theta = \theta_{opt}$ that we obtained by training on the FSI and artificial dataset. In our numerical experiments $\| \nabla \diss \|^2$ reaches values up to around $0.6$ for the nonlinear learned extension operator applied to the FSI benchmark II.

\begin{figure}
    \centering
    \includegraphics[width = 0.45\textwidth]{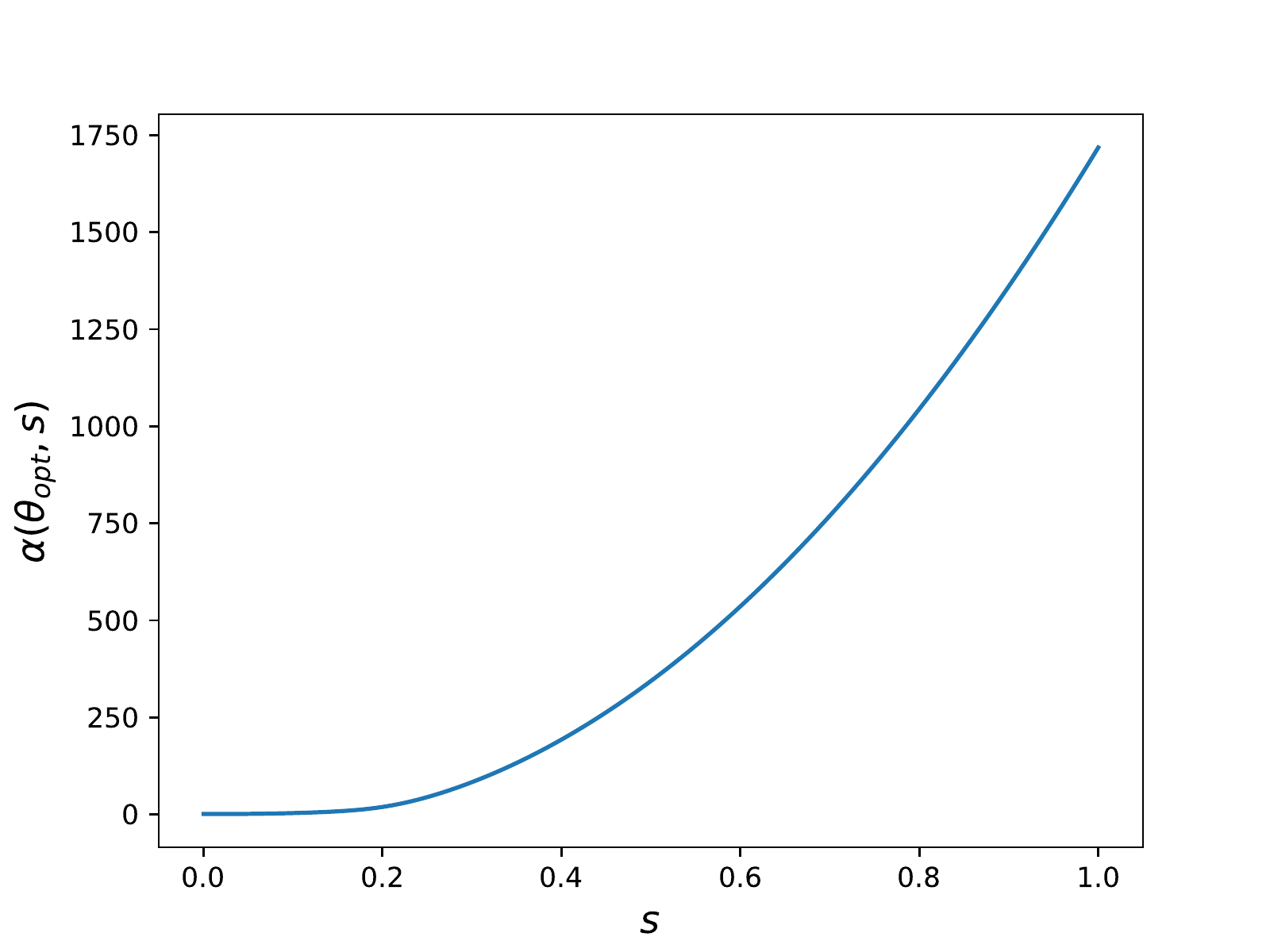}
    \includegraphics[width = 0.45\textwidth]{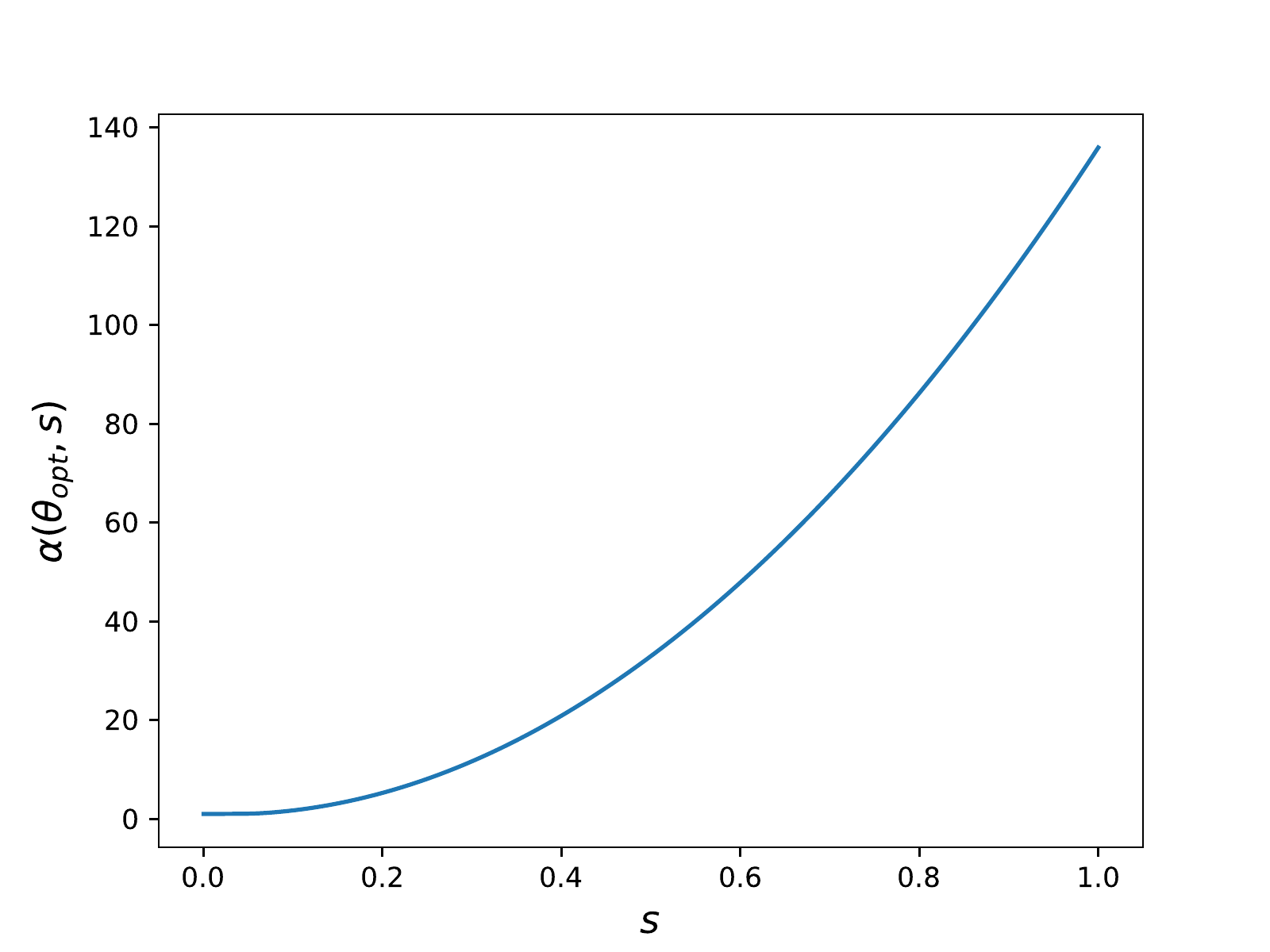}
      \vspace{-15pt}
    \caption{Visualization of learned neural net $\alpha(\theta, s)$ in the hybrid PDE-NN approach based on training data via FSI benchmark II (left) and artificial training data (right).}
    \label{LearnedNet}
\end{figure}

We present numerical results for the FSI benchmark II in Figure \ref{Figure_learned_hybrid} and Figure \ref{Figure_learned_hybrid_artificial}, where we choose $\Delta t_{max} = 0.01$ and $\Delta t_{min}= \frac1{128} \Delta t_{max}$. More precisely, we show the $y$-displacement of the tip of the flap and the minimal determinant of the deformation gradient, which serves as an indicator for mesh degeneration, see \cite{Wick}. We consider three different strategies for solving \eqref{firstordercond}: 
\begin{itemize}
    \item \textbf{Strategy 1:} First we try the learned nonlinear extension operator, i.e., we solve \eqref{PDEalpha} with $\bar \alpha$ defined by \eqref{choiceNN} and $\theta = \theta_{opt}$.
    This strategy does not fail due to mesh degeneration. However, it requires several Newton-steps for large displacements.
    \item \textbf{Strategy 2:} In order to avoid the nonlinearity, we consider a strategy that uses a linear version of the extension operator. Working with a ``lagging nonlinearity'', i.e., solving \eqref{PDEalpha} with $\bar \alpha$ defined by \eqref{choiceNN} being approximated by
    \begin{align}
        \bar \alpha (\theta_{opt}, \xi, \diss, \nabla \diss) \approx \alpha(\theta_{opt}, \| \nabla \diss_{old} \|^2),
    \end{align}
    where $\diss_{old}$ denotes the displacement of the previous time-step, shows unsatisfactory numerical behavior. Therefore, we combine it with the idea of using incremental versions for extension operators, see \cite{Shamanskiy2020}. We present results for the incremental ``lagging nonlinearity'' strategy, where we solve
    \begin{align*}
        - \mathrm{div}(\alpha(\theta_{opt}, \| \nabla \diss_{old} \|^2) \nabla \diss_\Delta) = 0 \qquad &\text{in } (\mathrm{id} + \diss_{old})(\Omega), \\ 
        \diss_\Delta = \bc - \bc_{old} \qquad&\text{on }  \partial ((\mathrm{id} + \diss_{old})(\Omega)), 
    \end{align*}
    and use the update
    \begin{align*}
        u = u_{old} + u_\Delta \circ (\mathrm{id} + u_{old})^{-1}.
    \end{align*}
    This strategy is successful even for the largest appearing displacements. However, the mesh quality decreases during the course of the simulations and it is not possible to simulate the whole $15$ seconds.
    \item \textbf{Strategy 3:} In order to counteract the decreasing mesh quality that occurs with the second strategy during the oscillations, we consider a third strategy that combines strategies 1 and 2. If the vertical displacement of the tip of the flap is below a certain threshold (we use $0.005$), we work with strategy 1, otherwise we choose strategy 2. This approach allows us to simulate the whole $15$ seconds and the mesh quality stays stable during the oscillations. Moreover, the extension equation of strategy 1 just needs to be evaluated for few time-steps and in a regime with small deformations at the tip. Instead of using strategy 1 for small deformations, one can also directly use the harmonic extension. We design the hybrid PDE-NN extension such that it is equal to the harmonic extension for small deformations, see \eqref{choiceNN}.
\end{itemize}

\begin{figure}
    \includegraphics[width=0.48\textwidth]{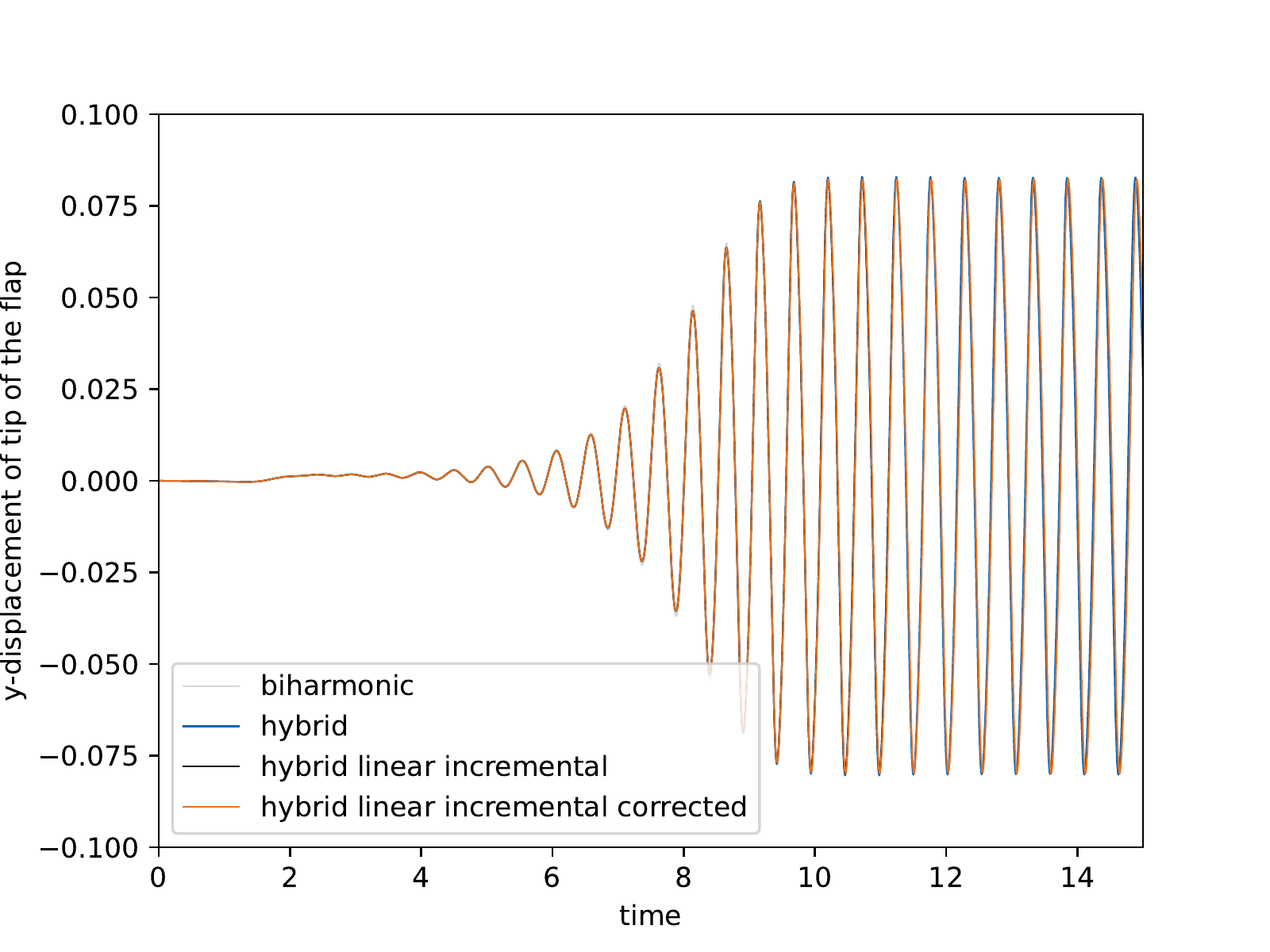}
    \includegraphics[width=0.48\textwidth]{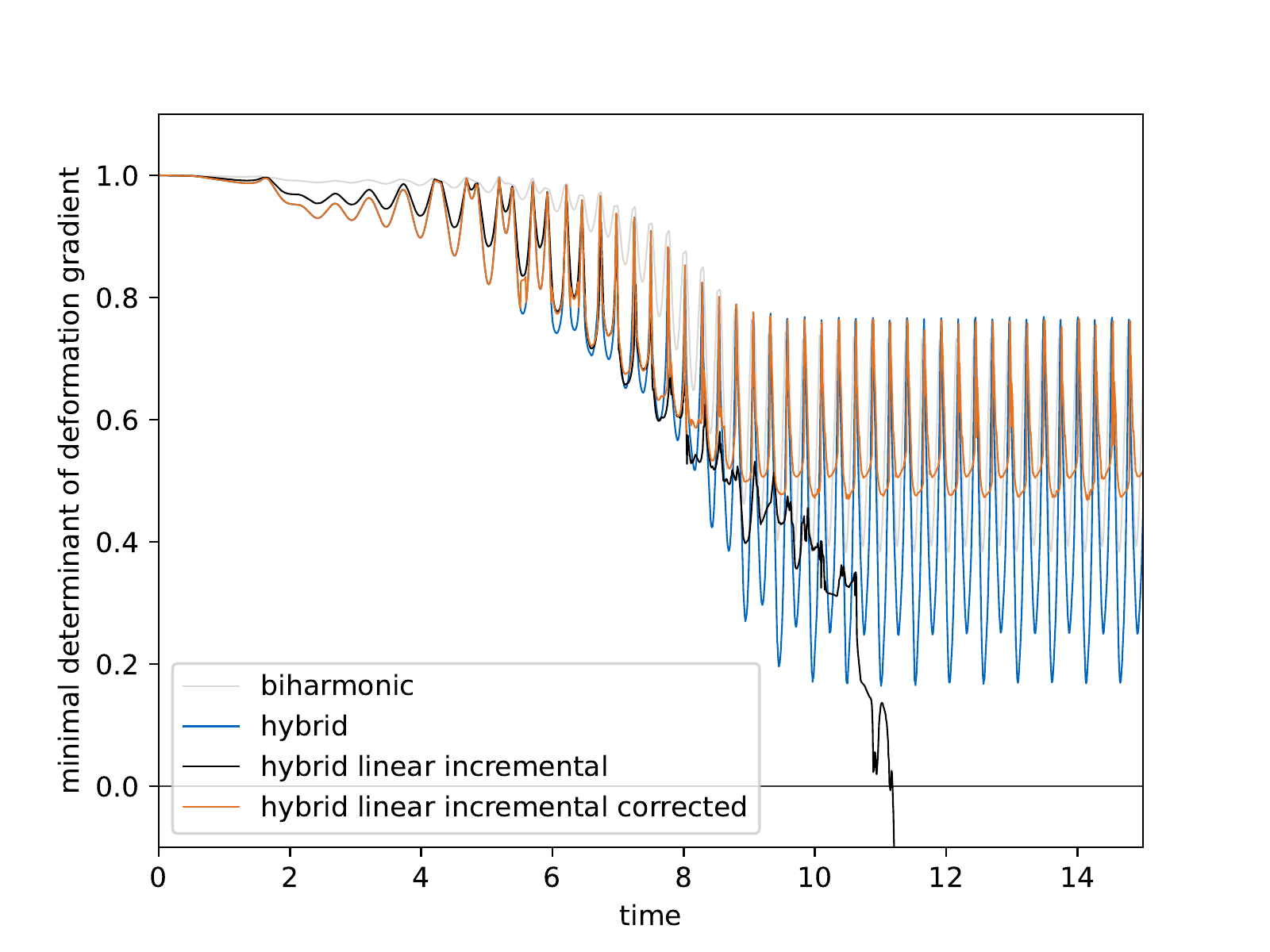}
\vspace{-15pt}
    \caption{Numerical results for hybrid PDE-NN approach, trained on the FSI benchmark II training set.}
    \label{Figure_learned_hybrid}
\end{figure}

    \begin{figure}
    \includegraphics[width=0.48\textwidth]{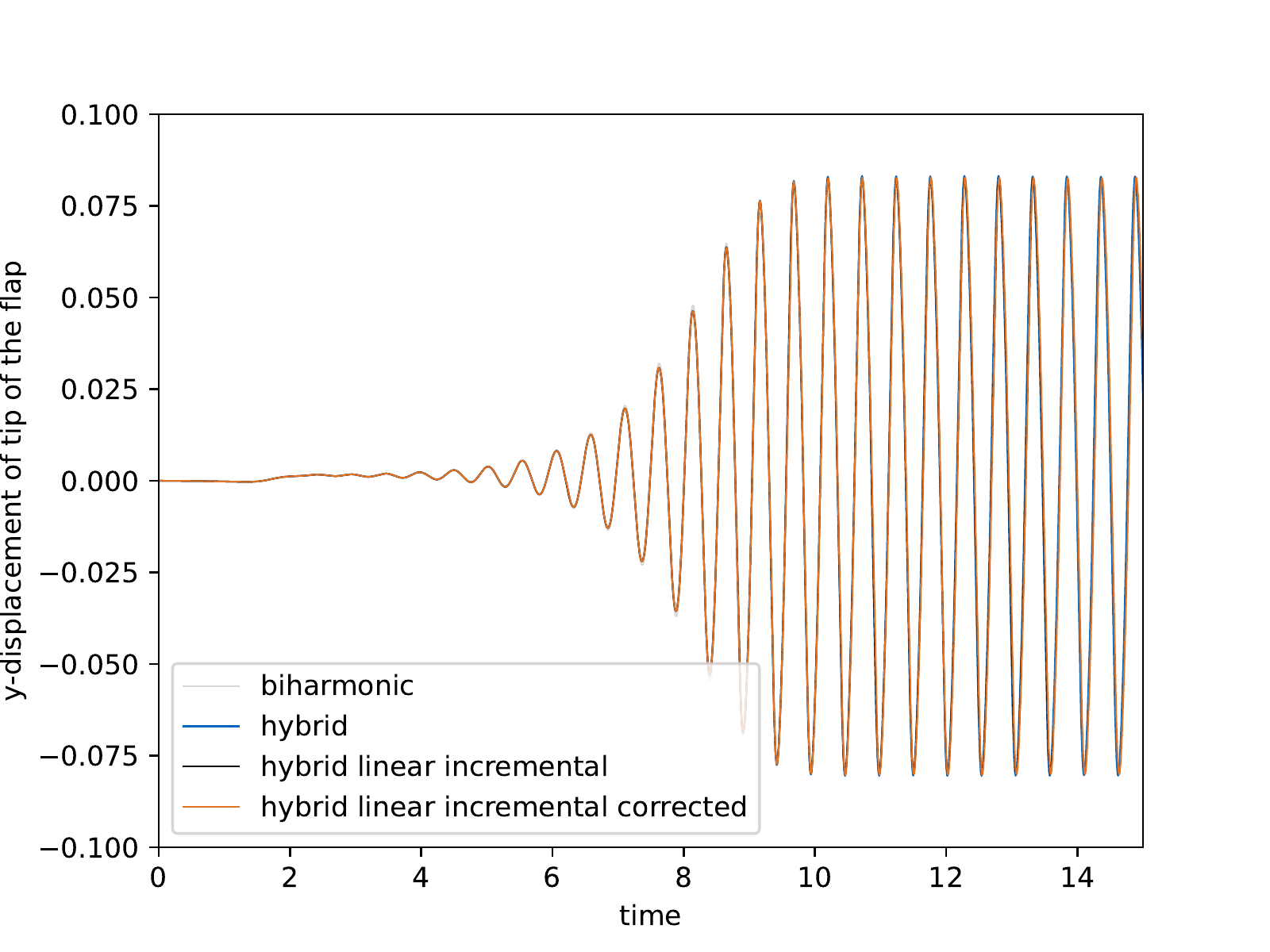}
    \includegraphics[width=0.48\textwidth]{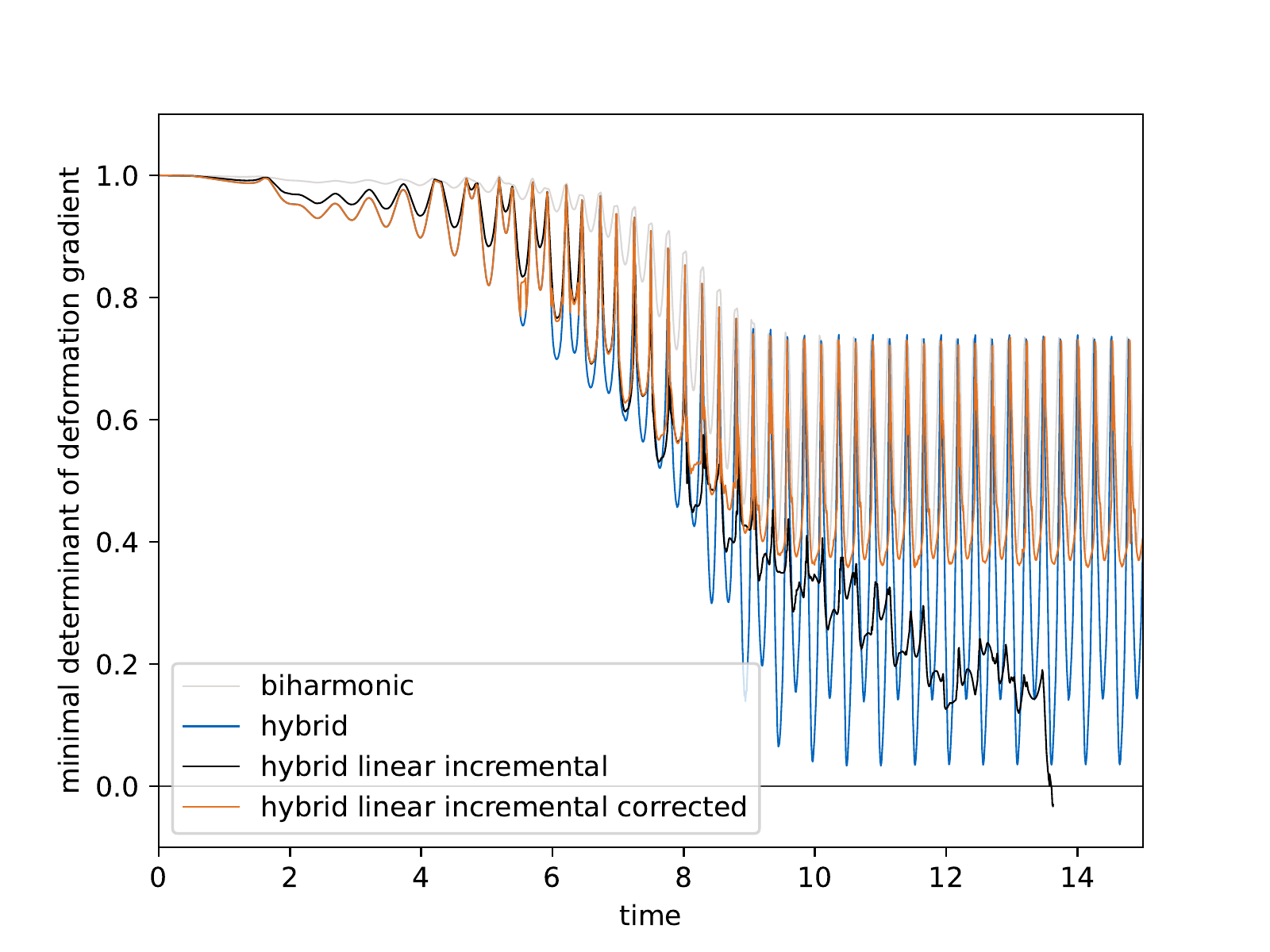}
\vspace{-15pt}
    \caption{Numerical results for hybrid PDE-NN approach, trained on the artificial training set.}
    \label{Figure_learned_hybrid_artificial}
\end{figure}

\subsubsection{NN-corrected harmonic extension approach}\label{sec:corrected_harmonic_results}

For the NN-corrected harmonic extension, we partition the FSI dataset by selecting the first 1800 snapshots for the training dataset, the following 200 snapshots for the validation dataset, and the last 400 snapshots for the test dataset. The validation set is used to make decisions during training, such as for learning rate scheduling, but is not used to compute gradients of the cost function by backpropagation. The test set is set aside until the networks are fully trained and then used to evaluate the performance of the networks. Because the artificial dataset is not periodic like the FSI dataset, it was randomly split into 85\% training data and 15\% validation data. 
The FSI test dataset is used for testing networks trained on both the FSI dataset and the artificial dataset. 

We use a standard MLP for $\mathcal{N}_\theta$ with 6 hidden layers of width 128 and $\mathrm{ReLU}$ activation function. The network has $8$ input features (as here $\Omega\subset\mathbb{R}^2)$, one for each component of $\xi$, $u_\mathrm{harm}$, and $\nabla u_\mathrm{harm}$, and $2$ output features, one for each component of the extension.
Before the MLP, we include a normalization layer, which transforms the neural network's inputs by the affine transformation $x_j \mapsto \frac{x_j - \mu_j}{\sigma_j}$,
where the real numbers $\mu_j$, $\sigma_j$, $j=1, 2, \dots, 8$ are chosen such that the inputs of $\mathcal{N}_\theta$ will have componentwise zero mean and unit variance over all vertices and snapshots of the training set. 
This normalization layer is part of the neural network's architecture and therefore the values of $\mu_j$ and $\sigma_j$ are not changed between training and testing. In contrast to normalization approaches like batch and layer normalization, the values of $\mu_j$ and $\sigma_j$ are not trainable parameters.

We implement the neural network $\mathcal{N}_\theta$ using \texttt{PyTorch}\cite{paszke2019pytorch}. The harmonic extension $u_\mathrm{harm}$ and boundary condition-preserving function $l$ are computed using \texttt{FEniCS}\cite{alnaes2015fenics}, the FEM framework we use for the rest of our computations. The correction made by $\mathcal{N}_\theta$ is translated into the finite element function spaces by adding the vertex evaluations of $\mathcal{N}_\theta$ to the basis coefficient vectors.

We train our networks in a supervised fashion by minimizing the sum of absolute error in each mesh vertex for all snapshots, with the biharmonic extension as target. This can be formulated in the learning problem
\begin{equation}\label{eq:NN-corr_training_problem}
     \min_\theta \sum_{i} \sum_{\xi \in \mathcal{V}} \lVert u_\mathrm{harm}^i(\xi) + l(\xi) \mathcal{N}_\theta\left(\xi, u_\mathrm{harm}^i(\xi), D_c u_\mathrm{harm}^i(\xi) \right) - u_\mathrm{biharm}^i(\xi) \rVert_1,
\end{equation}
where $\mathcal{V}$ denotes the set of vertices in the computational mesh used to generate the dataset.

To solve \eqref{eq:NN-corr_training_problem}, we used the \texttt{AdamW} optimizer\cite{Loshchilov2017AdamW}, with random batches of 128 snapshots and every mesh vertex included for each snapshot in the batch. We used the \texttt{PyTorch}-built in \texttt{ReduceLROnPlateau} learning rate scheduler, with the reduction factor set to $0.5$ and otherwise default parameters, and weight decay regularization with factor $0.01$. On the FSI dataset we trained the network for 500 epochs. When training on the artificial dataset, which is smaller, we stopped the training process after only 200 epochs, as we found the training process had stabilized by then.

The network and training hyper-parameters were chosen by using common practices in deep learning and no extensive hyper-parameter search was conducted. The number and width of the hidden layers was increased until we found the network produced satisfactory mesh motion. A brief analysis of the models dependence on random initialization, the function $l$ used, and the number of hidden layers in the neural network is included in section \ref{sec:parameter_study}.

To predict the correction for a given harmonic extension $u_\mathrm{harm}$, we pre-process $u_\mathrm{harm}$ to supply the Clément interpolated gradients at each mesh vertex. For the training process, we did this pre-processing beforehand. When using the trained network as part of the FSI-simulation, the Clément interpolation is done in each time step before the second step of the splitting method described in section \ref{Section2}. This interpolation is cheap when done repeatedly, as it can be realized by multiplication with the same pre-assembled sparse matrix for any given mesh. As can be seen in Figure \ref{fig:timings_plot_2}, the cost of the Clément interpolation is modest compared to computing the harmonic extension and neural network output and scales well with increased mesh refinement.

To test the fully trained neural network correction \eqref{eq:corrected_harm_extension}, we evaluate it by computing the scaled Jacobian mesh quality measure over the FSI test set and by using the extension strategy on the same FSI test problem as the hybrid PDE-NN approach. We solve the FSI-system \eqref{eq:monolithic_fsi} using the same splitting and discretization approach, outlined in sections \ref{Section2}-\ref{sec:fsi_discretization}. The NN-corrected extension operator is not a PDE, but the splitting of the system allows us to use this extension.

The scaled Jacobian mesh quality measure returns the value 1 for equilateral triangles and goes to zero as a cell degenerates. 
The extension trained on the FSI dataset was able to produce a quality mesh for all snapshots in the test set, with the minimal scaled Jacobian mesh quality measure being approximately 0.32 across all cells and snapshots. Trained on the artificial dataset, the extension was able to produce a functional mesh for all snapshots, but not of high quality. On some snapshots, the resulting mesh was nearly degenerate, with minimal scaled Jacobian mesh quality indicator being approximately 0.07 across all cells and snapshots. The minimum cell values of scaled Jacobian mesh quality indicator over the FSI test set are shown in Figure \ref{fig:min_sjmq}. In Figure \ref{fig:cellcomparison}, we show the cell resulting in the lowest scaled Jacobian mesh quality value before and after the mesh motion.

\begin{figure}
    \centering
    \includegraphics[width=0.98\linewidth]{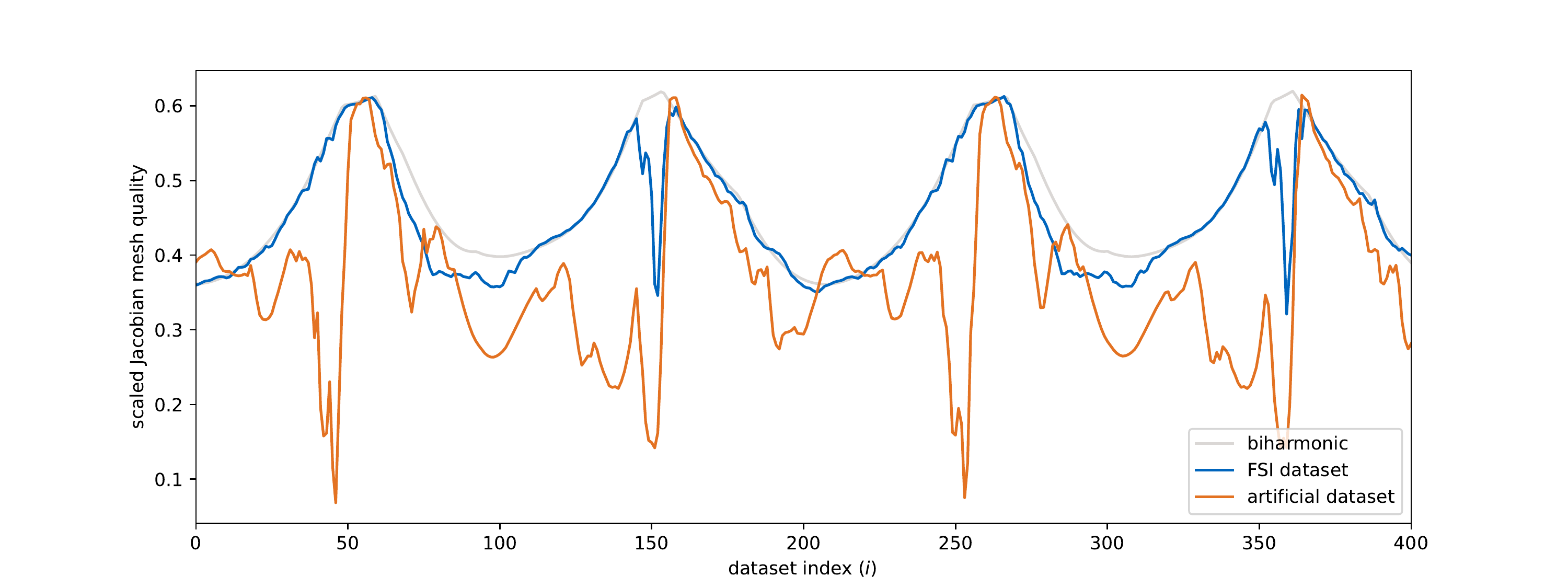}
    \caption{Minimum cell value of scaled Jacobian mesh quality indicator over the FSI test set for biharmonic extension and NN-corrected harmonic extension trained on FSI dataset and artificial dataset.}
    \label{fig:min_sjmq}
\end{figure}

\begin{figure}
    \centering
    \includegraphics[width=0.9\linewidth]{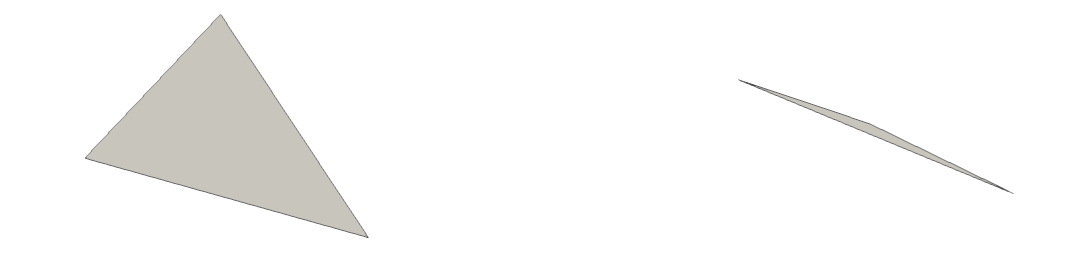}
    \caption{Cell which attains lowest value of scaled Jacobian mesh quality across all timesteps with the artificial-trained NN-corrected harmonic extension, before (left) and after (right) mesh motion. The left cell has scaled Jacobian mesh quality value of $0.755$ and the right cell has scaled Jacobian mesh quality value of $0.068$.}
    \label{fig:cellcomparison}
\end{figure}

When implemented as the mesh motion strategy in the FSI benchmark II simulation, the FSI-trained extension is successful and the simulation runs the whole 15 seconds period, as shown in Figure \ref{Figure_torch_corrected_fsi_trained}. Using the extension trained on the artificial dataset, the FSI simulation was also able to run for the whole 15 second period, as shown in Figure \ref{Figure_torch_corrected_artificial_trained}, but with lower mesh quality.

Figure \ref{Figure_artificial_data_hist_upper_right} shows a histogram of inputs at $\xi = (0.6, 0.21)$, the upper right corner of the flag, which is one of the points where we observed the mesh to degrade the most. The missing tails in the neural network inputs at this point from the artificial dataset indicate that there are characteristic deformations in the FSI simulation that are not available for the artificial dataset. 

The network was trained on a GPU and the process took approximately 16 minutes.
Evaluating the network to produce an extension correction for every vertex of the mesh in the FSI simulations takes approximately 10 milliseconds per time step. The simulations run to produce the results in Figures \ref{Figure_torch_corrected_fsi_trained}-\ref{Figure_torch_corrected_artificial_trained} consisted of $2233$ time steps, resulting in around 20 seconds spent on evaluating the neural network, although the real number is higher, since a portion of the time steps were rejected because the non-linear solver did not converge. In comparison, the total run time of the simulations was one order of magnitude higher than the training process.

\begin{figure}
    \includegraphics[width=0.48\textwidth]{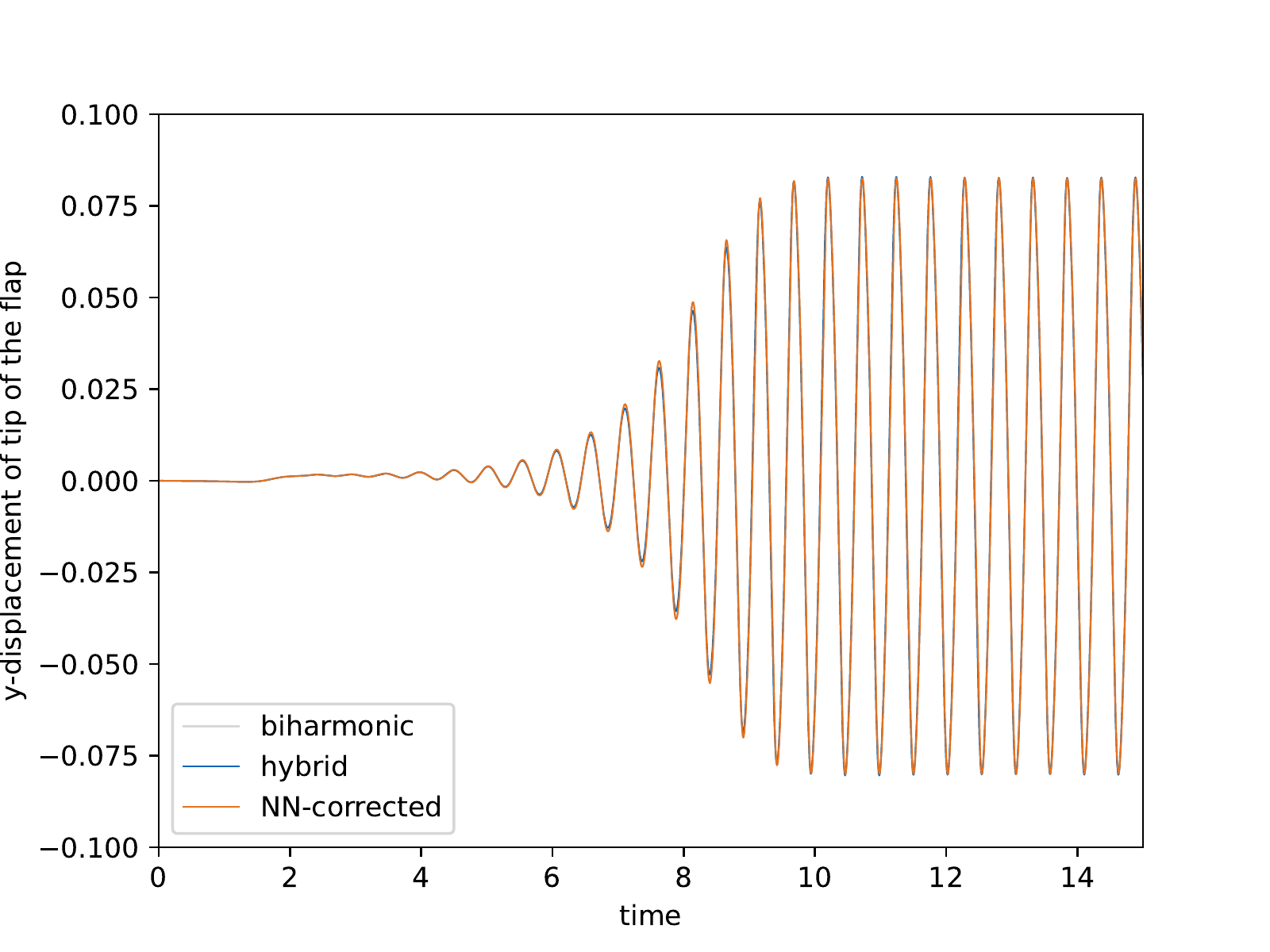}
    \includegraphics[width=0.48\textwidth]{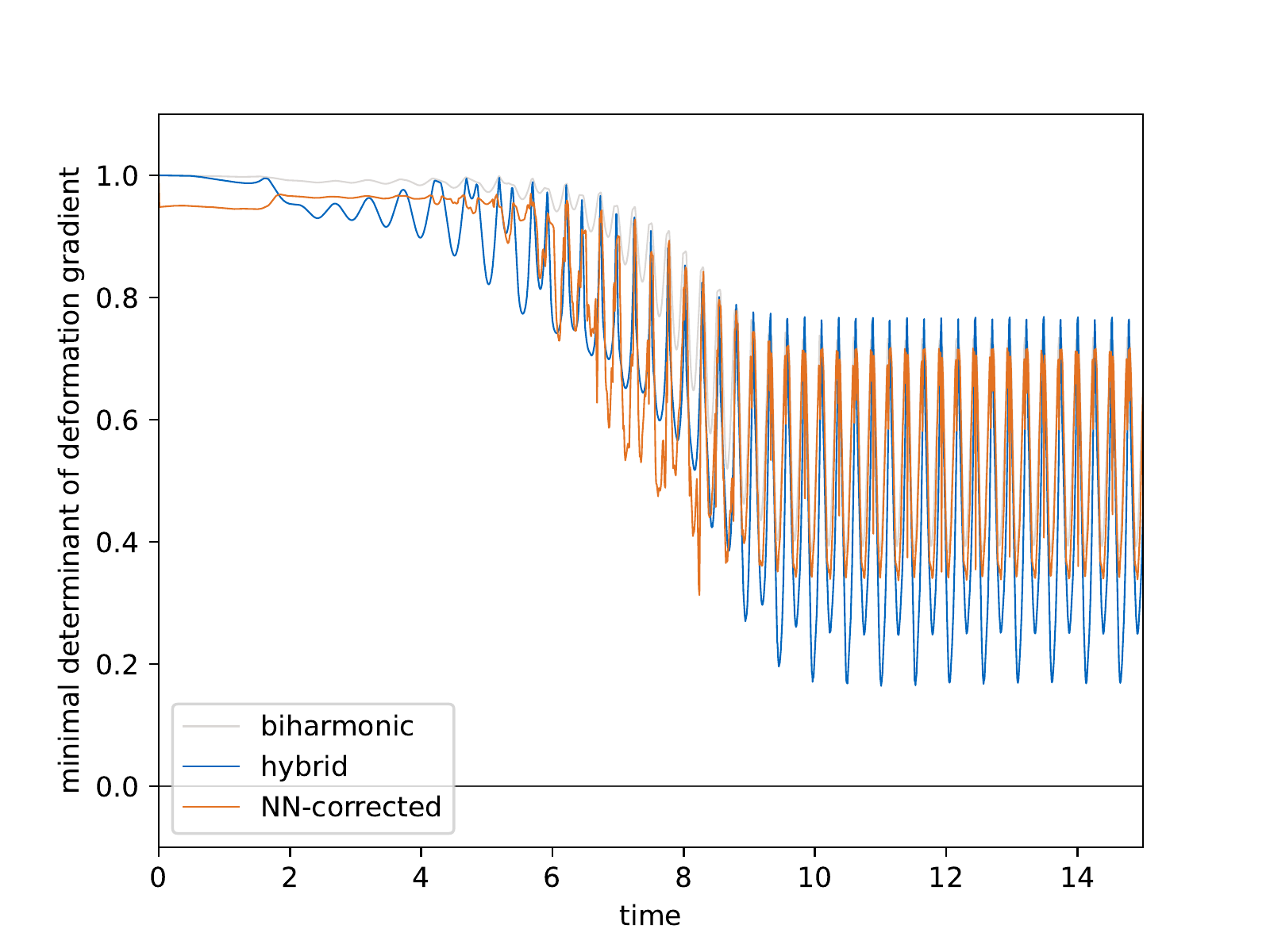}
     \vspace{-15pt}
    \caption{Numerical results for the NN-corrected approach compared to the biharmonic and hybrid PDE-NN approaches, trained on the FSI benchmark II dataset.}
    \label{Figure_torch_corrected_fsi_trained}
\end{figure}

\begin{figure}
    \includegraphics[width=0.48\textwidth]{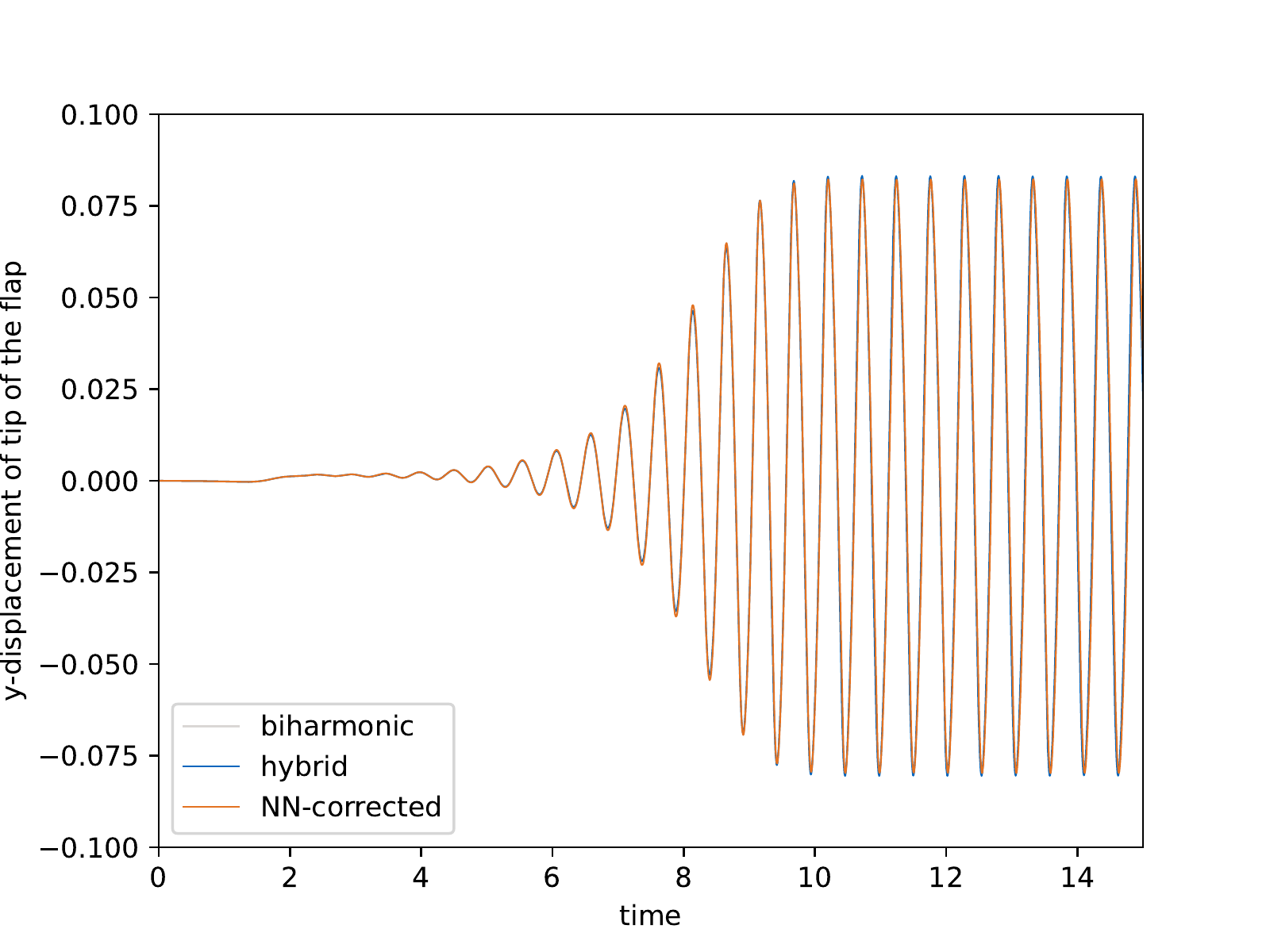}
    \includegraphics[width=0.48\textwidth]{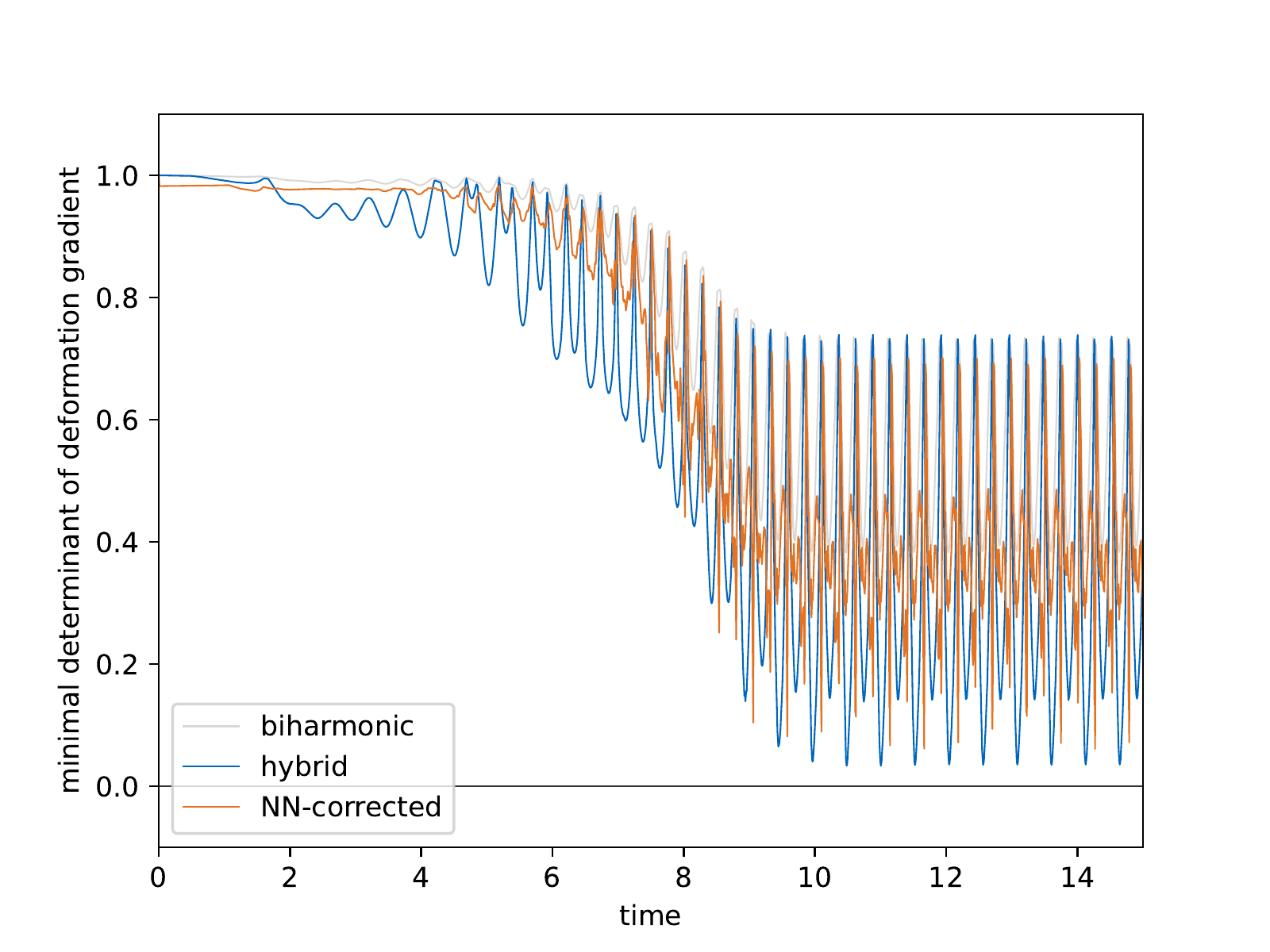}
      \vspace{-15pt}
    \caption{Numerical results for the NN-corrected approach compared to the biharmonic and hybrid PDE-NN approaches, trained on the artificial dataset.}
    \label{Figure_torch_corrected_artificial_trained}
\end{figure}

\begin{figure}
    \centering
    \includegraphics[width=0.85\textwidth]{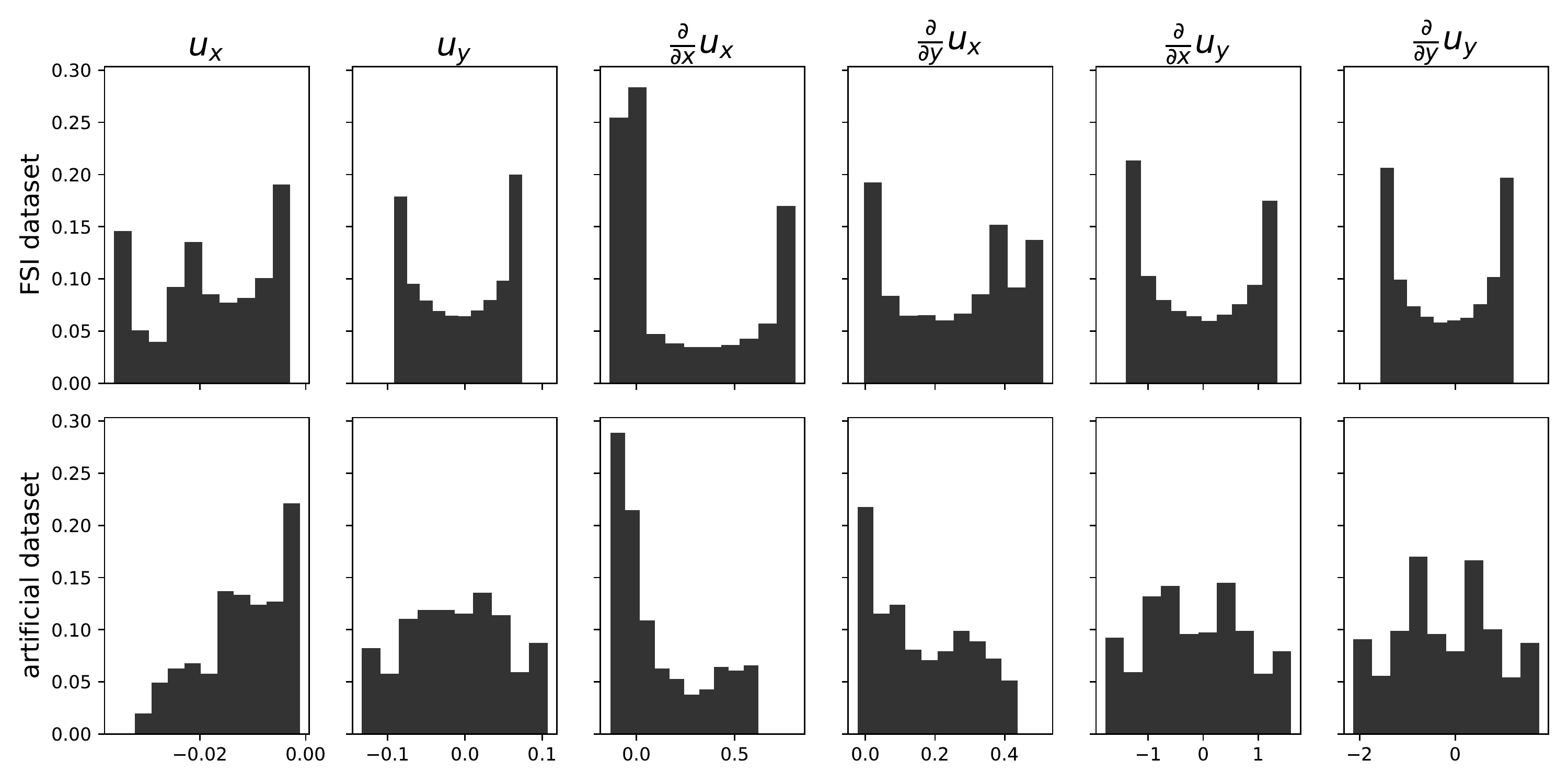}
    \vspace{-15pt}
    \caption{Histogram of neural network inputs at upper right corner of the flag in FSI and artificial datasets. Parts of the distributions for the FSI dataset are not represented in the artificial dataset, indicating that there are features of the deformations in the FSI simulation that are not included in the artificial dataset. }
    \label{Figure_artificial_data_hist_upper_right}
\end{figure}

\subsection{Parameter study}\label{sec:parameter_study}

We performed a parameter study for the neural network in the NN-corrected harmonic extension by analyzing the quality of meshes produced with varying random seeds and hyperparameter selections. The factors we varied are the number of hidden layers in the network and the choice of boundary condition-preserving function $l$ used.

To analyze the sensitivity of the network to random initialization, we trained 10 networks with equal hyperparameters as in section \ref{sec:corrected_harmonic_results}, using the random seeds 0-9. For the fully trained networks, we then computed the scaled Jacobian mesh quality indicator over the deformed fluid domains of the FSI test set. In Figure \ref{fig:random_paramstudy} we plot the median and quartiles over the 10 random initializations of the minimal cell value of scaled Jacobian mesh quality.
    
\begin{figure}
    \centering
    \includegraphics[width=0.98\linewidth]{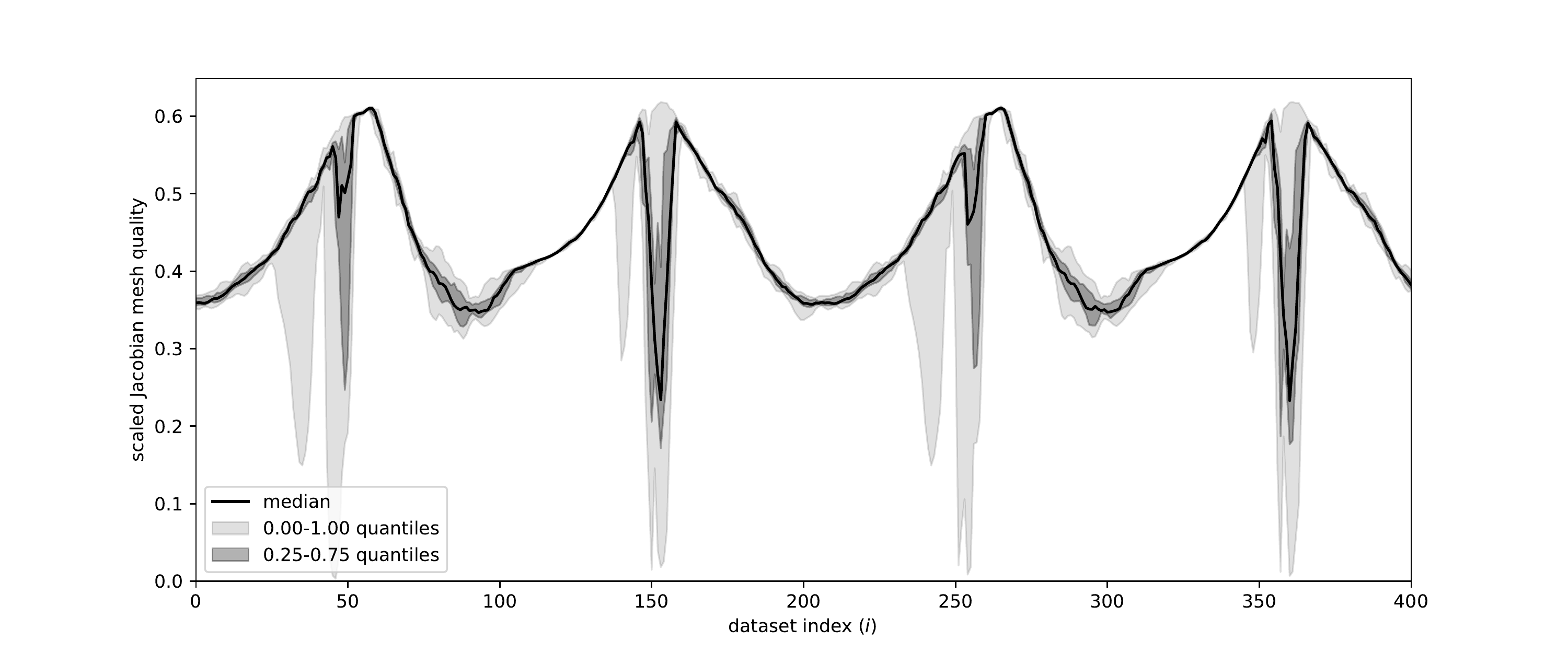}
    \caption{Quantiles of minimum over cells of scaled Jacobian mesh quality indicator over snapshots in test dataset for 10 trained networks with varying initializations and same hyperparameters as in FSI-study.}
    \label{fig:random_paramstudy}
\end{figure}

From Figure \ref{fig:random_paramstudy}, we can see that not all the trained networks produce quality mesh extensions. For some random initializations, the scaled Jacobian mesh quality approaches 0 at certain points of the periodic solid deformation. 

To judge the importance of the choice of boundary condition-preserving function $l$, we train 10 neural networks to correct the harmonic extension, but compute $l$ from \eqref{eq:mask_poisson} with $f \equiv 1$ 
instead of using the right hand side \eqref{eq:mask_rhs_hand_tuned}. Other than the choice of $l$, the network architecture and training parameters are the same as those used in section \ref{sec:corrected_harmonic_results}. We compute the scaled Jacobian mesh quality indicator's minimum over the mesh elements for the snapshots in the FSI test set and show the median and quantiles over the 10 random initializations in Figure \ref{fig:mask_paramstudy}.

\begin{figure}
    \centering
    \includegraphics[width=0.98\linewidth]{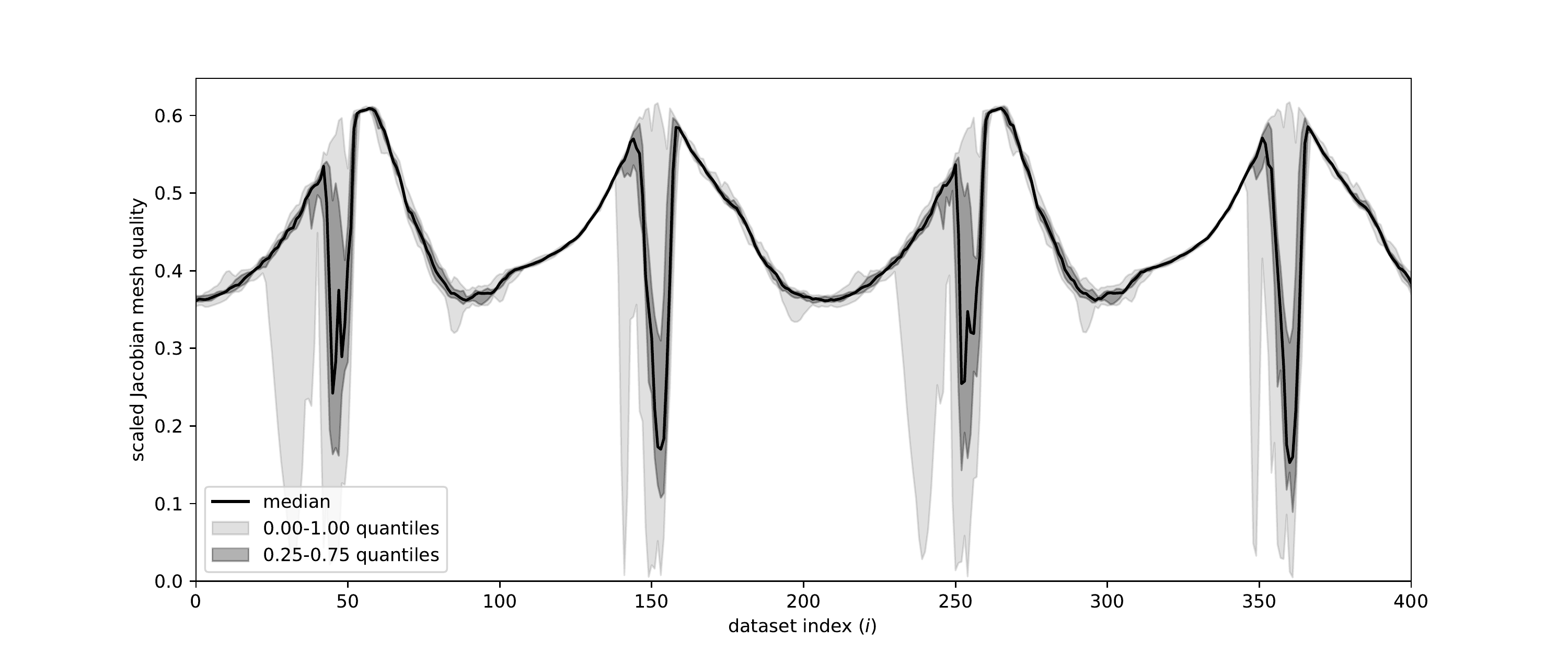}
    \caption{Quantiles of minimum over cells of scaled Jacobian mesh quality indicator over snapshots in test dataset for 10 trained networks with varying initializations and same hyperparameters as in FSI-study. The boundary condition-preserving function $l$ is computed from \eqref{eq:mask_poisson} with $f \equiv 1$. 
    }
    \label{fig:mask_paramstudy}
\end{figure}

The results shown in Figure \ref{fig:mask_paramstudy} are similar to those in Figure \ref{fig:random_paramstudy}, but with lower quality extensions in some areas. 
These results indicate that NN-corrected harmonic extension can be successful with $l$ that are not hand-tuned for the specific task at hand, but the use of hand-tuned $l$ can be beneficial.

To judge the impact of the number of hidden layers in the neural network correction, we analyzed 10 random initializations of networks with 2, 3, 4, and 5 hidden layers and similar complexity. In these experiments, the width of the network was always chosen such that the total number of parameters was as close as possible to the number of parameters in the architecture used in section \ref{sec:corrected_harmonic_results}. This was done by adjusting the width of the networks. The architectures used are summarized in table \ref{tab:param_study_param_count}.

\begin{table}[]
    \centering
    \begin{tabular}{c|c c c c c}
        Depth & 2 & 3 & 4 & 5 & 6 \\
        Width & 284 & 202 & 165 & 143 & 128 \\
        \# Parameters & 84066 & 84236 & 83987 & 83943 & 83970
    \end{tabular}
    \caption{Architectures and parameter counts used for analysis of impact of depth in neural network-correction of harmonic extension.}
    \label{tab:param_study_param_count}
\end{table}

Using the neural network architectures in table \ref{tab:param_study_param_count}, we run the training process from section \ref{sec:corrected_harmonic_results} with 10 different random initializations for each architecture. We compute the scaled Jacobian mesh quality measure's minimum value over all mesh elements over the FSI test set and report the median and quantiles over the random initializations in Figure \ref{fig:depth_paramstudy}.

\begin{figure}
    \centering
    \includegraphics[width=0.98\linewidth]{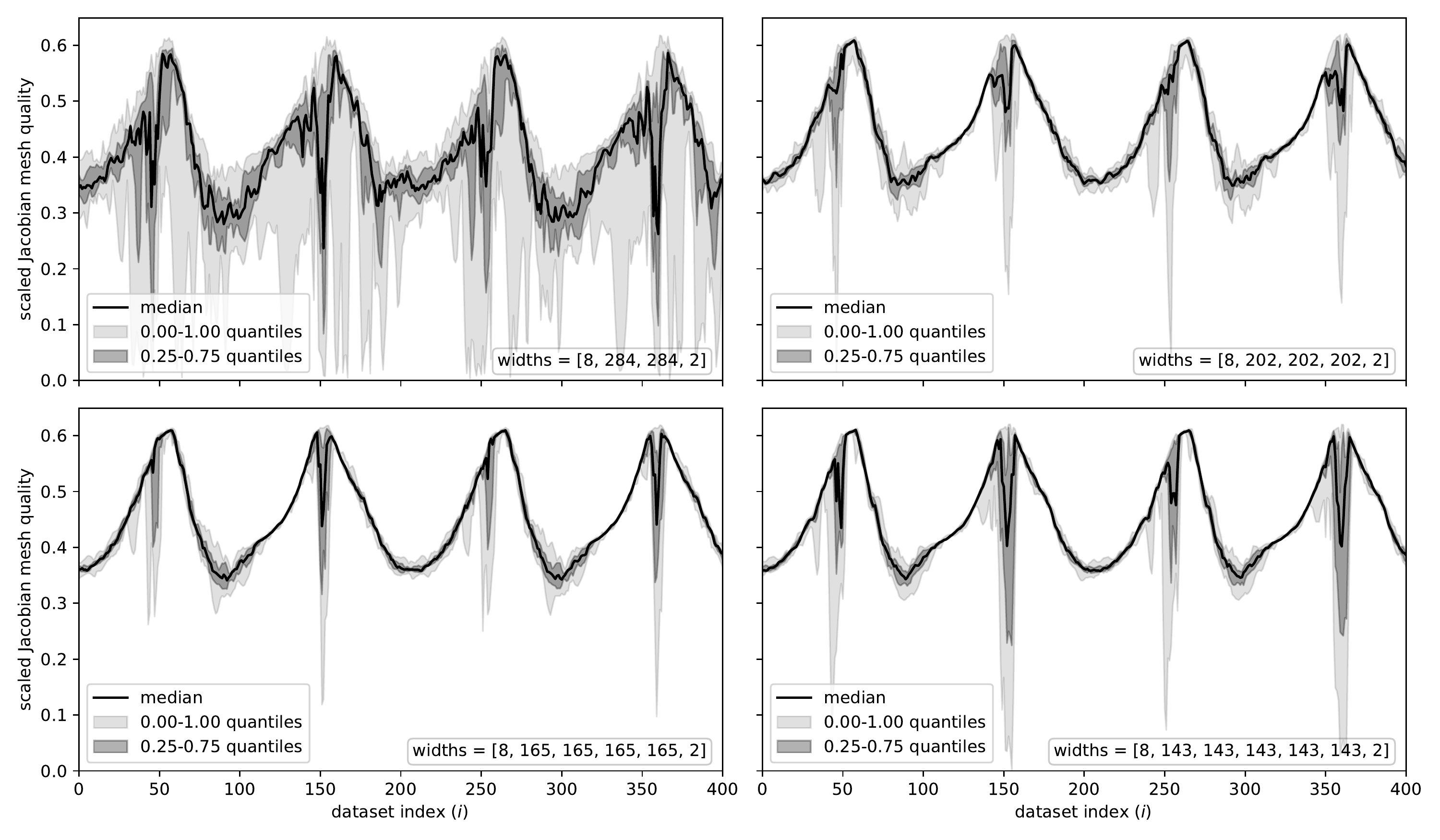}
    \caption{
    Quantiles of minimum over cells of scaled Jacobian mesh quality indicator over snapshots in test dataset for 10 trained networks with varying initializations for a varying number of hidden layers. Other than width and depth of neural networks, the same hyperparameters are used as in section \ref{sec:corrected_harmonic_results} and the total number of parameters is as close as possible to the number of parameters used in section \ref{sec:corrected_harmonic_results}.
    }
    \label{fig:depth_paramstudy}
\end{figure}

The results in Figure \ref{fig:depth_paramstudy} indicate that the NN-corrected harmonic extension requires a sufficiently deep network to learn a stable and high-quality mesh extension.

Judging only by the scaled Jacobian mesh quality indicator, a shallower network of for instance 3 hidden layers can produce satisfactory results. However, we found that deeper networks tend to produce smoother mesh motions and chose to use 6 hidden layers in our network. We show how the smoothness of the mesh motion increases with depth in Figure \ref{fig:shimmer_paramstudy}, where we plot the vertical displacement over the FSI benchmark problem II test set of two vertices located at the tip of the deformable solid. 

\begin{figure}
    \centering
    \includegraphics[width=0.9\linewidth]{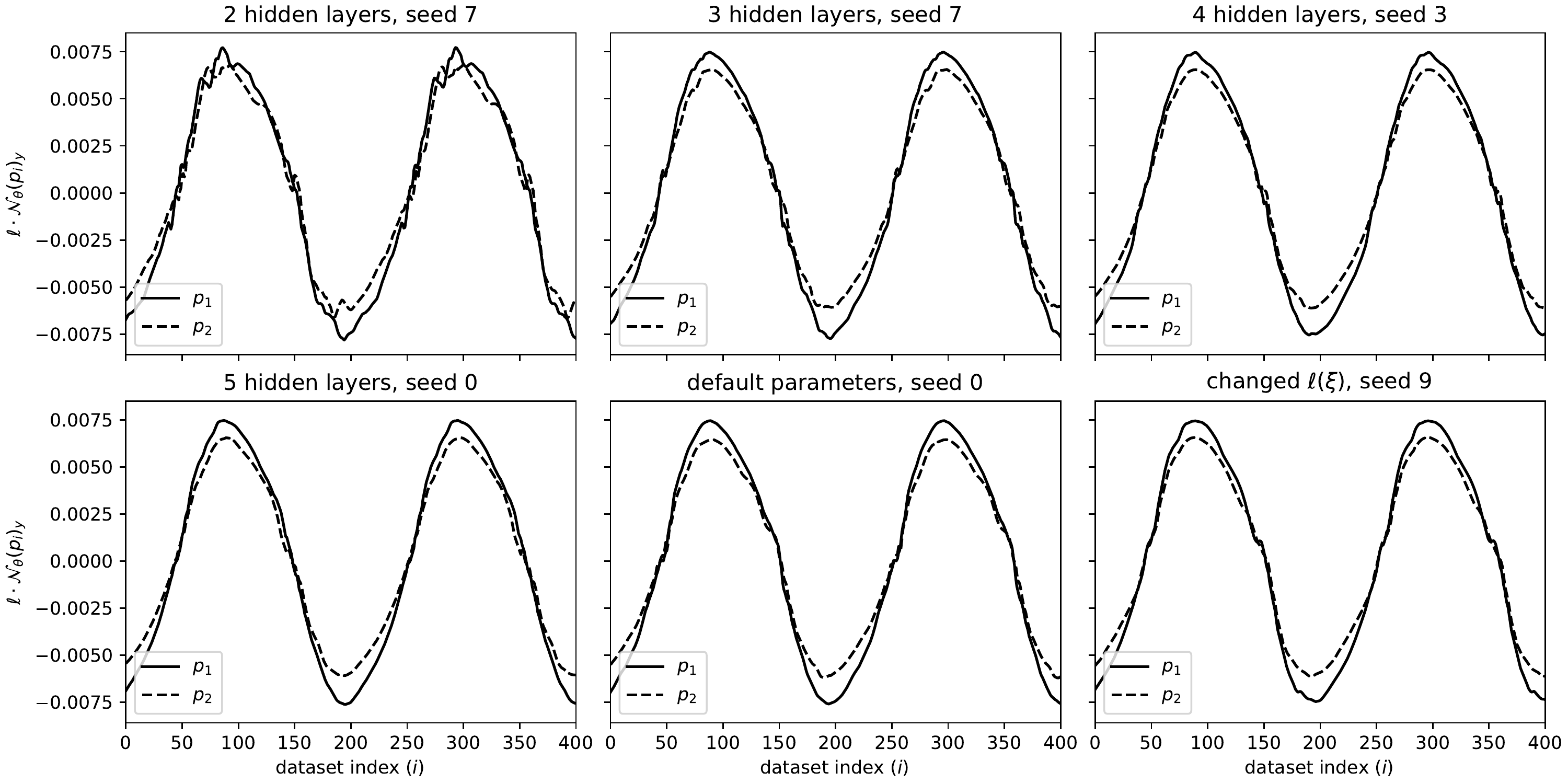}
    \caption{The $y$-component of the NN-correction $l \cdot \mathcal{N}_\theta$ evaluated at two selected mesh vertices over the FSI test set. The selected networks $\mathcal{N}_\theta$ are the best performing by visual inspection for 2, 3, 4, and 5 hidden layers, the default parameters of section \ref{sec:corrected_harmonic_results}, and with boundary condition-preserving $l$ chosen by $-\Delta l \equiv 1$. The selected vertices are the two non-boundary vertices connected to the vertex at the center of the tip of the deformable solid, at $p_1 = (0.60433, 0.2025)$ and $p_2 = (0.60363, 0.1975)$.}
    \label{fig:shimmer_paramstudy}
\end{figure}

\subsection{Generalizability}\label{sec:generalizability_comp_cost}

In order to assess the generalizability of the extension operators, we apply the extension operators to larger deformations (section \ref{subsec:gravity}) and for a different geometry (section \ref{subsec:membrane}). 

\subsubsection{Gravity-driven deformation test}
\label{subsec:gravity}

We analyze the quality of transformed meshes obtained by extending solid boundary deformations produced from an elasticity problem \cite[section 4]{Shamanskiy2020}. Using the geometry of the FSI benchmark problem II, we apply a uniform gravitational load $(0, f_\mathrm{grav})$ to the solid, with homogeneous Dirichlet boundary conditions on the cylinder boundary and zero-traction boundary conditions elsewhere. We then compute the transient displacement of the solid according to the St. Venant-Kirchoff model \eqref{eq:stvk_material}, with the same material constants as in the FSI benchmark problem II. The elasticity problem is discretized using implicit Euler time-stepping and second order Lagrange elements in space.

For three different values of $f_\mathrm{grav} \in \lbrace 1.0, 2.0, 2.5 \rbrace$, we simulate the system until the solid first reaches its maximum deformation in the $y$-direction. These maximal solid deformations are then extended to the fluid domain to evaluate the quality of the meshes produced by the different mesh motion techniques. The maximal value of $f_\mathrm{grav} = 2.5$ was chosen as this, for our set-up and discretization, results in deformations close to the limit of biharmonic mesh motion.

Similar to \cite{stein2004automatic, Wick, Shamanskiy2020}, loss of bijectivity of the ALE-mapping is identified by sampling $J = \mathrm{det}(I + \nabla F)$ at the degree of freedom locations of a sixth order discontinuous Galerkin space. If a negative value of $J$ is found in an element, we change the sign of the measured mesh quality indicator for that cell to be negative. Thus, if a deformed mesh contains negative mesh qualities in our experiments, the mesh is degenerate.

Figure \ref{fig:grav_test_histograms} shows histograms of the scaled Jacobian mesh quality indicator, with negative sign for degenerate cells, for the three test deformations with the biharmonic extension, the harmonic extension, the hybrid PDE-NN extension trained on FSI-data and artificial data, and the NN-corrected harmonic extension trained on FSI-data and artificial data.

\begin{figure}
    \centering
    \includegraphics[width=0.98\linewidth]{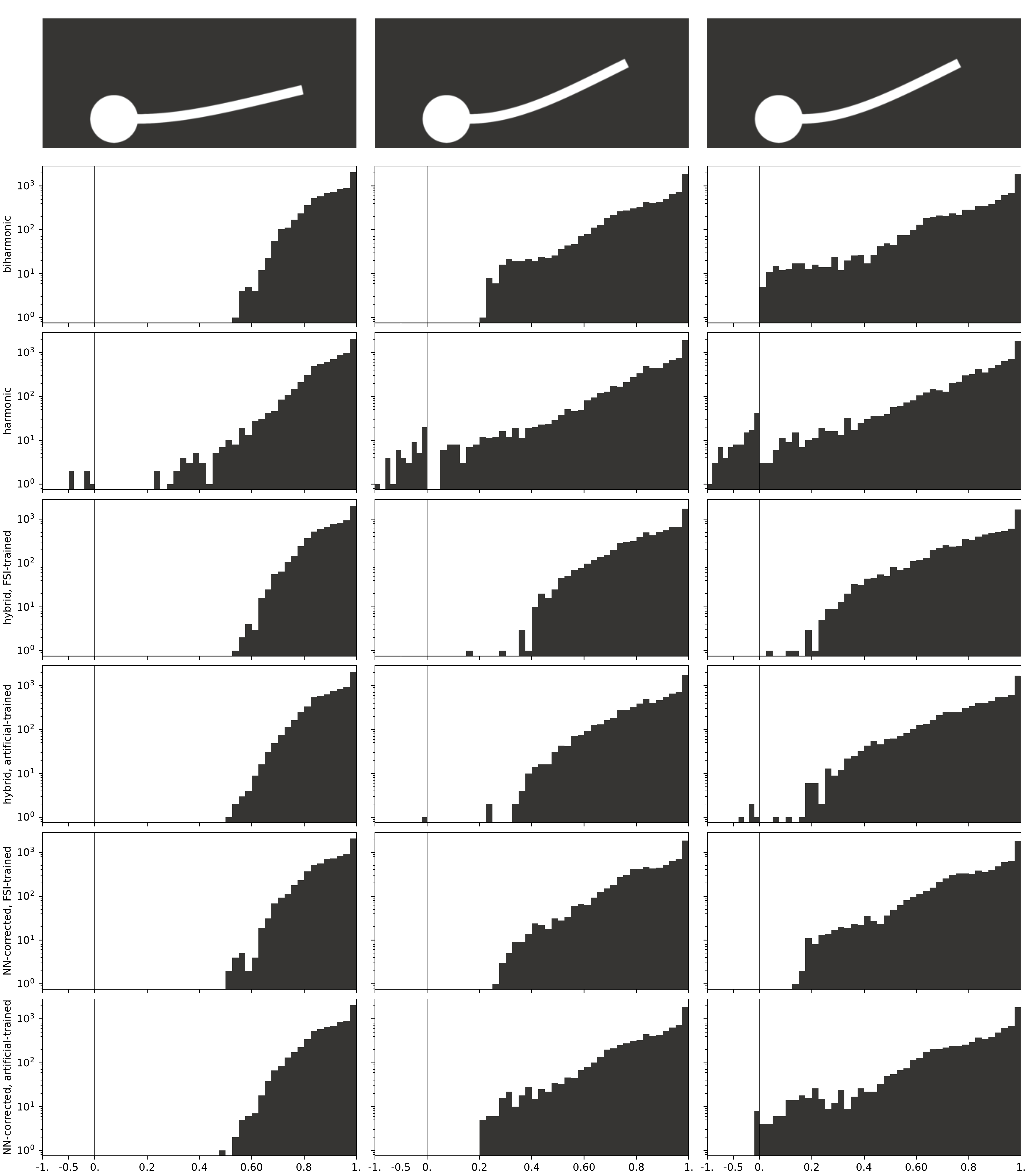}
    \caption{Mesh quality histograms for the biharmonic, harmonic, hybrid and NN-corrected extension operators on three different deformations of the gravity-driven deformation test (depicted in the first row)}
    \label{fig:grav_test_histograms}
\end{figure}

The setup of this test problem is similar to the FSI benchmark problem II and Figure \ref{fig:grav_test_histograms} shows that the learned extensions are able to produce meshes of comparable quality to the biharmonic extension. The extensions trained on FSI data perform better than the ones trained on the artificial dataset. The artificial-trained hybrid PDE-NN extension degenerates on the two largest deformations and the artificial trained NN-corrected harmonic extension degenerates on the largest deformation. The other learned extensions are non-degenerate in that the mesh quality is bounded away from zero and the determinant of the deformation gradient is positive.

\subsubsection{Membrane on fluid test}
\label{subsec:membrane}

In order to test the generalizability properties of the different extension operators (trained on the FSI benchmark II dataset) onto different geometries, we consider the membrane fluid test presented in \cite[Fig. 7, Sec. 4.2]{Wick}. We generate another data set by running an FSI simulation with biharmonic extension operator, IMR type material and parameters $\rho_f = 10^3$, $\nu_f = 4 \cdot 10^{-3}$, $\rho_s = 8 \cdot 10^2$, $\mu_s = 2 \cdot 10^7$ and $\mu_2 = 1 \cdot 10^5$. We apply the different extension operators to the three snapshots of the simulation depicted in the first line of Figure \ref{fig:membrane_test_histograms}. 

Applying the trained hybrid PDE-NN extension to the new geometry is straight forward, since it includes no specifics about the geometry it is used for. 
In contrast, for the neural network-corrected harmonic extension, one has to compute a new boundary condition-preserving function $l$. Switching $l$ between training and testing amounts to a change in network parameters between training and testing, with which it is not reasonable to expect the network to succeed. Also, since the position vector is an input to the network, the network is specific to the training geometry and transferability cannot be expected.

Figure \ref{fig:membrane_test_histograms} shows histograms of the scaled Jacobian mesh quality indicator, with negative sign for degenerate cells, for the three different snapshots of the membrane FSI problem with the biharmonic extension, the harmonic extension, and the hybrid PDE-NN extension trained on FSI-data and artificial data. The sign for degenerate cells is determined in the same way as for the gravity-driven motion test.

\begin{figure}
    \centering
    \includegraphics[width=0.98\linewidth]{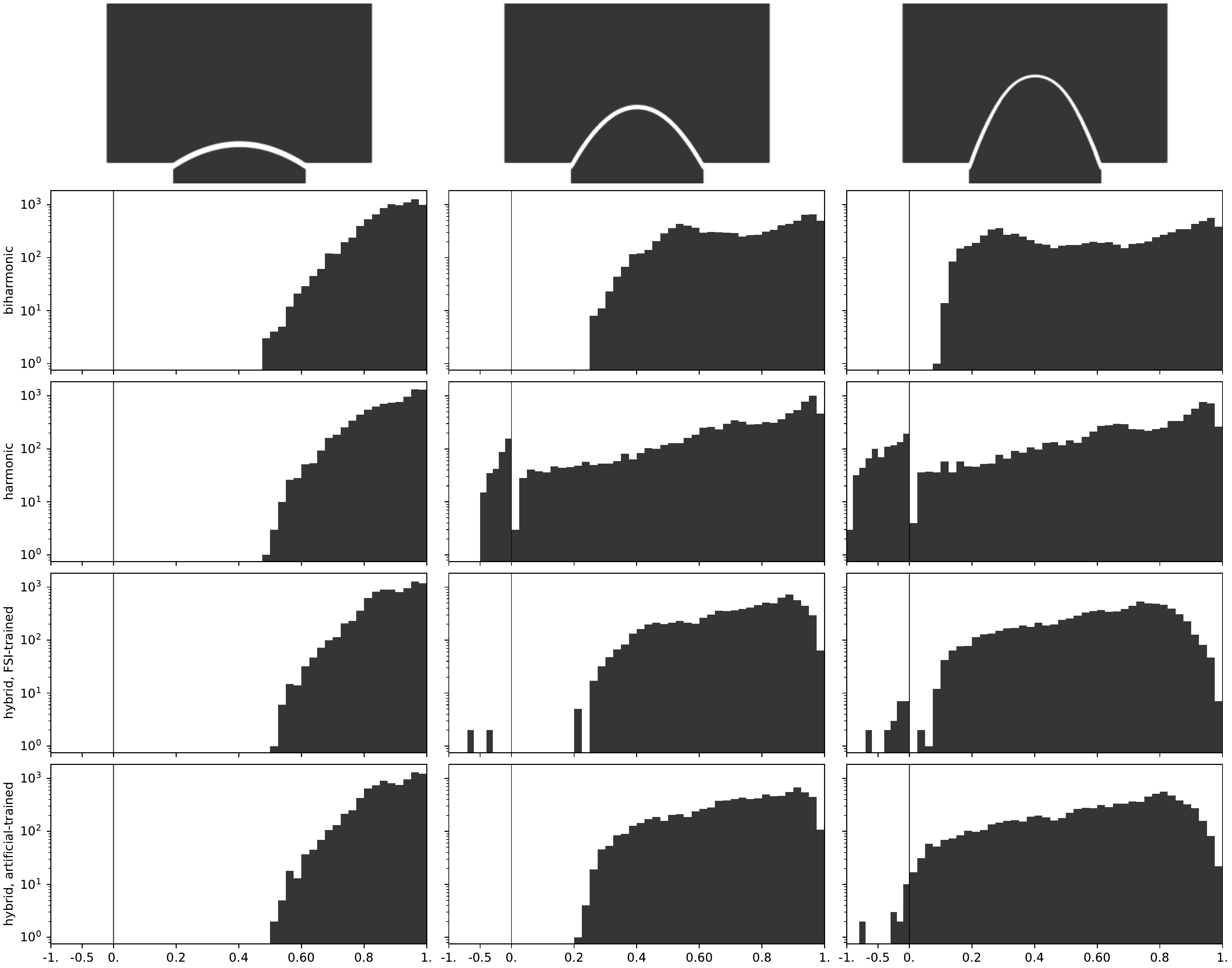}
    \caption{Mesh quality histograms for the biharmonic, harmonic, and hybrid PDE-NN extension operators on three different deformations of the membrane on fluid test (depicted in the first row)}
    \label{fig:membrane_test_histograms}
\end{figure}

From the results shown in Figure \ref{fig:membrane_test_histograms}, we can conclude that the hybrid PDE-NN extensions are to some extent transferable between geometries.

\subsection{Runtime comparison for FSI benchmark II}\label{sec:runtime_comparison}

\label{subsec:timings}

Figures \ref{fig:timings_plot}, \ref{fig:timings_plot_2} present the average computational time of the different extension operators on the first 40 snapshots of the FSI benchmark II data set split into assembly, linear solves, neural network correction and rest (e.g. mesh moving) for three different refinement levels (with 3935, 15300 and 60320 vertices). While the runtime for the NN correction just slightly increases the runtime of the harmonic extension, the hybrid PDE-NN variants increase the runtime significantly. Qualitatively, the plots for the incremental linearized variants are similar to the results presented in \cite[Fig. 7]{Shamanskiy2020}. The assembly and factorization of the matrices for the harmonic and biharmonic extension only needs to be done once while initializing the extension operators. For the (linearized versions of the) hybrid PDE-NN approach, the matrices change and the assembly needs to be done several times. The computational time for the training process is not reflected in these plots. The training takes approximately two hours and 18 minutes for the hybrid PDE-NN approach and approximately 16 minutes for the NN-corrected harmonic approach. The hybrid PDE-NN training can be significantly sped up if the PDEs for each function and gradient evaluation are computed in parallel, instead of sequentially solving for each training data point. 

\begin{figure}
\centering
    \includegraphics[width=0.45\linewidth]{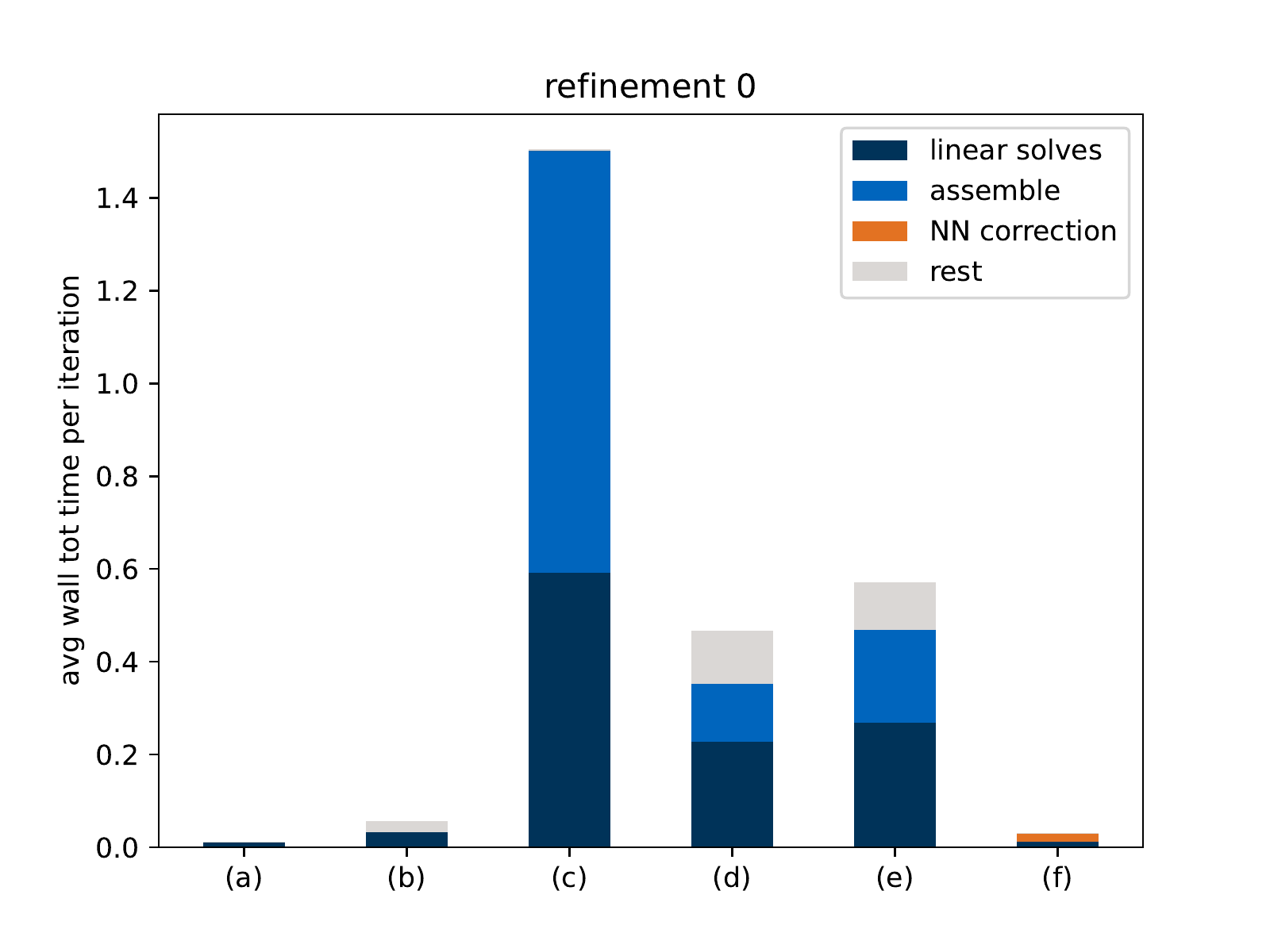} \hfill
    \includegraphics[width=0.45\linewidth]{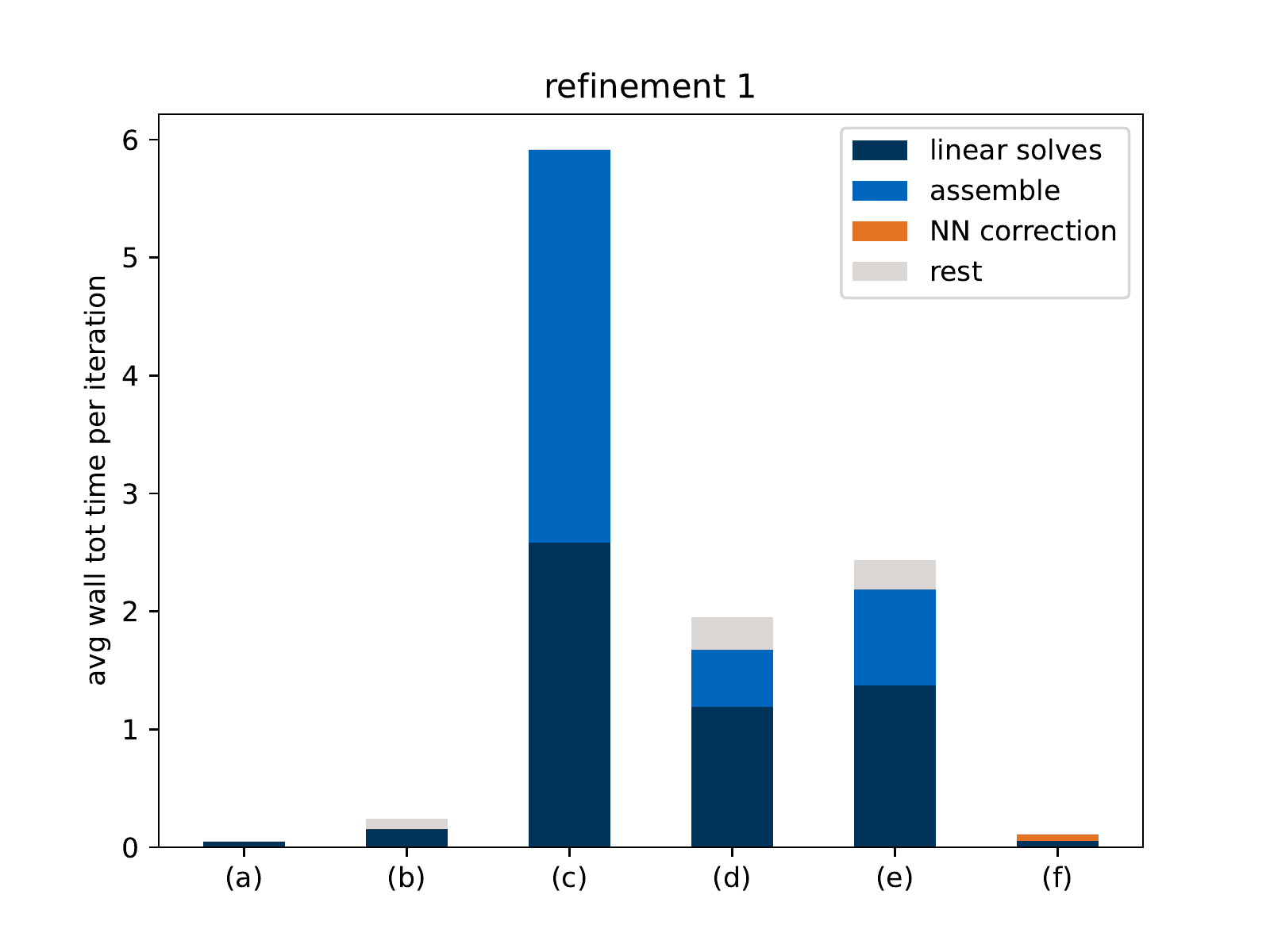} \hfill
    \includegraphics[width=0.45\linewidth]{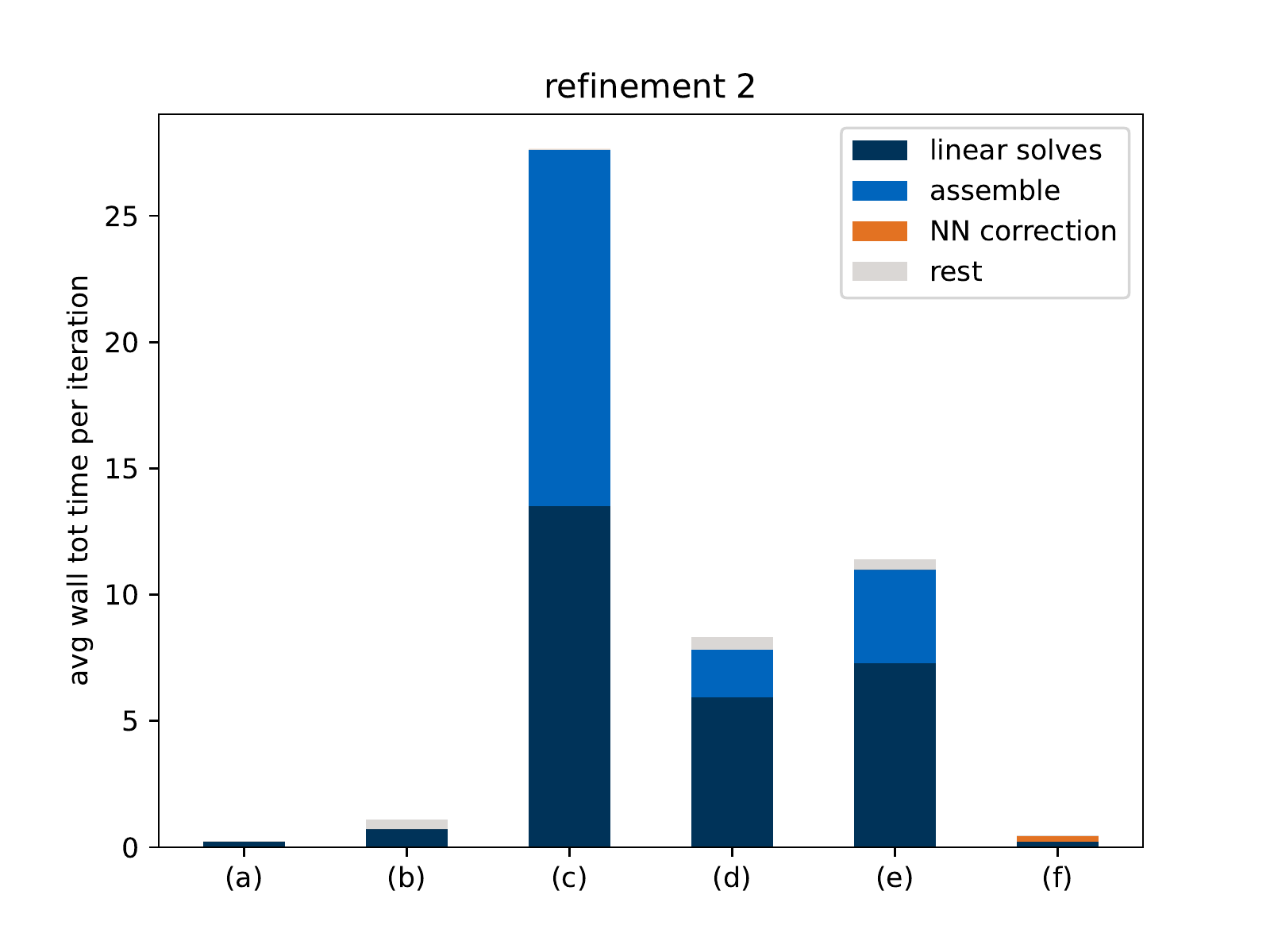} \hfill
    \caption{Timings for (a) harmonic extension, (b) biharmonic extension, (c) learned extension, (d) incremental learned extension, (e) incremental corrected learned extension, (f) NN-corrected extension on three different refinement levels. }
    \label{fig:timings_plot}
\end{figure}

\begin{figure}
\centering
    \includegraphics[width=0.45\linewidth]{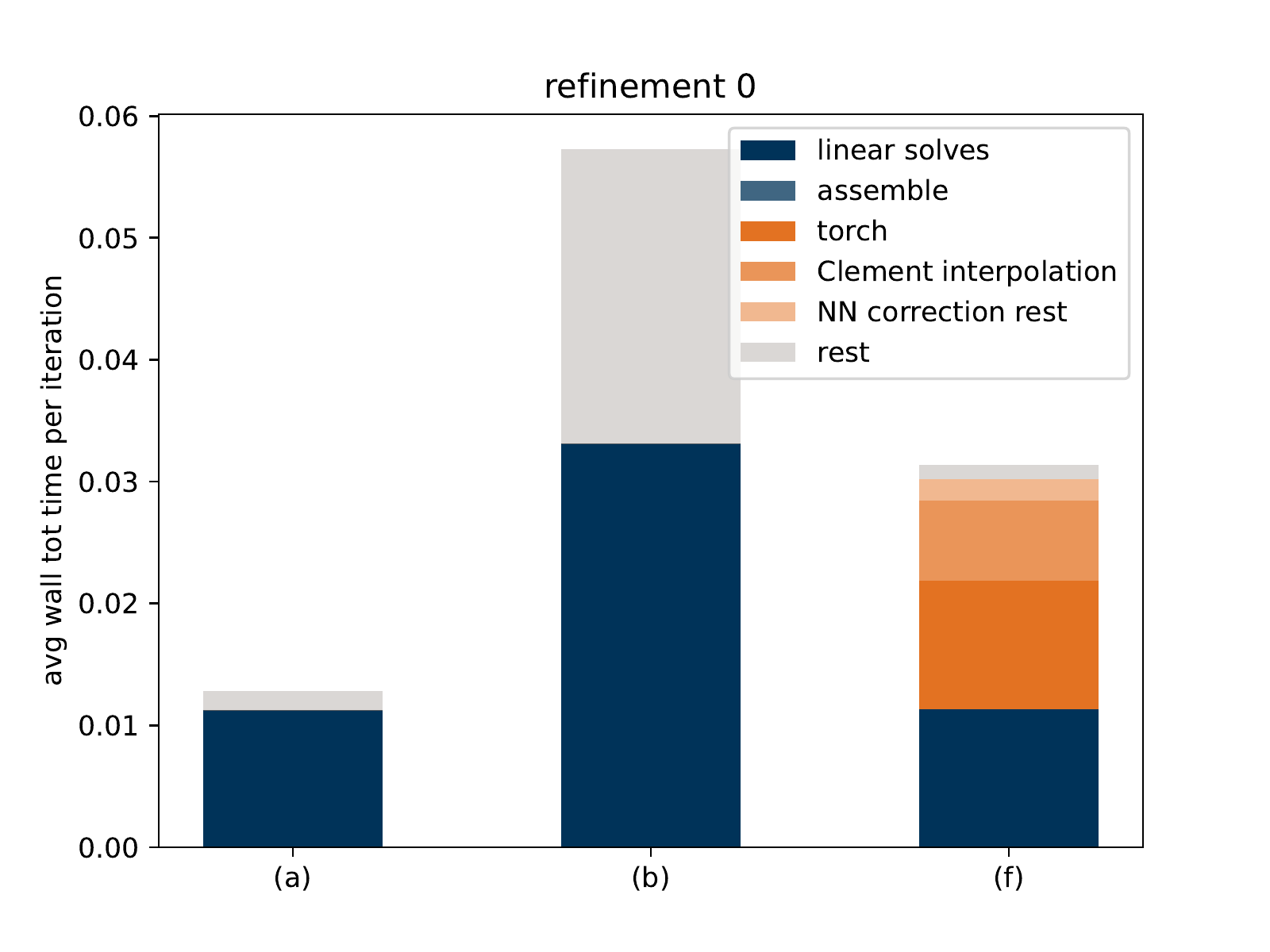} \hfill
    \includegraphics[width=0.45\linewidth]{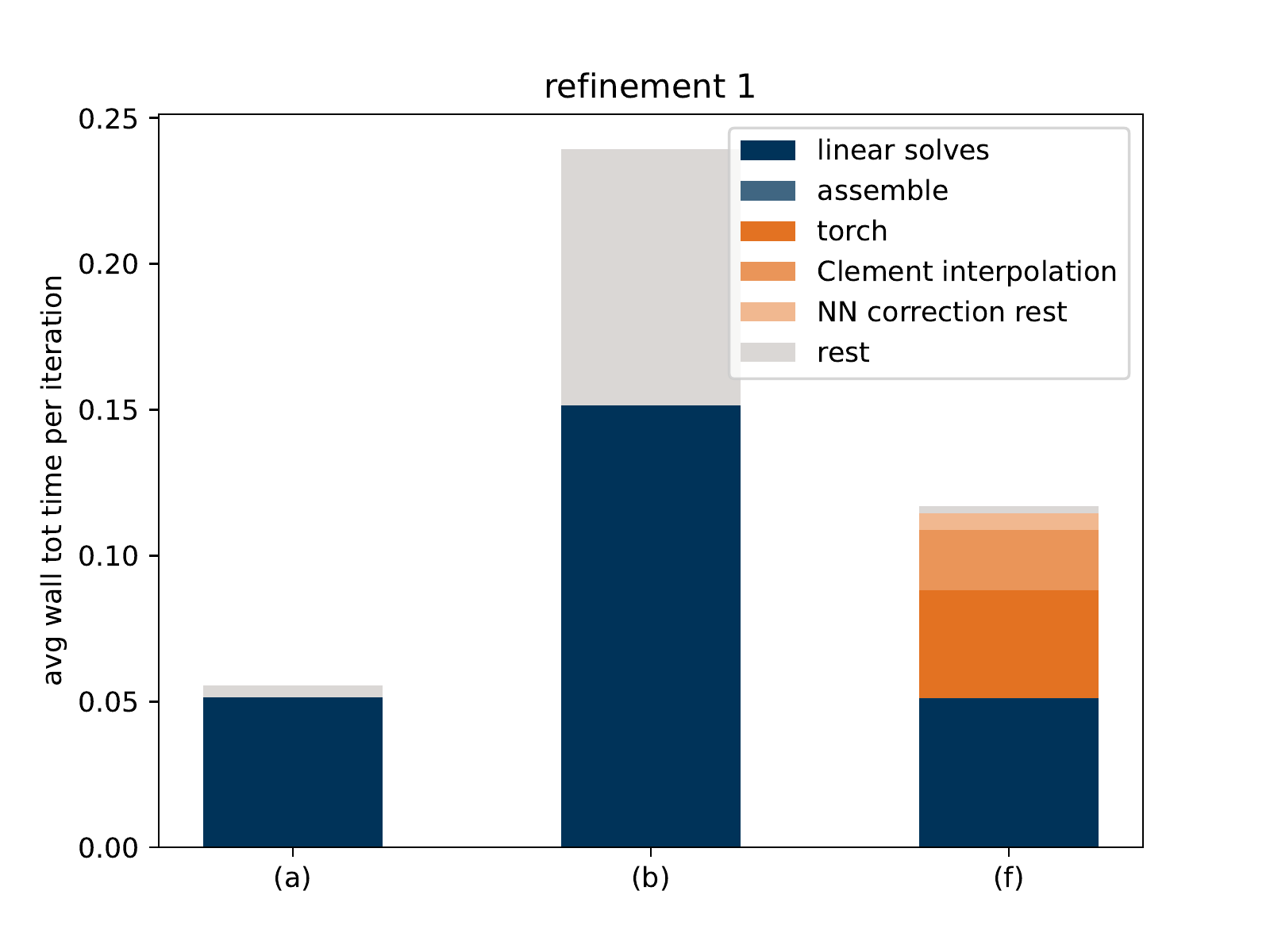} \hfill
    \includegraphics[width=0.45\linewidth]{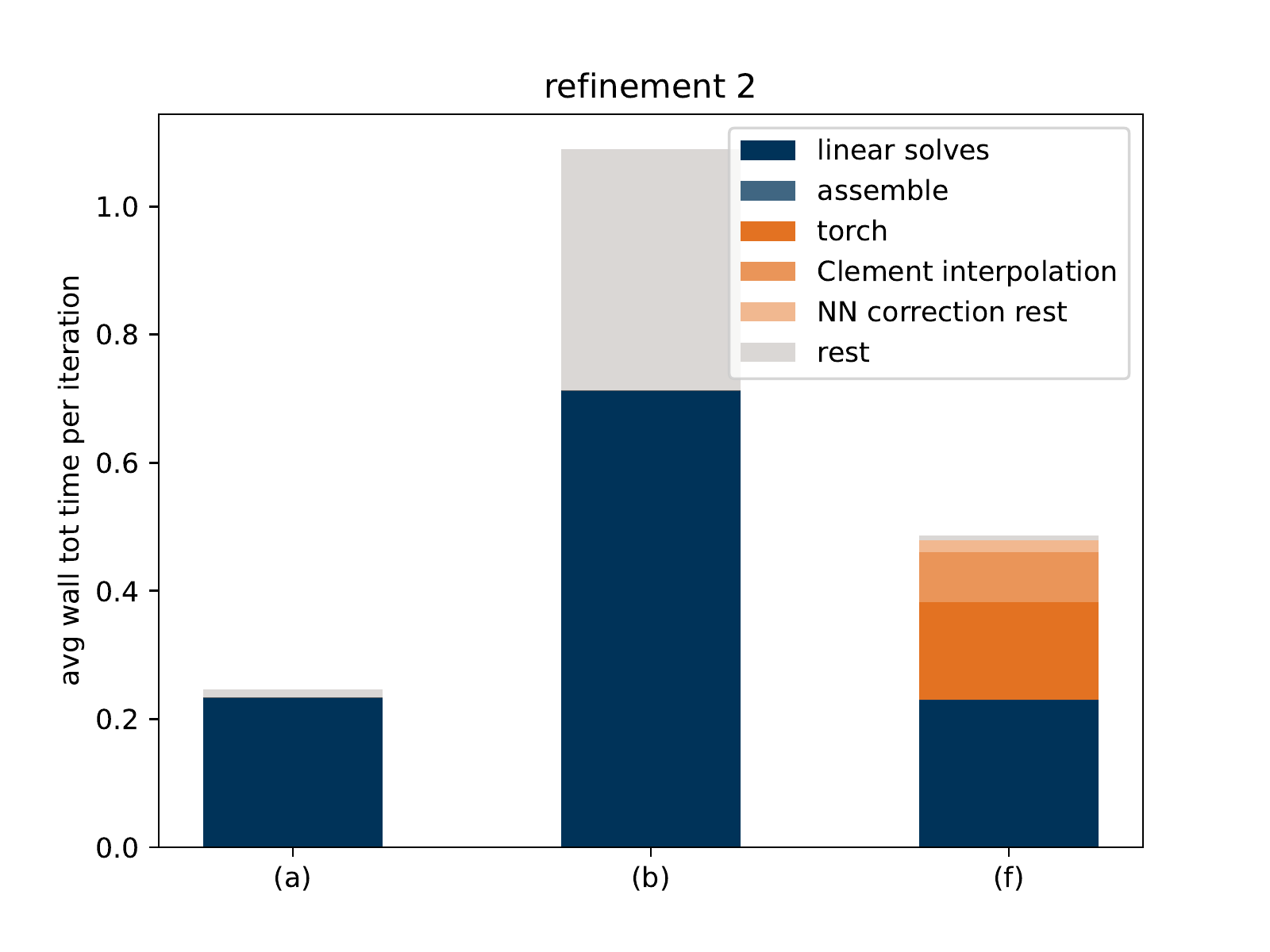} \hfill
    \caption{Timings for (a) harmonic extension, (b) biharmonic extension, (f) NN-corrected extension and three different refinement levels (zoom of Figure \ref{fig:timings_plot}).}
    \label{fig:timings_plot_2}
\end{figure}

\section{Conclusion}

We presented two approaches to learn an operator that extends boundary deformation to the interior of the domain. We chose a supervised learning approach such that the operators approximate the solution of the biharmonic equation. We demonstrated that the learned operators have the potential to serve as mesh motion technique for FSI simulations. So far, we considered supervised learning approaches. For future work, e.g. in cases where the biharmonic extension does not give satisfactory results, it is also interesting to consider unsupervised approaches, where one tries to optimize a measure for the quality of the deformed mesh. Moreover, we worked with approaches that require the solution of a partial differential equation. Considering approaches that do not require PDE solves is another interesting direction for future research.

\bibliographystyle{plain}
\bibliography{bib}

\appendix

\section{Counterexample}
\label{appendix}

We consider a shallow input convex neural network of the form 
\begin{align}
    f_\theta(x) = \sum_{i=1}^N c_i \rho(\tilde w_i \cdot x + \tilde b_i),
\end{align}
where $c_i\geq 0$. If we work with $\rho = ReLU$, $f_\theta$ can be defined as
\begin{align}
    f_\theta(x) = \sum_{i=1}^N \rho(w_i \cdot x + b_i),
    \label{func_form}
\end{align}
where $w_i = c_i \tilde w_i$ and $b_i = c_i \tilde b_i$. 
While every convex function $h: \mathbb R \to \mathbb R$ that is bounded from below can be approximated by a function of the form \eqref{func_form}, this is not the case in general, e.g., not every convex function $h: \mathbb R^2 \to \mathbb R$ that is bounded from below can be approximated. 

In the following, we prove that we can not approximate the function
\begin{align}
    h: \, (x, y) \mapsto \max(\max(x + y, 0), \max(x -y , 0))
    \label{def_h}
\end{align}
with a function of the form \eqref{func_form} (even for arbitrary large $N$). We start by showing that the function $h$ can not be represented exactly by \eqref{func_form}. To do so, we use the following lemma. 

\begin{lemma}
Let $d \in \mathbb N$, $N \in \mathbb N$, $\lbrace a_i, b_i \rbrace_{i \in \lbrace 1, \ldots, N \rbrace}$ be such that $a_i \in \mathbb R^d$, $b_i \in \mathbb R$ for all $i \in \lbrace 1, \ldots, N\rbrace$. Let $f(x) := \sum_{i=1}^N \rho(a_i^\top x + b_i)$ with $\rho(s) = ReLU(s)$. Let $x_1, x_2 \in \mathbb R^d$ be such that $f$ is differentiable at $x_1$ and $x_2$, and $\mathcal A(x_1) \cap \mathcal A(x_2) = \emptyset$, where $$\mathcal A (x) = \lbrace j \in \lbrace 1, \ldots, N \rbrace ~: ~ a_j^\top x + b_j > 0 \rbrace.$$ Then, \begin{align}
    \tilde f(x):= \sum_{j=1}^2 \rho( f^\prime(x_j)^\top x + \sum_{i\in \mathcal A(x_j)} b_i)
    \label{definition_A}
\end{align}
fulfills 
\begin{align}
   f(x_j) = \tilde f(x_j)\text{ for } j \in \lbrace 1, 2 \rbrace \text{ and } \tilde f(x) \leq f(x)
   \label{fapprox}
\end{align}
for all $x \in \mathbb R^d$.
\label{lemma_1}
\end{lemma}

\begin{proof}
Since $f$ is assumed to be differentiable at $x_j$, $j \in \lbrace 1, 2 \rbrace$, \begin{align}
f^\prime(x_j) = \sum_{i \in \mathcal A(x_j)} a_i.
\label{f_derivative}
\end{align}
For $k \in \lbrace 1, 2 \rbrace$, since $\mathcal A(x_1) \cap \mathcal A(x_2) = \emptyset$, we know that
\begin{align*}
    \tilde f(x_k) &= \sum_{j=1}^2 \rho(\sum_{i \in \mathcal A(x_j)} a_i^\top x_k + \sum_{i\in \mathcal A(x_j)} b_i ) = \rho(\sum_{i \in \mathcal A(x_k)} (a_i^\top x_k + b_i) ) = \sum_{i \in \mathcal A(x_k)} (a_i^\top x_k + b_i) = f(x_k).
\end{align*}
Moreover, since $\rho(a + b) \leq \rho(a) + \rho(b)$ and $\mathcal A(x_1) \cap \mathcal A(x_2) = \emptyset$,
\begin{align*}
    \tilde f(x) & = \sum_{j=1}^2 \rho(\sum_{i \in \mathcal A(x_j)} (a_i^\top x + b_i) ) \leq \sum_{j=1}^2 \sum_{i \in \mathcal A(x_j)}  \rho(a_i^\top x + b_i) \leq \sum_{i=1}^N \rho(a_i^\top x + b_i) = f(x) 
\end{align*}
for all $x \in \mathbb R^d$. 
\end{proof}

\subsection{Exact representation}

We assume that there exists $N \in \mathbb N$ and $\lbrace a_i, b_i \rbrace_{i\in \lbrace 1, \ldots, N \rbrace}$ such that 
$f_\theta (x) = h(x)$ for all $x \in \mathbb R^d$. We choose $x_1 = (-\frac{k}2, \frac{3k}4)^\top$ and $x_2 = (-\frac{k}2, -\frac{3k}4)^\top$ for $k>0$ sufficiently large. Since $h(x) = 0$ for $x \in (-\infty, 0) \times \lbrace 0 \rbrace$, and $h(x) = 0$ for $x \in \lbrace -\frac{k}2 \rbrace \times (-\frac{k}2, \frac{k}2)$, and 
$f_\theta = h$, $\mathcal A(x_1) \cap A(x_2) = \emptyset$. Moreover, differentiability of $h$ at $x_1$ and $x_2$ gives 

\begin{align*}
 & f_\theta^\prime (x_1) = (1, 1)^\top, \quad f_\theta^\prime(x_2) = (1, -1)^\top. 
\end{align*}
By lemma \ref{lemma_1},
\begin{align}
f_\theta(x) \geq \rho((1, 1)^\top x + \sum_{i \in \mathcal A(x_1)} b_i) 
+ \rho((1, -1)^\top x + \sum_{i \in \mathcal A(x_2)} b_i).
\label{estimate}
\end{align}
For $x = x_j$, $j \in \lbrace 1,2\rbrace$, we obtain
$f_\theta(x_j) = h(x_j) = \frac{k}4$.
Therefore, \eqref{estimate} gives for $x \in \lbrace x_1, x_2 \rbrace$,
\begin{align*}
    \frac{k}4 = \frac{k}4 + \sum_{i \in \mathcal A(x_j)} b_i,
\end{align*}
(where we used \eqref{fapprox}, the definition of $\mathcal A$ and $\mathcal A(x_1) \cap \mathcal A(x_2) = \emptyset$,) and 
\begin{align*}
    \sum_{i \in \mathcal A(x_j)} b_i = 0.
\end{align*}
Thus, 
\begin{align*}
    f_\theta (x) \geq \rho((1, 1)^\top x) + \rho((1, -1)^\top x ).
\end{align*}
Hence, for $x = (\alpha, 0)^\top$, $\alpha > 0$, we obtain
$f_\theta((\alpha, 0)^\top) \geq 2 \alpha$. However, $h((\alpha, 0)^\top) = \alpha$. This yields a contradiction and implies that $h(x)$ can not be represented by a function of the form \eqref{func_form}. 

\begin{figure}
\includegraphics[width=0.45\textwidth]{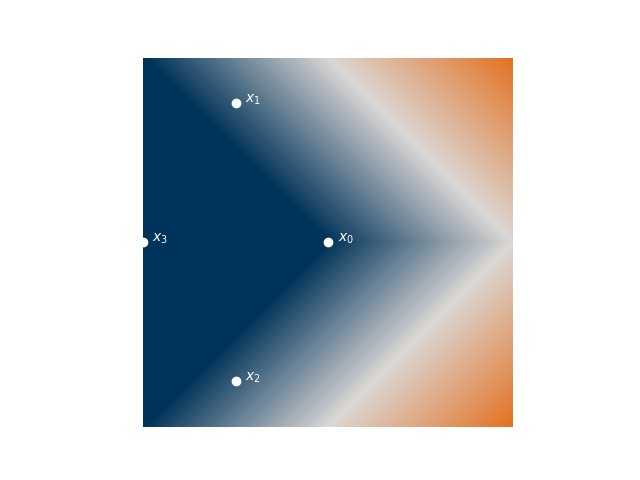}
\includegraphics[width=0.45\textwidth]{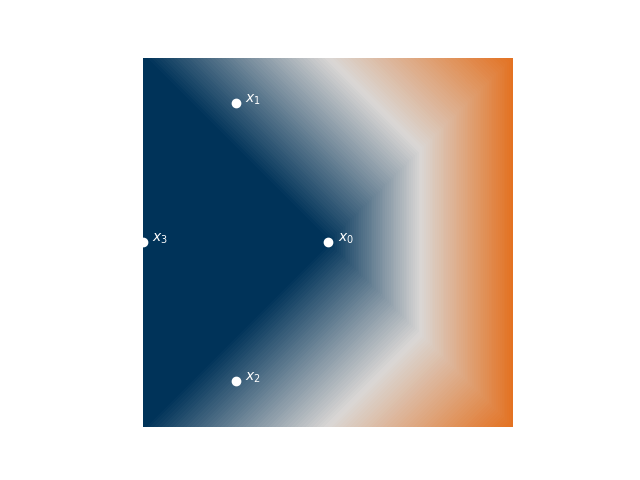}
  \vspace{-20pt}
\caption{Contour plot of the function $h$ (left) and a function $f_\theta$ of the form \eqref{func_form} (right). This function is a lower bound for every function of the form \eqref{func_form} that is equal to $h$ in $[-k,0] \times (-k, k)$.}
\end{figure}

\subsection{Approximation}

In the same fashion, we can also show that on the compact set $K := [- k, k] \times [- k , k]$, $k > 0$, there exists $\epsilon > 0$ such that $\sup_{x \in K} | f_{\theta}(x) - h(x) | > \epsilon$ for all $f_\theta$ of the form \eqref{func_form}. 

We show it by contradiction.
We assume that for every $\epsilon > 0$ there exists a function $f_{\theta, \epsilon} $ of the form \eqref{func_form} that fulfills 
\begin{align}
    \sup_{x \in K} | f_{\theta, \epsilon}(x) - h(x) | \leq \epsilon
    \label{approx_property}
\end{align}

\subsubsection{Assumption on \texorpdfstring{$\mathcal A(x_1) \cap \mathcal A(x_2)$}{intersection of active sets}}
\label{subsec_disjoint}

We show that we can assume without loss of generality that $f_{\theta, \epsilon}$ is such that $\mathcal A(x_1) \cap \mathcal A(x_2) = \emptyset$.

Since $h(x) = 0$ for $x \in (-k, 0) \times \lbrace 0 \rbrace $ and $x \in \lbrace - \frac{k}2 \rbrace \times (- \frac{k}2, \frac{k}2)$, and \eqref{approx_property},
\begin{align*}
0 \leq \sum_{i \in \mathcal A(x_1) \cap \mathcal A(x_2)} \rho(a_i^\top x + b_i) \leq \epsilon 
\end{align*}
for all $x \in (- k, 0) \times \lbrace 0 \rbrace \cup \lbrace - \frac{k}2 \rbrace \times (- \frac{k}2, \frac{k}2)$. Let $x_0:= (0, 0)^\top$ and $x_3 := (-k, 0)^\top$. We know (due to linearity of $a_i^\top x + b_i$) that $\mathcal A(x_1) \cap \mathcal A(x_2) =  (\mathcal A(x_1) \cap \mathcal A(x_2) \cap \mathcal A(x_0)) \cup ( \mathcal A(x_1) \cap \mathcal A(x_2) \cap \mathcal A(x_3))\setminus (\mathcal A(x_1) \cap \mathcal A(x_2) \cap \mathcal A(x_0))$. Hence, 
\begin{align}
\begin{split}
    \sum_{i \in \mathcal A(x_1) \cap \mathcal A (x_2)} \rho(a_i^\top x + b_i) & \leq \sum_{i \in \mathcal A(x_1) \cap \mathcal A (x_2) \cap \mathcal A(x_0)} \rho(a_i^\top x + b_i)\\ & + \sum_{i \in \mathcal A(x_1) \cap \mathcal A (x_2) \cap \mathcal A(x_3)\setminus(\mathcal A(x_1) \cap \mathcal A(x_2) \cap \mathcal A(x_0))} \rho(a_i^\top x + b_i).
\end{split}
\label{split_eq}
\end{align}
We define
\begin{align*}
    &f_{1} (x) :=  \sum_{i \in \mathcal A(x_1) \cap \mathcal A (x_2) \cap \mathcal A(x_0)} \rho(a_i^\top x + b_i), \\
    &f_{2} (x) :=  \sum_{i \in (\mathcal A(x_1) \cap \mathcal A (x_2) \cap \mathcal A(x_3))\setminus(\mathcal A(x_1) \cap \mathcal A(x_2) \cap \mathcal A(x_0))} \rho(a_i^\top x + b_i). 
\end{align*}
We exemplary do the discussion for $f_1$. 
Due to $ f_j \geq 0$, $j \in \lbrace 1, 2 \rbrace$, \eqref{split_eq} and by assumption \eqref{approx_property} we know that
\begin{align*}
    \epsilon \geq |f_1((-\frac{k}2, 0)^\top) - f_1(x_0)| = | \sum_{i \in \mathcal A(x_1) \cap \mathcal A (x_2) \cap \mathcal A(x_0) } a_i^\top (- \frac{k}2, 0)^\top | 
\end{align*}
and, thus, 
\begin{align}
    |\sum_{i \in \mathcal A(x_1) \cap \mathcal A (x_2) \cap \mathcal A(x_0) } (a_i)_1| \leq \frac{2 \epsilon}{k}.
    \label{estim1}
\end{align}
Moreover, due to the properties of $h$ on $\lbrace -\frac{k}2 \rbrace \times (-\frac{k}2, \frac{k}2)$ we know that
\begin{align*}
    \epsilon \geq |f_1((- \frac{k}2, 0)^\top) - f_1((- \frac{k}2, \pm \frac{k}2)) | = | \sum_{i \in \mathcal A(x_1) \cap \mathcal A (x_2) \cap \mathcal A(x_0) } a_i^\top (0, \pm \frac{k}2) |
\end{align*}
and, thus, 
\begin{align}
    |\sum_{i \in \mathcal A(x_1) \cap \mathcal A (x_2) \cap \mathcal A(x_0) } (a_i)_2| \leq \frac{2 \epsilon}{k}.
    \label{estim2}
\end{align}
Assumption \eqref{approx_property} further yields
\begin{align}
    \epsilon \geq |f_1((0, 0)^\top)| =  |\sum_{i \in \mathcal A(x_1) \cap \mathcal A (x_2) \cap \mathcal A(x_0)} b_i |.
    \label{estim3}
\end{align}
Combination of \eqref{estim1}, \eqref{estim2} and \eqref{estim3} yields
\begin{align*}
    |f_1(x)| & = | \sum_{i \in \mathcal A(x_1) \cap \mathcal A (x_2) \cap \mathcal A(x_0)} a_i^\top x + b_i | \\ & \leq \|\sum_{i \in \mathcal A(x_1) \cap \mathcal A (x_2) \cap \mathcal A(x_0)} a_i \| \| x\| + \| \sum_{i \in \mathcal A(x_1) \cap \mathcal A (x_2) \cap \mathcal A(x_0)} b_i \|
    \\
    & \leq \frac{2\sqrt{2}\epsilon}{k} \sqrt{2}k + \epsilon \leq 5 \epsilon.
\end{align*}
A similar estimate can also be derived for $f_2$. Hence, subtracting all $i \in \mathcal A(x_1) \cap \mathcal A(x_2)$ just changes the approximation in the order of $\epsilon$. Therefore, without loss of generality (by possibly adapting $\epsilon$), we can assume that for $f_{\theta, \epsilon}$ it holds that $\mathcal A(x_1) \cap \mathcal A(x_2) = \emptyset$.

\subsubsection{Proof by Contradiction}

\begin{lemma}
Let $d \in \mathbb N$, $N \in \mathbb N$, $\lbrace a_i, b_i \rbrace_{i \in \lbrace 1, \ldots, N \rbrace}$ be such that $a_i \in \mathbb R^d$, 
$b_i \in \mathbb R$ for all $i \in \lbrace 1, \ldots, N \rbrace$. 
Let $f(x):= \sum_{i=1}^N \rho(a_i^\top x + b_i)$ with $\rho (s) = \mathrm{ReLU}(s)$. Let $x_1, x_2 \in \mathbb R^d$ be such that $\mathcal A(x_1) \cap \mathcal A (x_2) = \emptyset$, where $\mathcal A(x)$ is defined by \eqref{definition_A}. Then 
\begin{align*}
\tilde f(x) := \sum_{j=1}^2 \rho(\sum_{i \in \mathcal A(x_j)} a_i^\top x + b_i)
\end{align*}
fulfills $\tilde f(x_j) = f(x_j)$ for $j \in \lbrace 1, 2\rbrace$ and $\tilde f(x) \leq f(x)$ for all $x \in \mathbb R^d$. Moreover, $\sum_{i \in \mathcal A(x_j)} a_i$ is an element of the subdifferential $\partial f(x_j)$ of $f$ in $x_j$.
\end{lemma}
\begin{proof}
Follows with the similar arguments as in the proof of Lemma \ref{lemma_1}. 
\end{proof}

Assume that $f_\theta$ is a convex function with $\mathcal A(x_1) \cap \mathcal A(x_2) = \emptyset$ (can be assumed due to section \ref{subsec_disjoint}) and such that 
\begin{align}
 | f_\theta (x) - h(x) | \leq \epsilon 
 \label{statement_1}
\end{align}
for all $x \in K$. Hence, for $\epsilon > 0$ sufficiently small, there exists a $C > 0$ independent of $k$ and $\epsilon$ such that
\begin{align}
\begin{split}
& \| g_1 - (1, 1)^\top \| \leq \frac{C}k \epsilon, \\
& \| g_2 - (1, -1)^\top \| \leq \frac{C}k \epsilon,
\end{split}
\label{eq_subdiff}
\end{align}
for all $g_1 \in \partial f_\theta(x_1)$, $g_2 \in \partial f_\theta(x_2)$, 
and 
\begin{align}
|\sum_{i \in A(x_j)} b_i | \leq C \epsilon
\label{eq_bs}
\end{align}
for $j \in \lbrace 1, 2 \rbrace$.

Due to the properties of $\rho$, for $a, c \in \mathbb R^d$ with $\| a - c \| \leq \frac{C}k \epsilon$ and $b, d \in \mathbb R$ with $\| b - d \| \leq C \epsilon$, there exists a constant $\tilde C >0$ such that 
\begin{align*}
    \| \rho(a^\top x + b) - \rho(c^\top x + d) \| \leq \tilde C \epsilon
\end{align*}
for all $x \in K$.
Hence, with \eqref{eq_subdiff} and \eqref{eq_bs},
\begin{align}
 f_\theta(x) \geq ( \rho((1,1)^\top x) + \rho((1,-1)^\top x)) - \tilde C \epsilon. 
\label{statement_2}
\end{align}
Combining statements \eqref{statement_1} and \eqref{statement_2} yields a constant $\hat C > 0$ such that
\begin{align}
h(x) \geq ( \rho((1,1)^\top x) + \rho((1,-1)^\top x)) - \hat C \epsilon.
\label{eq::contra}
\end{align}
However, for $x_\alpha = (\alpha, 0)^\top$, it holds that
\begin{align}
 ( \rho((1,1)^\top x_\alpha) + \rho((1,-1)^\top x_\alpha)) - h(x_\alpha)  = \alpha 
\end{align}
for all $\alpha \in (0, k)$. Hence, for $\epsilon$ sufficiently small or $k$ sufficiently big, this yields a contradiction to \eqref{eq::contra}.

\end{document}